\documentclass[11pt]{amsart}  
\usepackage{amssymb, latexsym, amsmath, amsfonts}

\newtheorem{thm}{Theorem}[section]

\newtheorem{lem}[thm]{Lemma}
\newtheorem{prop}[thm]{Proposition}
\theoremstyle{definition}    

\theoremstyle{remark}
  
\numberwithin{equation}{section}

\newcommand{\mbb}{\mathbb}
\newcommand{\ra}{\rightarrow}
\newcommand{\z}{\zeta}
\newcommand{\pa}{\partial}
\newcommand{\ov}{\overline}

\newcommand{\ep}{\epsilon}

\newcommand{\al}{\alpha}

\newcommand{\ti}{\tilde}

\newcommand{\norm}[1]{\left\Vert#1\right\Vert}

\newcommand{\abs}[1]{\left\vert#1\right\vert}

\begin{document}
\title{Bounds for invariant distances on pseudoconvex Levi corank one domains and applications}
\thanks{The third named author was supported by the DST SwarnaJayanti Fellowship 2009--2010 and a UGC--CAS Grant}
\author{G. P. Balakumar, Prachi Mahajan and Kaushal Verma}

\address{G.P. Balakumar: Department of Mathematics,
Indian Institute of Science, Bangalore 560 012, India}
\email{gpbalakumar@math.iisc.ernet.in}

\address{Prachi Mahajan: Department of Mathematics, Indian Institute of Technology -- Bombay, Mumbai 400076, India}
\email{prachi.mjn@iitb.ac.in}

\address{Kaushal Verma: Department of Mathematics, Indian Institute of Science, Bangalore 560 012, India}
\email{kverma@math.iisc.ernet.in}

\pagestyle{plain}

\begin{abstract}
Let $D \subset \mbb C^n$ be a smoothly bounded pseudoconvex Levi corank one domain with defining function $r$, i.e., the Levi form $\pa \ov \pa r$ of the boundary $\pa D$ 
has at least $(n - 2)$ positive eigenvalues everywhere on $\pa D$. The main goal of this article is to obtain  bounds for the Carath\'{e}odory, Kobayashi and the Bergman distance 
between a given pair of points $p, q \in D$ in terms of parameters that reflect the Levi geometry of $\pa D$ and the distance of these points to the boundary. Applications include an 
understanding of Fridman's invariant for the Kobayashi metric on Levi corank one domains, a description of the balls in the Kobayashi metric on such domains that are centered at points 
close to the boundary in terms of Euclidean data and the boundary behaviour of Kobayashi isometries from such domains.
\end{abstract}

\maketitle

\section{Introduction}

\noindent The efficacy and ubiquity of invariant metrics such as those of Bergman, Carath\'{e}odory and Kobayashi, whenever they are defined, is a fact that hardly needs to be 
justified. Each of these metrics has its own strengths and weaknesses. For example, while the Carath\'{e}odory and Kobayashi metrics are distance decreasing under holomorphic 
mappings, a property not enjoyed by the Bergman metric in general, it makes up by being Hermitian whereas the Carath\'{e}odory and Kobayashi metrics are only upper 
semicontinuous in general. For a bounded domain $D \subset \mathbb C^n$, a point $z \in D$ and a tangent vector $v \in \mathbb C^n$, let $B_D(z, v), C_D(z, v)$ and $K_D(z, v)$ be the 
infinitesimal Bergman, Carath\'{e}odory and Kobayashi metrics on $D$. To recall their definition, let $\Delta_r \subset \mathbb C$ be the open disc of radius $r > 0$ centered at the 
origin and this will be abbreviated as $\Delta$ when $r = 1$. Then

\begin{itemize}
\item $C_D(z, v) = \sup \big\{ \abs{df(z)v} : f \in \mathcal O(D, \Delta) \big\}$,
\item $K_D(z, v) = \inf \big\{ 1/r : \; \text{there exists} \; f \in \mathcal O(\Delta_r, D) \; \text{with} \; f(0) = z \;\text{and} \;df(0) = v \big\}$, and
\item $B_D(z, v) = b_D(z, v) / (K_D(z, \overline z))^{1/2}$, where
\[
K_D(z, \ov z) = \sup \big\{ \abs{f(z)}^2 : f \in \mathcal O(D), \norm{f}_{L^2(D)} \le 1 \big\} 
\]
is the Bergman kernel and 
\[
b_D(z, v)  = \sup \big\{ \abs{df(z)v} : f \in \mathcal O(D), f(z) = 0 \; \text{and} \; \norm{f}_{L^2(D)} \le 1 \big\}.
\]
\end{itemize}

\noindent While no exact formulae for these metrics are known in general, part of their utility derives from an understanding of how they behave when $z$ is close to the boundary $\pa 
D$. This is a much studied question and we may refer to \cite{Ab}, \cite{AT}, \cite{D1}, \cite{Ca}, \cite{Cho}, \cite{Cho1}, \cite{G}, \cite{Her1}, \cite{Her2} and \cite{Ma} 
which provide quantitative boundary estimates for these infinitesimal metrics on a wide variety of smoothly bounded pseudoconvex domains in $\mathbb C^n$. Now given a pair of distinct 
points $p, q \in D$, we may compute the lengths of all possible piecewise smooth curves in $D$ that join $p$ and $q$ using these infinitesimal metrics. The infimum of all such lengths 
gives a distance function on $D$ -- these integrated versions are the Bergman, the inner Carath\'{e}odory and the Kobayashi distances on $D$ and they will be denoted by $d^b_D(p, q), 
d^c_D(p, q)$ and $d^k_D(p, q)$ respectively for $p, q \in D$. Recall that the Carath\'{e}odory distance  $ d_D^{Cara} $ for $D$ is defined by setting
\[
d_D^{Cara}(p,q) = \sup \big\{ d_\Delta^p\big( f(p), f(q) \big): f \in \mathcal{O}(D, \Delta) \big\},
\]
where $d_\Delta^p$ denotes the Poincar\'e distance on the unit disc. It is well-known that $d_D^c$ is always at least as big as $d_D^{Cara}$. In general, $ d_D^{Cara} $ need not 
coincide with the inner Carath\'{e}odory distance $d_D^c$. Moreover, it is known from the work of Reiffen (\cite{Rei}) that the integrated distance of the infinitesimal metric $C_D(z,v)$, 
coincides with the inner distance associated to $d_D^{Cara}$ which we denote by $d_D^c$. Recent work on the comparison of $d_D^{Cara}$ and another invariant function, the Lempert function, may be found in \cite{N1}.  Much less is known about the boundary 
behaviour of the integrated distances and our interest here lies in obtaining estimates of them. Partial answers to this question may be found in \cite{Ab}, \cite{Al} and \cite{FR}, all of which deal with strongly pseudoconvex domains in 
$\mbb C^n$. Optimal estimates of the boundary behaviour of invariant distances for domains with $C^{1 + \epsilon}$-smooth boundary in dimension one, may be found in the recent work \cite{N2} where estimates for convexifiable domains are also dealt with. A more complete treatment for strongly pseudoconvex domains in $\mbb C^n$ was given by Balogh--Bonk in \cite{BB} using the Carnot--Carath\'{e}odory metric that exists 
on the boundary of these domains. An analogue of these estimates was later obtained by Herbort in \cite{Her} on smoothly bounded weakly pseudoconvex domains of finite type in 
$\mbb C^2$ using the bidiscs of Catlin (see \cite{Ca}).

\medskip

We are interested in supplementing the results of \cite{BB} and \cite{Her} by obtaining bounds for these distances on a Levi corank one domain in $\mbb C^n$. A smoothly 
bounded pseudoconvex domain $D \subset \mbb C^n$ of finite type with smooth defining function $r$ (here the sign of $r$ is chosen so that $D = \{r < 0\}$) is said to be a Levi corank 
one domain if the Levi form $\pa \ov \pa r$ of $\pa D$ has at least $n - 2$ positive eigenvalues everywhere on $\pa D$. An example of a Levi corank one domain is the egg domain
\begin{equation} \label{Egg}
E_{2m} = \big\{ z \in \mbb C^n : \abs{z_1}^{2m} + \abs{z_2}^2 + \ldots + \abs{z_{n-1}}^2 + \abs{z_n}^{2} <1  \big\}
\end{equation}
for some integer $m \ge 1$. Let $2m$ be the least upper bound on the $1$--type of the boundary points of $D$ and let $U$ be a tubular neighbourhood  of $\pa D$ such that for any 
$A \in U$ there is a unique orthogonal projection to $\pa D$ which will be denoted by $A^{\ast}$ 
such that $\abs{A - A^{\ast}} = {\rm dist}(A, \pa D) = \delta_D(A)$. Furthermore, we may assume that the normal vector field, given at any $\zeta \in U$ by
\[
\nu(\zeta) = \big(\pa r /\pa \ov z_1(\zeta), \pa r / \pa \ov z_2(\zeta), \ldots, \pa r / \pa \ov z_n(\zeta) \big)
\]
has no zeros in $U$ and is normal to the hypersurface $\Gamma = \{ r(z) = r(\zeta) \}$. Fix $\zeta \in U$. After a permutation of co-ordinates if necessary, we may assume 
that $\partial r /\partial \overline{z}_n(\zeta) \neq 0$. Then note that the affine transform 
\[
\phi^\zeta(z) = \big( z_1 - \zeta_1, \ldots, z_{n-1} - \zeta_{n-1}, \langle \nu(\zeta), z - \zeta \rangle \big)
\] 
translates $\zeta$ to the origin and is invertible (by virtue of the fact that $\partial r /\partial \overline{z}_n(\zeta) \neq 0$). Moreover, $ \phi^{\zeta} $
 reduces the linear part of the Taylor expansion of $r^\zeta = r \circ (\phi^\zeta)^{-1}$ about the origin, to 
\begin{equation}\label{nrmlfrm0}
r^\zeta(z) = r(\zeta) +  2\Re z_n + \text{ terms of higher order. }
\end{equation} 
\noindent  In particular, the origin lies on the hypersurface $\Gamma_\zeta^{r^\zeta}$, the zero set of $r^\zeta(z) - r(\zeta)$ and the normal to this hypersurface at the origin, is the unit vector along the $\Re z_n$-axis. In fact, by the continuity of $\partial r^\zeta/\partial z_n(\zeta)$, we 
get a small ball $ B(0,R_0)$ (where as usual $ B(p, r) $ is the ball around $ p $ with radius $ r $), with the property that the vector field $\nu(z)$ has a 
non-zero component along (the constant vector field) $L_n = \partial / \partial z_n$ for all $ z $ in $ B(0, R_0) $. Indeed, we may assume that $\vert \partial r^\zeta/\partial z_n (z) \vert $ 
is bounded below by any positive constant less than $1$. By shrinking $ B(0, R_0) $, if necessary, we can ensure that $\vert \partial r^\zeta/\partial z_n (z) \vert \geq 1/2 $. We will
perform such shrinking of neighbourhoods henceforth tacitly, taking care only that the number of times we have done this at the end of all, is finite. Furthermore, we may repeat the 
above procedure for any $\zeta \in U$. Since $r$ and $(\phi^\zeta)^{-1}$ are smooth (as functions of $\zeta$), the family $\{  \partial r^\zeta/\partial z_n(z) \}$ of functions 
(parametrized by $\zeta$) is equicontinuous. Moreover, the neighbourhood $U$ is precompact, and hence, we may choose the radius $R_0$ to be independent of $\zeta$.\\  

\noindent Continuing with some helpful background and terminology, we assign a weight of $1/2m$ to the variable $z_1$, $1/2$ to the variable $z_{\al}$ for $2 \le \al \le n - 1$ and $1$ to the variable $z_n$ and for multi-indices 
$J = (j_1, j_2, \ldots, j_n)$ and $K = (k_1, k_2, \ldots k_n)$, the weight of the monomial $z^J \ov z^K = z_1^{j_1} z_2^{j_2} \ldots z_n^{j_n} \ov z_1^{k_1} \ov z_2^{k_2} \ldots \ov 
z_n^{k_n}$, is  defined to be
\[
{\rm wt}(z^j \overline{z}^K) = (j_1 + k_1)/2m + (j_2 + k_2)/2 + \ldots + (j_{n-1} + k_{n-1})/2 + (j_n + k_n).
\]
Analogous to the definition of degree, we say that a polynomial in $\mathbb{C}[z,\overline{z}]$ is of weight $\lambda$ if the maximum of the 
weights of its monomials is $\lambda$; the weight of a polynomial mapping is the maximum of the weights of its components. Thus note that 
the weight of the map $\phi^\zeta$ is $1$ while its `multiweight' is $(1/2m, 1/2, \ldots,1/2,1)$. \\

\noindent We now summarize the special normal form for Levi corank one domains (cf. \cite{Cho2}), details of which are discussed in the appendix. For each  
$\z \in U$, there is a radius $R > 0$ and an injective holomorphic mapping $\Phi^{\z} : B(\z, R) \ra \mbb C^n$ such that the transformed defining 
function $\rho^{\z} = r^{\z} \circ (\Phi^{\z})^{-1}$ becomes
\begin{multline}\label{nrmlfrm}
\rho^{\z}(w) = r(\z) + 2 \Re w_n + \sum_{l = 2}^{2m}P_l(\z; w_1) + \abs{w_2}^2 + \ldots + \abs{w_{n-1}}^2 \\
+ \sum_{\al = 2}^{n - 1} \sum_{ \substack{j + k \le m\\
                                                          j, k > 0}} \Re \Big( \big(b_{jk}^{\al}(\z) w_1^j \ov w_1^k \big) w_{\al} \Big)+ R(\z;w)
\end{multline}
where 
\[
P_l(\z; w_1) = \sum_{j + k = l} a^l_{jk}(\z) w_1^j \ov w_1^k
\]
are real valued homogeneous polynomials of degree $l$ without harmonic terms and the error function $R(\zeta,w) \to 0$ as $w \to 0$ faster than at least one of the monomials of weight $1$. Further, the map $\Phi^{\z}$ is 
actually a holomorphic polynomial automorphism of weight one  of the form
\begin{alignat}{3} \label{E45}
\Phi^{\z}(z) = \Big(z_1 - \z_1, G_{\z}(\ti z - \ti \z) - Q_2(z_1 - \z_1), \langle \nu(\zeta), z- \zeta \rangle - Q_1({}'z - {}' \zeta) \Big)
\end{alignat}
where $G_{\z} \in GL_{n-2}(\mbb C), \ti z = (z_2, \ldots z_{n-1}), {}'z = (z_1, z_2, \ldots, z_{n-1})$ and $Q_2$ is a vector valued polynomial whose $\al$-th component is a polynomial 
of weight at most $1/2$ of the form
\[
Q^{\al}_2(t) = \sum_{k = 1}^m b_k^{\al}(\z) t^k
\]
for $t \in \mbb C$ and $2 \le \al \le n - 1$. Finally, $Q_1({}' z - {}'\zeta)$ is a polynomial of weight at most $1$ and is of the form $\hat{Q}_1 \big(z_1, -\zeta_1, G_\zeta (\tilde{z} - \tilde{\zeta}) \big)$ with $\hat{Q}_1$ of the form
\[
\hat{Q}_1(t_1, t_2, \ldots, t_{n-1}) = \sum_{k=2}^{2m} a_{k0}(\z)t_1^k - \sum_{\al = 2}^{n-1} \sum_{k=1}^m a_k^{\al}(\z) t_{\al}t_1^k - \sum_{\al = 2}^{n-1}c_{\al}(\z) t^2_{\al}.
\]
Since $G_\zeta$ is just a linear map, $Q_1({}'z - {}' \zeta)$ also has the same form when considered as an element of the ring of holomorphic polynomials 
$\mathbb{C}[{}'z - {} '\zeta]$, when $\zeta$ is held fixed. The coefficients of all the polynomials, mentioned above, are smooth functions of $\zeta$. By shrinking $U$,
 if needed, we can ensure that $R > 0$ is independent of $\z$ because these new coordinates depend smoothly on $\z$. Further $Q_1(0, \ldots, 0) = 0$ and that the lowest 
degree of its monomials is at least two. On the other hand, while $Q_2(0) = (0, \ldots, 0)$, the lowest degree of the terms 
in $Q^{\al}_2$ is at least (and can be) one. In case, the polynomials $Q_2^\al
$ and $Q_1$ are identically zero, it turns out that the arguments become even simpler and this will be evident from the sequel. Note that $\Phi^{\z}(\z) = 0$ and 
\[
\Phi^{\z}(\z_1, \ldots, \z_{n-1}, \z_n - \ep) = (0, \ldots, 0, -\ep \;\pa r/ \pa \ov z_n(\z)).
\]
Finally, the transforms that reduce degree two terms to the form described in (\ref{nrmlfrm}), force that $L_j$'s (associated to $\rho^\zeta$) {\it at the origin}
form an orthonormal basis of eigenvectors for the Levi form. In fact, $ L_j$'s are the unit vectors along the co-ordinate axes.\\ 

\noindent Two remarks are in order here. Firstly, observe that the defining function in (\ref{nrmlfrm0}) becomes simple, however, the change of variables that effects this is not simple
(in the sense of being weighted homogeneous in the variables $z - \zeta$). This is evident in (\ref{nrmlfrm0}) where only the linear part of the expansion is reduced to its simplest
form (up to a permutation) and later in the final transformation (\ref{E45}). Note that each $ \Phi^{\zeta} $ is a polynomial automorphism of $ \mathbb{C}^n $ and will be referred to
as the canonical change of variables for Levi corank one domain. Evidently, $ \Phi^{\zeta} $ is neither a `decoupled polynomial mapping' nor a weighted homogeneous map in $z - \zeta$. 
This poses difficulties in imitating Herbort's calculations directly and we shall discuss these difficulties as we encounter them, namely in Section \ref{sixth}.\\

\noindent Secondly, recall that, in general, the degree (or weight) of a holomorphic polynomial automorphism  is not equal to that of its 
inverse. However, the inverse of $\Phi^{\z}$ has the same form as $\Phi^{\zeta}$. It can be checked (done in the last section) that these maps belong to the algebraic 
group $\mathcal E_{L}$ of all weight preserving polynomial automorphisms whose first component has weight $1/2m$, the next $n-2$ components are of weight $1/2$ each and the last component 
has weight $1$; in particular, the weight of $\Phi^{\z}$ is one. Thus the collection $\{Q(\zeta, \delta_e), \Phi^{\z}\}$ -- where $\delta_e$ is defined below -- forms an atlas for a tubular neighbourhood of $\pa D$ giving it the structure of an $\mathcal E_{L}$-manifold, i.e., the 
transition maps associated to this atlas lie in $\mathcal E_{L}$.\\

\noindent To construct the distinguished polydiscs around $\z$ (more precisely, biholomorphic images of polydiscs), with notations as in \cite{Her}
define for each $\delta > 0$, the special-radius
\begin{alignat}{3} \label{E46}
\tau(\z, \delta) = \min \Big\{ \Big( \delta/ \vert P_l(\zeta, \cdot) \vert \Big)^{1/l}, \; \Big(\delta^{1/2}/B_{l'}(\zeta) \Big)^{1/l'}  \; : \; 2 \le l \le 2m, \; 2 \leq l' \leq m   \Big \}.
\end{alignat}
where
\[
B_{l'}(\zeta) = {\rm max} \{ \vert b_{jk}^\alpha (\zeta) \vert \; : \; j+k=l', \; 2 \leq \alpha \leq n-1 \}, \; 2 \leq l' \leq m.
\]
It was shown in \cite{Cho2} that the coefficients $b_{jk}^\alpha$'s in the above definition are insignificant and may be dropped out, so that 
\[
\tau(\z, \delta) = \min \Big\{ \Big( \delta/ \vert P_l(\zeta, \cdot) \vert \Big)^{1/l}  \;: \; 2 \le l \le 2m \Big\}.
\]
Set
\[
\tau_1(\z, \delta) = \tau(\z, \delta) = \tau, \tau_2(\z, \delta) = \ldots = \tau_{n-1}(\z, \delta) = \delta^{1/2}, \tau_n(\z, \delta) = \delta
\]
and define
\[
R(\z, \delta) = \{ z \in \mbb C^n : \vert z_k \vert < \tau_k(\z, \delta), \; 1 \le k\le n \}
\]
which is a polydisc around the origin in $\mbb C^n$ with polyradii $\tau_k(\z, \delta)$ along the $z_k$ direction for $1 \le k \le n$ and let
\[
Q(\z, \delta) = (\Phi^{\z})^{-1}\big(R(\z, \delta) \big)
\]
which is a distorted polydisc around $\z$. It was shown in \cite{Cho2} that for all sufficiently small positive $\delta$ -- say, for all $\delta$ in some interval $(0,\delta_e)$ -- there is a uniform constant $C_0 > 1$ such that 
\begin{enumerate}
\item[(i)] these `polydiscs' satisfy the engulfing property, i.e., for all $\z \in U$ if $\eta \in Q(\z, \delta)$, then $Q(\eta, \delta) \subset Q(\z, C_0 \delta)$ and 
\item[(ii)] if $\eta \in Q(\z, \delta)$ then $\tau(\eta, \delta) \le C_0 \tau(\z, \delta) \le C_0^2 \tau(\eta, \delta)$. 
\end{enumerate}

\medskip

\noindent For $A, B \in U$ let 
\[
M(A, B) = \{ \delta > 0 : A \in Q(B, \delta) \}
\]
which is a sub-interval of $[0, \infty)$ since the distorted polydiscs $Q(B, \delta)$ are monotone increasing as a function of $\delta$. Then define
\[
d'(A, B) = \inf M(A, B)
\]
if $M(A, B) \not= \emptyset$ which happens precisely when $\vert A - B \vert \le R$ (otherwise we simply set it equal to $+\infty$) and let
\begin{equation} \label{pseudistdefn}
d(A, B) = \min \{d'(A, B), \vert A - B \vert_{l^{\infty}} \}
\end{equation}
which is an auxiliary pseudo-distance function on $D$. Here we work with the $l^{\infty}$ norm of $A - B$ instead of the usual Euclidean distance for convenience. Now, let
\[
\eta(A,B) =  \log  \Bigg(1 + \frac{\abs{\hat{\Phi}^A(B)_n}}{\delta_D(A)} + \sum_{\al = 2}^{n-2} \frac{\abs{\hat{\Phi}^A(B)_{\al}}}{\sqrt{\delta_D(A)}} + 
\frac{\abs{\hat{\Phi}^A(B)_1}}{\tau \big(A, \delta_D(A)\big)}  \Bigg) 
\]
where $\hat{\Phi}^A(B)$ differs from $\Phi^A(B)$ by a permutation of co-ordinates to ensure $\nu(A)_n \neq 0$ i.e., $\hat{\Phi}^A(B) = \Phi^A \circ P_A(B)$ where $P_A(z)$ is any 
permutation of the variables $z_1, \ldots , z_n$ such that $\partial (r \circ P_A^{-1})/\partial z_n \neq 0$. The existence of $P_A$ follows from the assumption that $\nu(A) \neq 0$. 
If $A,B$ are close enough, then we may certainly choose $P_A=P_B$. Finally, let
\[
\varrho(A, B) = 1/2 \Big( \eta(A,B) + \eta(B,A) \Big)
\]
and as usual, for quantities $S, T$ we will write $S \lesssim T$ to mean that there is a constant $C > 0$ such that $S \le C T$, while $S \approx T$ means that $S \lesssim T$ and 
$T \lesssim S$ both hold.

\medskip

\noindent It is known (\cite{Cho1}) that these infinitesimal metrics are uniformly comparable on a pseudoconvex Levi 
corank one domain $D$, i.e., $B_D(z, v) \approx C_D(z, v) \approx K_D(z, v)$ uniformly for all $(z, v) \in D \times \mbb C^n$. Thus 
to get lower bounds for the integrated distances it suffices to understand $d^c_D$ alone. We use the inequality $ d_D^{Cara} \leq d^c_D  $ 
to get sharper lower bounds on $ d^c_D $ and this is the only motivation for introducing $ d_D^{Cara} $. 

\begin{thm} \label{thm1}
Let $D \subset \mbb C^n$ be a smoothly bounded Levi corank one domain and $U$ a tubular neighbourhood of $\pa D$ as above. Then for all $A, B \in U \cap 
D$ we have that
\[
  \varrho(A, B) - l \lesssim d^c_D(A, B) \leq d_D^k(A,B) \lesssim \varrho(A, B) + L
\]
where $l$ and $L$ are some positive constants.
In particular, the same bounds hold for $d^b_D(A, B)$ as well.
\end{thm}

\noindent We may also express the bounds here by replacing the term $\Phi^A(\cdot)_n$ involved in the definition of $\eta(A,B)$ which involves monomials in $B_j$ for $1 \leq j \leq n$, i.e., involves all the components of $B$ in a slightly complicated way, by the more comprehensive term $d(\cdot,A)$ with $d$ being the special pseudo-distance defined as above, available for Levi co-rank one domains; for the precise statements see (\ref{dclowbd}) and (\ref{kobdistuppbd}).\\ 

\noindent Obtaining a lower bound for $d^c_D$ amounts to understanding the separation properties of bounded holomorphic functions on $D$. To do this, the boundary $\pa D$ will be bumped 
outwards near a given $\z \in \pa D$ as in \cite{Cho1} to get a larger domain that contains $D$ near $\z$. The pluricomplex Green function for these large domains will then be used to 
construct weights for an appropriate $\ov \pa$-problem as in \cite{Her} and the solution thus obtained will be modified to get a holomorphic function with a suitable $L^2$ bound near 
$\z$. Solving another $\ov \pa$-problem will extend this to a bounded holomorphic function on $D$ with control on its separation properties. This function is then used to bound 
$d^c_D$ from below. Theorem \ref{thm1} has several consequences and we elaborate them in the following paragraphs.

\medskip

In \cite{Fr}, Fridman defined an interesting non-negative continuous function on a given Kobayashi hyperbolic complex manifold of dimension $n$, say $X$ that essentially determines the 
largest Kobayashi ball at a given point on $X$ which is comparable to the unit ball $\mbb B^n$. To be more precise, let $B_X(p, r)$ denote the Kobayashi ball around $p \in X$ of radius 
$r > 0$. The hyperbolicity of $X$ ensures that the intrinsic topology on $X$ is equivalent to the one induced by the Kobayashi metric. Thus for small $r > 0$, the ball $B_X(p, r)$ is 
contained in a coordinate chart around $p$ and hence there is a biholomorphic imbedding $f : \mbb B^n \ra X$ with $B_X(p, r) \subset f(\mbb B^n)$. Let $\mathcal R$ 
be the family of all $r > 0$ for which there is a biholomorphic imbedding $f : \mbb B^n \ra X$ with $B_X(p, r) \subset f(\mbb B^n)$. Then $\mathcal R$ is evidently non-empty. Define
\[
h_X(p,\mbb B^n) = \inf_{r \in \mathcal R} \frac{1}{r}
\]
which is a non-negative real valued function on $X$. Since the Kobayashi metric is biholomorphically invariant, the same holds for $h_X(p, \mbb B^n)$ which shall henceforth be called 
Fridman's invariant. The same construction can be done using any invariant metric that induces the intrinsic topology on $X$ and we may also work with homogeneous domains other than the 
unit ball. However we shall work with the Kobayashi metric exclusively. A useful property identified in \cite{Fr} was that if $h_X(p_0, \mbb B^n)$ for some $p_0 \in X$, then $h_X(p, 
\mbb B^n) = 0$ for all $p \in X$ and that $X$ is biholomorphic to $\mbb B^n$. Moreover, $p \mapsto h_X(p, \mbb B^n)$ is continuous on $X$. The boundary behaviour of 
Fridman's invariant was studied in \cite{MV} for a variety of pseudoconvex domains and the following statement extends this to the class of Levi corank one domains. 

\begin{thm} \label{thm2}
Let $D \subset \mbb C^n$ be a smoothly bounded pseudoconvex domain of finite type. Let $\{p^j\} \subset D$ be a sequence that converges to $p^0 \in \pa D$. Assume that the Levi form of 
$\pa D$ has rank at least $n - 2$ at $p^0$. Then
\[
h_D(p^j, \mbb B^n) \ra h_{D_{\infty}}(('0, -1), \mbb B^n)
\]
as $j \ra \infty$ where $D_{\infty}$ is a model domain defined by
\[
D_{\infty} = \big\{ z \in \mbb C^n : 2 \Re z_n + P_{2m}(z_1, \ov z_1) + \abs{z_2}^2 + \ldots + \abs{z_{n-1}}^2 < 0  \big\}
\]
and $P_{2m}(z_1, \ov z_1)$ is a subharmonic polynomial of degree at most $2m$ ($m \ge 1$) without harmonic terms, $2m$ being the $1$-type of $\pa D$ at $p^0$.
\end{thm}

\noindent By scaling $D$ along $\{p^j\}$, we obtain a sequence of domains $D^j$, each containing the base point $('0, -1)$, that converge to $D_{\infty}$ as defined above. It should be 
noted that the polynomial $P_{2m}(z_1, \ov z_1)$ depends on how the sequence $\{p^j\}$ approaches $p^0$. This is simply restating the known fact that unlike the strongly pseudoconvex 
case, model domains near a weakly pseudoconvex point are not unique. Since Fridman's invariant is defined in terms of Kobayashi balls, the main technical step in this theorem is to 
show the convergence of the Kobayashi balls in $D^j$ around $('0, -1)$ with a fixed radius $R > 0$ to the corresponding Kobayashi ball on $D_{\infty}$ with the same radius. Theorem 1.1 
is used in this step.

\medskip

Another consequence of scaling near a Levi corank one point combined with Theorem 1.1 is a description of the Kobayashi balls near such points in terms of parameters that reflect the 
Levi geometry of the boundary. This is well known in the strongly pseudoconvex case -- indeed, the Kobayashi ball around a given point $p$ near a strongly pseudoconvex boundary point is 
essentially an ellipsoid whose major and minor axis are of the order of $(\delta_D(p))^{1/2}$ and $\delta_D(p)$ respectively. 

\begin{thm} \label{thm3}
Let $D \subset \mbb C^n$ be a smoothly bounded pseudoconvex Levi corank one domain. Then for all $R>0$, there are constants $C_1, C_2 > 0$ depending only on $R$ and $D$ such that 
\[
Q(q, C_1 \; \delta_D(q)) \subset B_D(q, R) \subset Q(q, C_2 \; \delta_D(q))
\]
for each $q \in D$ sufficiently close to $\pa D$.
\end{thm}

\noindent Analogues of this for weakly pseudoconvex finite type domains in $\mbb C^2$ and convex finite type domains in $\mbb C^n$ were obtained in \cite{MV} by a direct scaling. 
These estimates are useful in establishing a generalized sub-mean value property for plurisubharmonic functions and defining suitable approach regions for boundary values of functions 
in $H^p$ spaces at least on strongly pseudoconvex domains (see \cite{Kr} for example). 

\medskip

The next and last class of applications of Theorem 1.1 deal with the problem of biholomorphic inequivalence of domains in $\mbb C^n$. The paradigm underlying many of the 
results in this direction (see for example \cite{Be} and \cite{Pi}) is that a pair of domains in $\mbb C^n$ cannot have boundaries with different Levi geometry and yet be biholomorphic. 
For proper holomorphic mappings, it is known (see for example \cite{DF} and \cite{CPS}) that the target domain cannot have a boundary with more complicated Levi degeneracies than the 
source domain. Fridman's invariant provides another approach to this problem with the advantage of quickly reducing it to the case of algebraic model domains. Here is an example to 
illustrate this point of view and we refer the reader to \cite{BV} for an alternative proof that works for proper holomorphic mappings as well.

\begin{thm} \label{thm4}
Let $D_1, D_2 \subset \mbb C^n$ be bounded domains with $p^0 \in \pa D_1$ and $q^0 \in \pa D_2$. Assume that $\pa D_1$ is $C^2$-smooth strongly pseudoconvex near $p^0$ and that $\pa 
D_2$ is $C^{\infty}$-smooth pseudoconvex and of finite type near $q^0$. Suppose further that the Levi form of $\pa D_2$ has rank exactly $n-2$ at $q^0$. Then there cannot exist a 
biholomorphism $f : D_1 \ra D_2$ with $q^0 \in cl_f(p^0)$, the cluster set of $p^0$.
\end{thm}

\noindent When $p^0$ is also a Levi corank one point on $\pa D_1$, it is known (\cite{Su}) that a proper holomorphic mapping $f : D_1 \ra D_2$ extends continuously to $\pa D_1$ near 
$p^0$. In fact, a similar result can be proved for isometries of these metrics as well. To set things in perspective, let $D, G$ be bounded domains in $\mbb C^n$ equipped with one of 
these invariant metrics. An isometry $f : D \ra G$ is simply a distance preserving map. Note that no further assumptions such as smoothness or holomorphicity are being included as part 
of the definition of an isometry. Of course, biholomorphisms are examples of isometries, but whether all isometries are necessarily holomorphic or conjugate holomorphic seems 
interesting to ask. Let us say that an isometry is {\it rigid} if it is either holomorphic or conjugate holomorphic. If isometries of the Bergman metric between a pair of strongly 
pseudoconvex domains in $\mbb C^n$ are considered, a result in \cite{GK} shows that 
the isometry must be rigid. Recent work on the rigidity of local Bergman isometries may be found in \cite{Mok}. Isometries of the Kobayashi metric between a strongly pseudoconvex domain and the ball are also shown to be rigid in \cite{KK} while a more 
recent result in \cite{GS} proves the rigidity of an isometry between a pair of strongly convex domains even in the non equidimensional case; the choice of either the Kobayashi or the 
Carath\'{e}odory metric is irrelevant here since the two coincide. However, this seems to be unknown for isometries of the Kobayashi or the Carath\'{e}odory metric between a pair of 
strongly pseudoconvex domains. On the other hand, the results of \cite{M} and \cite{BB} show that isometries behave very much like holomorphic mappings. In particular, they exhibit 
essentially the same boundary behaviour as biholomorphisms. The following statements further justify this claim and extend some of the results in \cite{M}.

\begin{thm} \label{thm5}
Let $ f : D_1 \rightarrow D_2 $ be a Kobayashi isometry
between two bounded domains in $ \mathbb{C}^n $. Let $ p^0 $ and $
q^0 $ be points on $ \partial D_1 $ and $ \partial D_2 $
respectively. Assume that $ \partial D_1 $ is $
{C}^{\infty}$-smooth pseudoconvex of finite type near $ p^0
$ and that $ \partial D_2 $ is $ {C}^2$-smooth strongly
pseudoconvex near $ q^0 $. Suppose further that the Levi form of $ \partial D_1 $ 
has rank at least $ n - 2 $ near $ p^0 $ and that $ q^0 $ belongs to the cluster 
set of $p^0 $ under $f$. Then $f$ extends as a continuous mapping to a 
neighbourhood of $ p^0 $ in $ \overline{D}_1 $.
\end{thm}

\noindent The following result provides an explicit computation of $K_{E_{2m}}$, the Kobayashi metric of the egg domain $E_{2m}$ introduced in (1.1) for $m \geq 1/2$ -- notice that 
when $m$ is not an integer, the boundary of $E_{2m}$ is not smooth. This extends 
the computation done for such egg domains in $\mbb C^2$ in \cite{Ma1} and \cite{BFKKMP}. It is straightforward to see that for any $ \theta \in \mathbb{R} $ and $ (p_1, \ldots, p_n) \in 
E_{2m} $,
\[
(z_1, \ldots, z_n) \mapsto \left( e^{\iota \theta} \frac{\left( 1 - \abs{\hat{p}}^2 \right)^{1/2m} } { \left( 1 - \langle \hat{z}, 
\hat{p} \rangle \right)^{1/m} }  z_1, \Psi( \hat{z} ) \right)
\]
is an automorphism of $ E_{2m} $. Here $ \langle \cdot, \cdot \rangle $ denotes the standard Hermitian inner product in $ \mathbb{C}^{n-1} $, $ z \in \mathbb{C}^n $ 
is written as $ z= (z_1, \hat{z}), \hat{z} = (z_2, \ldots, z_n) $ and $ \Psi $ is an automorphism of $ \mathbb{B}^{n-1} $ that takes $ \hat{p} $ to the origin.
More precisely, 
\[
\Psi(\hat{z}) = \frac{ \big( 1- |\hat{p}|^2 \big)^{1/2} \left( \hat{z} - \frac{\langle \hat{z}, \hat{p} \rangle} {|\hat{p}|^{2}} \hat{p} \right) - 
\left( 1 -\frac{\langle \hat{z}, \hat{p} \rangle} {|\hat{p}|^{2}} \right)  \hat{p} }{1 - \langle \hat{z}, \hat{p} \rangle }
\]
for $ \hat{p} \neq \hat{0} $. Since automorphisms are isometries for the Kobayashi metric, it is enough to compute the explicit formula for $ K_{E_{2m}} $ at the point $ (p, \hat{0})
\in E_{2m} $, for $ 0 < p < 1 $, from which the general formula follows by composition with an appropriate automorphism of $ E_{2m} $ as described above.

\begin{thm} \label{thm6}
The Kobayashi metric for $E_{2m}$ is given by
\begin{alignat*}{4}
K_{E_{2m}} \big( (p, 0, \ldots, 0), (v_1, \ldots, v_n) \big) = 
   \left\{ \begin{array}{lrl}
\left( \frac{m^2 p^{2m-2} |v_1|^2}{(1-p^{2m})^2}
+ \frac{|v_2|^2}{1- p^{2m}} + \cdots + \frac{|v_n|^2}{1- p^{2m}} \right)^{1/2} & \mbox{for} & u \leq p, \\
\\
\frac{m \alpha (1-t) |v_1|}{p(1- {\alpha}^2) \left(m(1-t) + t \right) }
& \mbox{for} & u > p,
\end{array}
\right.
\end{alignat*}
where
\begin{eqnarray}
u & = & \left( \frac{m^2 |v_1|^2}{(|v_2|^2 + \cdots + |v_n|^2)} \right)^{1/2}, \label{E4} \\
t & = & \frac{2m^2 p^2}{u^2 + 2m(m-1) p^2 + u \big( u^2 + 4m(m-1) p^2 \big)^{1/2}} \label{E5}
\end{eqnarray}
and $ \alpha $ is the unique positive solution of 
\begin{equation*}
 \alpha^{2m} - t \alpha^{2m-2} - (1-t) p^{2m} = 0.
\end{equation*}
Moreover, $ K_{E_{2m}} $ is $ C^1$-smooth on $ E_{2m} \times \left( \mathbb{C}^n \setminus \{0\} \right) $ for $ m > 1/2 $.
\end{thm} 

\noindent This is useful in proving the isometric inequivalence of strictly weakly spherical Levi
corank one domains (the notion of weak sphericity is recalled from \cite{BaBe} and defined in the last section) and strongly pseudoconvex domains -- see \cite{M} for a related result in $\mbb C^2$.

\begin{thm} \label{thm7}
Let $D_1, D_2 \subset \mbb C^n$ be bounded domains with $p^0 \in \pa D_1$ and $q^0 \in \pa D_2$. Assume that there are holomorphic coordinates in a neighbourhood $U_1$ around $p^0$ in 
which $U_1 \cap D_1$ is defined by
\[
\big\{ z\in \mbb C^n : 2 \Re z_n + \abs{z_1}^{2m} + \abs{z_2}^2 + \ldots + \abs{z_{n-1}}^2 + R(z,\overline{z}) < 0  \big\}
\]
where $m > 1$ is a positive integer and the error function $R(z,\overline{z}) \to 0$ faster than atleast one of the monomials of weights one. Suppose further that $\pa D_2$ is $C^2$-smooth strongly pseudoconvex near $q^0$. Then there cannot exist a 
Kobayashi isometry $f : D_1 \ra D_2$ with $q^0 \in cl_f(p^0)$.
\end{thm}

\noindent To outline the proof of this statement, we scale $D_1$ along a sequence of points that converges to $p^0$ along the inner normal to $D_1$. The limit domain is exactly 
$E_{2m}$. In trying to adapt the scaling method for isometries, note that the normality of the scaled isometries needs the stability of the Kobayashi distance function on Levi corank 
one domains, i.e., the arguments used in analysing Fridman's invariant. This ensures the existence of the limit of scaled isometries, the limit being an isometry between $E_{2m}$ and 
the unit ball. The final argument, which uses the explicit form of the Kobayashi metric on $E_{2m}$ as given above, involves showing that this continuous isometry is in fact rigid which 
then leads to a contradiction.

\medskip

\noindent {\it Acknowledgements:} Theorem 1.1 has also been independently proved by Gregor Herbort recently. We were informed of this after the completion of this manuscript and we 
would like to thank him for sharing a copy of his unpublished manuscript and for a remark on an earlier version of this article that has been incorporated. Thanks are also due to Peter 
Pflug for a timely comment that helped remove an ambiguity and to Nikolai Nikolov who pointed out the relevance of \cite{N1} and \cite{N2}.

\section{Proof of Theorem \ref{thm1}} \label{pfofthm1}
\noindent We begin with some useful notation and terminology. Let $ \rho $ be a smooth function 
defined on some open set $V $ in $ \mathbb{C}^n$. Consider the canonical Hermitian form associated to $ \rho $:
\[
\langle Y, W \rangle \to \Big( \sum\limits_{j=1}^{n} \overline{\frac{\partial \rho}{\partial z_j}(z) Y_j} \Big) \Big( \sum\limits_{k=1}^{n} \frac{\partial \rho}{\partial z_k}(z)W_k \Big) \\
= \sum\limits_{j,k=1}^{n} \overline{\frac{\partial \rho}{\partial z_j}(z)} \frac{\partial \rho}{\partial z_k}(z) \overline{Y}_jW_k 
\]
for $ Y, W $ tangent vectors at $ z $. The associated quadratic form is given by 
\begin{equation} \label{Canfrm}
C_\rho(z,Y) = \Big \vert \sum\limits_{j=1}^{n} \frac{\partial \rho}{\partial z_j}(z)Y_j \Big \vert^2
\end{equation}
which is evidently positive semi-definite. Denote by $M_\rho(z)$, the matrix associated to this form with respect to (unless otherwise mentioned) the standard co-ordinates. 
Now, let $\zeta \in U$ and for $1 \leq j \leq n-1$ define the vectors
\[
L_j(\zeta) = \Big(0, \ldots,0,1,0, \ldots,0, b_j^\zeta \Big)
\]
where the $j$th entry in the above tuple is $1$ and 
\[
b_j^\zeta=b_j(\zeta,\overline{\zeta}) =  - \Big( \frac{\partial r }{\partial z_n}(\zeta)\Big)^{-1} \Big( \frac{\partial r }{\partial z_j}(\zeta) \Big)
\]
This collection of $n-1$ vectors forms a basis for the complex tangent space at $\zeta$ to the hypersurface
\[
\Gamma_\zeta^r = \Big\{ z \in U : r(z) = r(\zeta) \Big\}
\]
and is called the canonical basis for the complex tangent space denoted $H \Gamma_\zeta^r$, being independent of the choice 
of the defining function (upto a permutation to ensure $\partial r/\partial z_n \neq 0$). Next, observe that each $ L_j $ is an  
eigenvector for $M_{\rho}(z)$ with eigenvalue $0$. Hence, $H \Gamma_\zeta^r$ is contained in the kernel of $M_r(\zeta)$ while the vector $\nu(\zeta)$ is an eigenvector of $M_\rho(\zeta)$ of eigenvalue $\vert \nu(\zeta) \vert^2$.
\medskip \\
\noindent There is another way to construct a Hermitian form out of $\rho$ whose associated quadratic form is
\begin{equation} \label{Hessnotn}
\mathcal{L}_{\rho}(z,Y) = \sum\limits_{j,k=1}^{n} \frac{\partial^2 \rho (z) }{\partial z_j \partial \overline{z}_k} Y_j\overline{Y}_k
\end{equation}
for any given $z \in V$ and $Y \in \mathbb{C}^n$. This is the complex Hessian and in general need not be positive or negative semi-definite. Restricted
to $H \Gamma_\zeta^\rho$, this is the standard Levi form. When $\Gamma_\zeta^\rho$ is pseudoconvex, the restriction of $\mathcal{L}_\rho$ to $H \Gamma_\zeta^\rho$ is positive semi-definite.\\

\noindent Henceforth, $C$ will denote a positive constant that may vary from line to line.
Positive constants will also be denoted by $K,L$ or $C_j$ for some integer $j$ (for 
instance as in the last part of Lemma \ref{dprop}) and these may also vary as we move from one part of the text to another.\\

\subsection{Basic properties of the pseudo-distance induced by the biholomorphically distorted polydiscs $Q$} \label{first}
\begin{lem} \label{dprop}
The function $d$ satisfies
\begin{itemize}
\item[(i)] For all $a,b \in U$ we have $d(a,b)=0$ if and only if $a=b$,
\item[(ii)] There exists $C >0$ such that for all $a,b \in U$
\[
d(a,b) \leq C d(b,a)
\]
\item[(iii)] There exists a constant $L>0$ such that for all $a,b \in U$ with $\abs{a-b} <R_0$ one has
\[
d(a,b) \geq 1/2L \; d'(a,b)
\]
\item[(iv)] There exists $C>0$ such that for all $a,b,c \in U$
\[
d(a,b) \leq C \big( d(a,c) + d(b,c) \big)
\]
\item[(v)] With a suitable constant $C>0$ we have $d(a,a^*) \leq C \delta_D(a)$ for all $a \in U$ and

\item[(vi)] There exist constants $C_1,C_2>0$ such  that for all $a,b \in U$ we have
\[
C_1 \abs{a-b}^{2m} \leq d(a,b) \leq C_2 \abs{a-b}
\]
\end{itemize}
\end{lem}

\begin{proof} The proofs of parts (i)-(v) follow exactly as in \cite{Her}. To establish (vi),
observe that $a \in Q_\delta(b)$ if and only if $ \vert \Phi^b(a)_1 \vert < \tau(b,\delta), \vert \Phi^b(a)_\alpha \vert < \sqrt{\delta} $ 
for all $ 2 \leq \alpha \leq n-1 $ and $  \vert \Phi^b(a)_n \vert < \delta $.
As a consequence,
\[
\vert G_b(\tilde{a} - \tilde{b}) - Q_2(a_1 - b_1) \vert < \sqrt{\delta}
\]
which implies that
\[
\vert G_b(\tilde{a} - \tilde{b}) \vert < \sqrt{\delta} + Q_2(a_1 - b_1).
\]
But we already know that 
\[
\vert a_1 - b_1 \vert < \tau(b,\delta) \lesssim \delta^{1/2m}.
\]
Moreover, since $G_b^{-1}$ is uniformly bounded below in norm in a neighbourhood of $b$, we see that
\[
\vert \tilde{a} - \tilde{b} \vert \lesssim \sqrt{\delta} + \tau(b,\delta) \lesssim \delta^{1/2m}.
\] 
It follows that $\vert {}'a - {}'b \vert \lesssim \delta^{1/2m} $ and consequently that $\vert a_n - 
b_n \vert \lesssim \delta^{1/2m} $. To summarize, we conclude that 
$\vert a_1 - b_1 \vert$, $\vert \tilde{a} - \tilde{b} \vert$, 
$\vert {}'a - {}'b \vert$ and $\vert a_n - b_n\vert$ are all less than $\delta^{1/2m}$ times a constant. 
Hence, $d(a,b) \gtrsim \vert a - b \vert^{2m}$. Now to prove the upper inequality of (vi), 
write
\begin{align*}
\big\vert \Phi^b (a) \big\vert^2 &= \big\vert a_1 - b_1 \big\vert^2 + \big\vert G_\zeta(\tilde{a} -
 \tilde{b}) - Q_2(a_1 - b_1) \big\vert^2 + \big \vert(b_n^\zeta)^{-1} \big( a_n - b_n \big) - Q_1({}'a - {}'b) \big\vert^2. 
\end{align*}
Note that the right hand side above is at most $ C_b \vert a - b  \vert^2$, where $C_b$ is the maximum of the absolute 
values of $ G_\zeta $, $ ( b_n^\zeta)^{-1}  $ and the coefficients of the polynomials $Q_1 ({}'a - {}'b)$, $Q_2^\alpha(a_1 - b_1)$ 
for $2 \leq \alpha \leq n-1$, all of which are smooth functions of $b$. Hence, $C_b$ is bounded above by a positive 
constant (say, $C_2$) that depends only on the domain $D$. This proves the lemma.
\end{proof}

\noindent Observe that open balls in this pseudo-distance $ d$ are the analytic polydiscs $Q(\cdot,\delta)$ and hence,
the topology generated by $ d$ coincides with the Euclidean topology. However, part (vi) of Lemma \ref{dprop} shows
that the pseudodistance $ d $ is not bi-Lipschitz equivalent to the Euclidean distance. Furthermore, $ d $ captures certain key 
aspects of the CR-geometry of the boundary of the Levi corank one domain $D$. The goal now is to compare $ d $ with the distance 
between a pair of points of $D$ measured in the space of uniformly bounded holomorphic functions (bounded by $1$ for 
simplicity). 

\subsection{Local domains of comparison and plurisubharmonic weights} \label{second}
\noindent The primary objective of this section is to construct a bounded holomorphic function that separates any given pair of points 
in $D$ or equivalently, in $D_\zeta := \Phi^{\zeta} \big(D \cap B(0, R_0) \big)$ where $ \zeta  \in U $ is fixed. The idea is to 
construct a smooth function first, then modify it to get a holomorphic one via a suitable $\overline{\partial}$ problem. The standard techniques 
of solving a $\overline{\partial}$-problem will yield a smooth $L^2$-solution which can be modified to ensure holomorphicity 
and the desired separation properties. This function will satisfy an $ L^2 $ bound near $ \zeta$. However, we require this function to lie in 
$ H^{\infty}(D) $. Catlin observed that this inclusion can be difficult to prove (cf. \cite{Ca}). To overcome this difficulty,
the domain $D_\zeta$ is bumped near $\zeta$ in a manner that $\partial D_\zeta$ lies well within the bumped domain $D_t^\zeta\; (t \in \mathbb{R}_{+})$. 
The bumping is done carefully, so that the function in $H^2(D_t^\zeta)$ obtained as a solution of the $\overline{\partial}$-problem can be
modified to get an $H^\infty$-function on $D$. As in \cite{Ca} and \cite{Cho1}, the bumped domain $D_t^\zeta$ 
is characterized by the property that its boundary is pushed out as far as possible, subject to the following two constraints: first,  
${\rm dist}(\zeta, \partial D_t^\zeta)<t$ and second,  $\partial D_t^\zeta$ is pseudoconvex. Also, consider the trivial bumping $ D^{t,\zeta} $ of the 
domain $D_\zeta$, defined by $D^{t,\zeta} = \{w \in \mathbb{C}^n \; : \; \rho^\zeta(w) < t \}$. To relate the bumped domain
 $D_t^\zeta$ with $D^{t,\zeta} $, introduce the one parameter family (parametrized by $t \in \mathbb{R}_{+}$) of functions 
\begin{equation*}
\hat{J}_{\zeta,t}(w) = J_{\zeta,t}(w) - t,
\end{equation*}
where 
\begin{equation}
J_{\zeta,t}(w) = \Big( t^2 + \vert w_n \vert^2 + \sum\limits_{j=2}^{2m} \vert P_j(\zeta, \cdot) \vert^2 \vert w_1 \vert^{2j} + 
\vert w_2 \vert^4 + \ldots + \vert w_{n-1} \vert^4 \Big)^{1/2}.
\end{equation}
It turns out that $J_{\zeta,t}(w)$ is useful in estimating the $ L^2 $ norms of holomorphic functions on $D_t^\zeta$ 
(cf. Lemma \ref{lem5.1} of sub-section \ref{fourth}). Moreover, $\hat{J}_{\zeta,t}(w)$ can be regarded as a pseudonorm: 
$\hat{J}_{\zeta,t}(w-z)$ is symmetric and satisfies the triangle inequality up to a positive constant. Indeed, let 
$P(\vert w \vert)$ denote the expression
\[
 \vert w_n \vert^2 + \sum\limits_{j=1}^{2m} \vert P_j(\zeta,\cdot) \vert^2 \vert w_1 \vert^{2j} + \vert w_2 \vert^4 + \ldots + \vert w_{n-1} \vert^4.
\]
It is well known that
\[
\vert A + B \vert^{2j} \leq 2^{2j-1} ( \vert A \vert^{2j} + \vert B \vert^{2j})
\]
for $ A, B \in \mathbb{R} $ and $ j \in \mathbb{N} $. Applying the above inequality to each term in $ P(|w + z|) $, it follows that
, $ P (|w + z|) \lesssim P(w) + P(z) $. This estimate yields the triangle inequality for $\hat{J}_{\zeta,t}(w)$. \\

\noindent Note that the bumping technique can be applied to a small neighbourhood of $\zeta$. It is, therefore, natural to 
examine the region where $\vert \rho^\zeta( \cdot ) \vert$ differs from $J_{\zeta,t}(\cdot)$ 
by a small constant. More precisely, consider a tubular neighbourhood $ U_{\zeta,t} $ of $\partial D_\zeta$, defined by
\begin{equation}
U_{\zeta,t} = \{w: \vert \rho^\zeta(w) \vert < s \vert J_{\zeta,t}(w) \vert \} \cap B(0, R_1), \; 0 < R_1 < R_0
\end{equation}
where $s $ is a small positive constant.\\

\noindent Observe that $\rho^\zeta(w)<0$ whenever $w \in D_\zeta =\Phi^\zeta \big( D \cap B(0,R_0) \big)$. In this case,
$U_{\zeta,t}$ is described by $\rho^\zeta(w) > - s J_{\zeta,t}(w)$, which determines a one-sided neighbourhood $\mathcal{U}$ 
of $\partial D_{\zeta} $. Bumping the boundary $ \partial D_{\zeta} $ of the domain will yield a pseudoconvex 
hypersurface $\mathcal{B}^\zeta$. Denote by $\mathcal{B}_\zeta^{int}$, the piece of the hypersurface $\mathcal{B}^\zeta$ which lies
within $D_t^\zeta$ and bounds the one-sided neighbourhood $\mathcal{U}$. Note that
$\rho^\zeta(w)$ equals $-s J_{\zeta,t}(w)$ on $\mathcal{B}_\zeta^{int}$. We claim that the defining function $\rho^{\zeta,t}(w)$
of the bumped domain behaves analytically like the algebraic function $J_{\zeta,t}$ i.e., $\rho^{\zeta,t} \approx J_{\zeta,t}$. To this end,
the following result will be useful (see Proposition 2.2 and Theorem 2.4 of \cite{Cho1} for a proof).

\begin{lem} \label{maxhess}
For sufficiently small $R_1<R_0$, $ s> 0 $ and each $\zeta \in \partial D \cap B(0, R_1) $,
there exists on $\{ \rho^\zeta<t\} \cup U_{\zeta,t}$, a smooth negative real valued  plurisubharmonic function $E_{\zeta,t}$ with the following properties:
\begin{itemize}
\item[(i)] $-C_3 J_{\zeta,t} \leq E_{\zeta,t} \leq -1/C_3 J_{\zeta,t}$ \\
\item[(ii)] The complex Hessian of $E_{\zeta,t}$ satisfies 
\[
\mathcal{L}_{E_{\zeta,t}}(w,Y) \approx J_{\zeta,t}(w) \Big(\; \Big\vert \frac{  Y_1 }{ \tau(\zeta, J_{\zeta,t}(w)) } 
\Big\vert^2 +  \sum\limits_{k=2}^{n-1} \Big\vert \frac{ Y_k }{\sqrt{J_{\zeta,t}(w)}} \Big\vert^2 + \Big\vert \frac{ Y_n }{ J_{\zeta,t}(w) } \Big\vert^2 \; \Big)
\]
for all $w \in \{ \rho^\zeta <t \} \cup U_{\zeta,t}$ and $Y \in \mathbb{C}^n$.
\item[(iii)] The first directional derivative of $E_{\zeta,t}$ satisfies
\[
\vert \langle \partial E_{\zeta,t}(w), Y \rangle \vert^2   \leq C_4 J_{\zeta,t}(w) \mathcal{L}_{E_{\zeta,t}}(w,Y)
\]
\item[(iv)] Let $\rho^{\zeta,t} = \rho^\zeta + \epsilon_0 E_{\zeta,t}$. Then for small enough $\epsilon_0>0$ the domain
\[
D_t^\zeta = \Big\{ w \in \{ \rho^\zeta<t \} \cup U_{\zeta,t} \; : \;  \rho^{\zeta,t}(w)<0 \Big\} 
\]
is pseudoconvex.
\end{itemize} 
\end{lem}

\noindent The constants $C_3, C_4$ and $ s $ above are positive and independent of the parameters $t$ and $\zeta$.         
The part (iii) above provides a comparison between the complex Hessian of $E_{\zeta,t}(w)$ and the positive 
semi-definite Hermitian form $\vert \langle \partial E_{\zeta,t} (w),Y \rangle \vert^2$ obtained from the smooth function
$E_{\zeta,t}$. In part (iv) above, the bumped domain $ D_t^\zeta $ contains $U_{\zeta,t}$ and a portion of the trivially bumped domain $D^{t, \zeta}$.
This ensures that the boundary $\partial D_t^\zeta$ is pseudoconvex, and also that the ratio $\vert \rho^{\zeta,t}(w) / J_{\zeta,t}(w) \vert$ is small.
Moreover, it follows that
\begin{equation} \label{norest}
\langle \partial \rho^{\zeta,t}(w), Y \rangle = \langle \partial \rho^\zeta(w),Y \rangle + \epsilon_0 \langle \partial E_{\zeta,t}(w),Y \rangle
\end{equation}
for $ Y \in \mathbb{C}^n $. Setting $w=0$ and restricting $ Y $ to the complex tangent space to $D^\zeta$ at the origin, i.e., $ 
\{z \in \mathbb{C}^n: z_n = 0 \}$, we see that
\[
\langle \partial \rho^{\zeta,t}(0),Y \rangle = \epsilon_0 \langle \partial E_{\zeta,t}(0),Y \rangle.
\]
But we already know from parts (iii) and (ii) that
\begin{align} 
\vert \langle \partial \rho^{\zeta,t}(0),Y \rangle \vert^2  \leq C_4 \epsilon_0^2 \big( J_{\zeta,t}(0) \big)^2  \left(  \left \vert \frac{Y_1}
{\tau \big(\zeta, J_{\zeta,t}(0) \big)} \right \vert^2 + \sum\limits_{\alpha=2}^{n-1} \left \vert \frac{Y_\alpha}{\sqrt{J_{\zeta,t}(0)}} \right\vert^2 \right) 
\lesssim \epsilon_0^2 \vert {}' Y \vert^2 \label{Oest1}.
\end{align}
since $\tau \big(\zeta, J_{\zeta,t}(0) \big) \gtrsim \sqrt{J_{\zeta,t}(0)} $ and $J_{\zeta,t}(0)=t$. Now, consider
\begin{align*} 
\mathcal{L}_{\rho^{\zeta,t}}(0,Y) = \mathcal{L}_{\rho^\zeta}(0,Y) + \epsilon_0 \mathcal{L}_{E_{\zeta,t}}(0,Y).
\end{align*}
Recall that $  \mathcal{L}_{\rho^\zeta}(0,Y) $ is positive semi-definite due to pseudoconvexity of $ D_{\zeta} $ at the origin. Hence,
\begin{align} 
\mathcal{L}_{\rho^{\zeta,t}}(0,Y) & \gtrsim \epsilon_0 J_{\zeta,t}(0) \left( \left \vert \frac{Y_1}{\tau \big(\zeta, J_{\zeta,t}(0) \big)} \right \vert^2 + 
\sum\limits_{\alpha=2}^{n-1} \left \vert \frac{Y_\alpha}{\sqrt{J_{\zeta,t}(0)}} \right \vert^2 \right) \nonumber  \\ \label{Oest2}
& = \epsilon_0 t \Big( \big \vert \frac{Y_1}{\tau(\zeta,t)} \big \vert^2 + \frac{1}{t} \sum\limits_{\alpha=2}^{n-1} \big \vert Y_\alpha \big \vert^2 \Big). 
\end{align}
Using the fact that $ \tau (\zeta,t) \lesssim t^{1/2m}$, we get that 
\[
\mathcal{L}_{\rho^{\zeta,t}}(0,Y) \gtrsim \epsilon_0  t^{1-1/m} \vert {}'Y \vert^2.
\]
Observe that (\ref{Oest1}) and (\ref{Oest2}) together imply that the bumped domains touch the domain $D_\zeta$ 
minimally in the complex tangential directions and approximately to the first order along the normal direction at the origin. Furthermore, 
rewriting (\ref{norest}) together with Lemma \ref{maxhess}(iii) yields the following estimates about the variation of the normal vector 
fields with respect to the bumping procedure:
\begin{align*}
\big \vert \langle \partial \rho^{\zeta,t}(w), Y \rangle  - \langle \partial \rho^{\zeta}(w), Y \rangle \big \vert^2 
&= \epsilon_0^2 \big \vert \langle \partial E_{\zeta,t}(w),Y \rangle \big \vert^2\\
& \leq C_4 \epsilon_0^2 \big( J_{\zeta,t}(w)\big)^2 \left( \left \vert \frac{Y_1}{\tau \big(\zeta, J_{\zeta,t}(w) \big)} \right \vert^2 + 
\sum\limits_{\alpha=2}^{n-1} \left \vert \frac{Y_\alpha}{\sqrt{J_{\zeta,t}(w)}} \right\vert^2 + \left\vert \frac{Y_n}{J_{\zeta,t}(w)}  \right\vert^2 \right) \\
& = C_4 \epsilon_0^2  \left( \left \vert \frac{J_{\zeta,t}(w)}{\tau (\zeta, J_{\zeta,t}(w))}\right \vert^2 \vert Y_1 \vert^2 + 
\sum\limits_{\alpha=2}^{n-1} \big \vert \sqrt{J_{\zeta,t}(w)} Y_\alpha \big\vert^2 + \big\vert Y_n \big \vert^2 \right) \\
&\lesssim C_4 \epsilon_0^2 \Big( J_{\zeta,t}(w) \big\vert {}'Y \big\vert^2 + \vert Y_n \vert^2 \Big).
\end{align*} 
Here we use the fact that $ \tau \big(\zeta, J_{\zeta,t}(w) \big) \gtrsim J_{\zeta,t}(w)^{1/2} $ and that $J_{\zeta,t}(w)\ll 1$ for $w$ near the origin. 
Then, the above analysis allows us to conclude that the bumping process pushes the boundary $\partial D^\zeta_t$ in the `normal' direction 
$Y_n$  relatively more than in the complex tangential directions. \\

\noindent As in \cite{Cho1} and \cite{Her}, the domains $D_t^\zeta$, serve as local pseudoconvex domains of comparison containing 
certain special polydiscs $ P(w,t,\theta) $ of optimal size: 
\[
P(w,t,\theta) = \Delta \Big( w_1, \tau\big(\zeta,\theta J_{\zeta,t}(w) \big) \Big) \times \Delta \big(w_2,  \sqrt{ \theta J_{\zeta,t}(w)} \big) \times \ldots \times \Delta \big(w_{n-1}, \sqrt{ \theta J_{\zeta,t}(w)} \big) \times \Delta \big(w_n, \theta J_{\zeta,t}(w) \big) 
\]
Evidently, these polydiscs are the pull backs of the unit polydisc around $ w $, under the affine scalings given by
\begin{alignat*}{3}
\Delta_{\zeta}^{\theta J_{\zeta,t}(w)}(z_1, \ldots, z_n) = \left( \frac{z_1}{\tau \big(\zeta, \theta J_{\zeta,t}(w) \big)}, \frac{z_2} {\sqrt{\theta J_{\zeta,t}(w)}}, \ldots,
\frac{z_{n-1}}{ \sqrt{\theta J_{\zeta,t}(w)}}, \frac{z_n} {\theta J_{\zeta,t}(w) } \right).
\end{alignat*}
The scaling factors of this linear map or equivalently, the polyradii of $ P(w,t,\theta) $ are so chosen that Lemma \ref{Ppolydisc} is true.
We use the Cauchy integral formula for holomorphic functions on the polydisc -- Let $ \Omega $ be a domain in $ \mathbb{C}^n $ and $ h \in \mathcal{O}(\Omega) $. 
Then for each $ z \in \Omega $ and each polydisc $ P = \Delta^1 \times \ldots \times \Delta^n $ around $ z $ such that $ P $ is compactly contained in $ \Omega$,
it follows from the Cauchy integral formula that
\begin{equation} \label{valest}
\vert h(z) \vert  \leq C \frac{\parallel h \parallel_{L^2}}{ \delta_{\Delta^1}(z_1) \times \ldots \times \delta_{\Delta^n}(z_n)}
\end{equation}
for some universal constant $ C $ that depends only on the dimension $n$. Here, each $\Delta^j$ denotes a disc in the $\mathbb{C}_{z_j}$-plane around at $z_j$. 
Recall that the $L^2$-norm of a function is bounded above by its $L^\infty$-norm (up to a constant) on any compact measure space. Observe that (\ref{valest})
enables an estimate with the inequality reversed. This estimate albeit well-known and elementary, will play a key role (for instance, (\ref{elest})) in rendering a local solution to 
our problem and another primary result (i.e., Lemma \ref{cardist}) towards our end goal. The following lemma from \cite{Cho1} will be useful for our purposes.

\begin{lem} \label{Ppolydisc}
There exist numbers $M_0, r_0, \theta>0$ such that 
\begin{itemize}
\item[(a)] For all $\zeta \in \partial D$, any $t>0$ and $w \in D^\zeta \cap B(o, {r_0})$ one has 
\[
D_t^\zeta \supset P(w,t,\theta). 
\]
\item[(b)] Next, the variation of the defining function $\rho^\zeta$ on the polydisc $P(w,t,\theta)$ is described by the statement 
\[
 \vert \rho^\zeta(x) \vert \leq M_0 \theta J_{\zeta,t}(w). 
\]
for all $x \in P(w, t, \theta)$.
\end{itemize} 
\end{lem}

\noindent Note that the bumping function $E_{\zeta,t}$, is obtained by taking a suitable weighted sum of a family of uniformly 
bounded plurisubharmonic functions $\{\lambda_\delta\}$ (parametrized by $\delta$), whose Hessians satisfy 
maximal lower bounds near the point $\zeta \in D_\zeta $. Furthermore, the derivatives of $\{\lambda_\delta\}$ satisfy 
uniform upper bounds in terms of the weights with respect to the inverses of the 
multitype (for the precise statements, see \cite{Cho1}). Applying Diederich -- Fornaess modification (cf.\cite{DF1}) to $ \rho^{\zeta, t} $ yields a 
family of smooth functions $\psi_{\zeta,t}^w$ which are strongly plurisubharmonic on $D_t^\zeta$. It turns out that
their Hessians satisfy sharp lower bounds on the polydisc $ P(w, t, \theta) \subset D^t_{\zeta} $. These lower bounds, in Lemma
\ref{Greencandidate}(ii), are in terms of the polyradii of $ P(w, t, \theta) $. This is owing to a finer control on the derivatives 
of $\lambda_\delta$ for Levi corank one domains, as opposed to a less sharp control known 
for general smooth pseudoconvex domains of finite type. The reader is referred to \cite{Cho} for general but less refined bumping constructs 
which can be used to derive a lower bound on the infinitesimal Kobayashi metric. It turns out that $\psi_{\zeta,t}^w$ is a plurisubharmonic barrier function 
for $\zeta \in D_\zeta$ of algebraic growth (see \cite{Cho}, for a proof). In particular,  $\partial D_\zeta$ near $\zeta$ is regular 
in the sense Sukhov (cf. \cite{Su}) and B-regular in the sense of Sibony (see \cite{S2} and \cite{S}). This, in turn, implies 
that $\overline{D}_\zeta$ has a Stein neighbourhood basis. This fact will be used in the sequel. But, before going further, let us put down 
the construction of $ \psi_{\zeta,t}^{w} $. This will be useful for constructing the weight functions.

\begin{lem} \label{Greencandidate}
After shrinking $\theta$, given $w \in D_t^\zeta \cap \{ \rho^\zeta <0 \}$ there exists on $D_t^\zeta$ a plurisubharmonic function $\psi_{\zeta,t}^w <0$ such that 
\begin{itemize}
\item[(i)] $\psi_{\zeta,t}^{w} \geq -1$ on $P(w,t, \theta)$
\item[(ii)] For any $Y \in \mathbb{C}^n$ and $x \in P(w,t, \theta)$, we have  that 
\[
\mathcal{L}_{\psi_{\zeta,t}^w}(x,Y) \geq C_5 \Big( \frac{\vert Y_1 \vert^2}{\big( \tau(\zeta, J_{\zeta,t}(w) \big)^2} + \sum\limits_{k=2}^{n-1} \frac{\vert Y_k \vert^2}{J_{\zeta,t}(w)} +  \frac{\vert Y_n \vert^2}{\big( J_{\zeta,t}(w) \big)^2}\Big)
\]
where $C_5$ is a positive constant. 
\end{itemize}
\end{lem}

\begin{proof} To construct $ \psi_{\zeta,t}^{w} $, we use Diederich -- Fornaess technique (\cite{DF1}) of constructing a bounded 
strongly plurisubharmonic exhaustion function from any given defining function for a smoothly bounded pseudoconvex domain. Applying
their constructs to $\rho^{\zeta,t}$, the defining function for the bumped domain $D_t^{\zeta}$ , we get 
\[
\psi(x) = - \Big( - \rho^{\zeta,t}(x) e^{-K \vert x \vert^2} \Big)^{\eta_0}
\]
which is strongly plurisubharmonic on $D_t^\zeta$. Following \cite{Her}, we define 
\[
\psi_{\zeta,t}^w(x) = \psi(x) \big/ J_{\zeta,t}(w)^{\eta_0}.
\]
First note that the pseudoconvexity of $ \Phi^\zeta \big(\partial D \cap {B}(\zeta,R_0) \big)$ gives
\begin{equation} \label{pscvxity}
\mathcal{L}_{\rho^\zeta}(x,Y) \geq -K_1 \vert \rho^\zeta(x) \vert \vert Y \vert^2 - K_1 \vert Y \vert \vert \langle \partial \rho^\zeta(x),Y \rangle \vert
\end{equation}
whenever $x \in \Phi^\zeta \big(B(\zeta,R_0) \big)$, for some constant $K_1>0$. Furthermore, a standard calculation of 
the complex Hessian of $\psi$ (see \cite{R} for details)  yields
\begin{align*}
\mathcal{L}_\psi (x,Y) &= \eta_0 \vert \psi(x) \vert \left( \frac{1-\eta_0}{2} \frac{\vert\langle \partial \rho^{\zeta,t}(x),Y \rangle \vert^2 }{ \rho^{\zeta,t}(x)^2}  
+ \Big\lvert \sqrt{\frac{1-\eta_0}{2}} \frac{\vert\langle \partial \rho^{\zeta,t}(x),Y \rangle \vert }{ \rho^{\zeta,t}(x)} + 
\frac{\sqrt{2}K\eta_0}{\sqrt{1-\eta_0}} \overline{\langle x,Y \rangle} \Big \vert^2 \right. \\
&\; \left. \hspace{0.8in} + K \Big( \vert Y \vert^2 - K \eta_0 \big( 1 + \frac{2\eta_0}{1-\eta_0} \big) \vert \langle x, Y \rangle \vert^2 \Big) + 
\frac{ \mathcal{L}_{\rho^{\zeta,t}} (x,Y) }{\rho^{\zeta,t}(x)} \right) \\
&\geq \eta_0 \vert \psi(x) \vert \left( \frac{1-\eta_0}{2} \frac{\vert\langle \partial \rho^{\zeta,t}(x),Y \rangle \vert^2 }{ \rho^{\zeta,t}(x)^2} + 
K \Big( \vert Y \vert^2 - K \eta_0 \big( 1 + \frac{2\eta_0}{1-\eta_0} \big) R_*^2 \vert Y \vert^2 \Big) \right.\\
&\; \left. \hspace{0.8in} + \frac{ \mathcal{L}_{\rho^{\zeta,t}} (x,Y) }{\rho^{\zeta,t}(x)} \right).
\end{align*}
Here we used the estimate $\vert \langle x, Y \rangle \vert^2 \leq R_*^2 \vert Y \vert^2$ for some positive constant $R_*$. Choosing 
$K$ and $\eta_0$ suitably, we can ensure that the following inequality holds throughout $D_t^\zeta$:
\begin{equation} \label{psiest}
\mathcal{L}_\psi (x,Y) \geq \eta_0 \vert \psi(x) \vert \Big( \frac{1-\eta_0}{2} \frac{\vert\langle \partial 
\rho^{\zeta,t}(x),Y \rangle \vert^2 }{ \rho^{\zeta,t}(x)^2} +\frac{1}{2} K \vert Y \vert^2 + \frac{ \mathcal{L}_{\rho^{\zeta,t}} (x,Y) }{\rho^{\zeta,t}(x)} \Big).
\end{equation}
To verify (i), let $x \in P(w,t,\theta)$. Writing this analytically, translates to the following string of inequalities:
\begin{align} 
\vert x_1 - w_1 \vert &<  \tau \big( \zeta, \theta J_{\zeta,t}(w) \big), \nonumber \\ \label{Pdefn}
\vert x_\alpha - w_\alpha \vert &< \sqrt{\theta J_{\zeta,t}(w)} \; \; \text{ for all } 2 \leq \alpha \leq n-1, \text{ and } \\
\vert x_n - w_n \vert &< \theta J_{\zeta,t}(w). \nonumber
\end{align}

\noindent Recall that $J_{\zeta,t}(w) = \sqrt{t^2 + P(\vert w \vert)} $, where $P(\vert w \vert)$ is the pseudo-norm
\[
P(\vert w \vert) = \vert w_n \vert^2 + \sum\limits_{j=2}^{2m} \vert P_j(\zeta)\vert^2 \vert w_1 \vert^{2j} + \vert w_2 \vert^4 + \ldots +\vert w_{n-1} \vert^4
\]
and write the difference $J_{\zeta,t}(x) - J _{\zeta,t}(w)$ as
\begin{equation}\label{Jest}
J_{\zeta,t}(x) - J _{\zeta,t}(w) = \frac{\big( J_{\zeta,t}(x) \big)^2 - \big( J_{\zeta,t}(w) \big)^2} { J_{\zeta,t}(x) + J _{\zeta,t}(w)}.
\end{equation} 
\noindent The aim now is to obtain an analogue of the estimate (4.3) in \cite{Her}. This estimate is essentially an assertion 
about the uniform comparability $J_{\zeta,t}(x) \approx J_{\zeta,t}(y)$. Furthermore, such an estimate will lead to the 
engulfing property for the polydiscs $P(w,t,\theta)$ analogous to that for the Catlin -- Cho polydiscs $Q( \cdot,\cdot)$. 
To this end, we begin by applying the triangle inequality to the numerator $(=\big\vert P(\vert w \vert) - P( \vert x \vert)\big\vert)$ 
on the right side of (\ref{Jest}) which gives
\begin{multline} \label{Nr}
\Big \vert \vert x_n \vert^2 - \vert w_n \vert^2 + \sum \limits_{\alpha=2}^{n-1} \big( \vert x_\alpha \vert^4 - \vert w_\alpha \vert^4 \big) + \sum\limits_{j=2}^{2m} \vert P_j(\zeta) \vert^2 \big( \vert x_1 \vert^{2j} - \vert w_1 \vert^{2j} \big) \Big \vert \\
\leq \big \vert \vert x_n \vert^2 - \vert w_n \vert^2 \big \vert + \sum \limits_{\alpha=2}^{n-1} \big \vert \vert x_\alpha \vert^4 - \vert w_\alpha \vert^4 \big \vert + \sum\limits_{j=2}^{2m} \vert P_j(\zeta) \vert^2 \big \vert \vert x_1 \vert^{2j} - \vert w_1 \vert^{2j}  \big \vert.
\end{multline}
\noindent Note that $\big \vert \vert x_n \vert^2 - \vert w_n \vert^2 \big \vert \lesssim \vert  x_n -  w_n\vert \lesssim \theta J_{\zeta,t}(w)$ 
by the last inequality in (\ref{Pdefn}) and similarly,
\[
\sum \limits_{\alpha=2}^{n-1} \big\vert  \vert x_\alpha \vert^4 -\vert w_\alpha \vert^4 \big \vert \leq 
2^3 \sum \limits_{\alpha=2}^{n-1} \big(\sqrt{\theta J_{\zeta,t}(w)}\big)^4 \lesssim \theta^2 \big( J_{\zeta,t}(w) \big)^2 \lesssim \theta \big( J_{\zeta,t}(w) \big)
\]
since $\theta$ and $J_{\zeta,t}(w)$ are very small. To estimate the last summand of (\ref{Nr}), we use the following well-known inequality:
\[
\vert x_1 \vert^{2j} \leq 2^{2j-1}( \vert x_1 - w_1 \vert^{2j} + \vert w_1 \vert^{2j} ),
\]
which implies that 
\[
\Big \vert \vert x_1 \vert^{2j} - \vert w_1 \vert^{2j} \Big \vert \leq 2^{2j-1} \vert x_1 - w_1 \vert^{2j} + (2^{2j-1} -1) \vert w_1 \vert^{2j}
\]
for each $ j \leq 2m $. Hence, 
\begin{equation}\label{Pest}
\vert P_j(\zeta) \vert^2 \big \vert \vert x_1 \vert^{2j} - \vert w_1 \vert^{2j} \big \vert \leq 2^{2j-1} \vert P_j(\zeta;) \vert^2 \vert x_1 - w_1 \vert^{2j} 
+ (2^{2j-1} -1) \vert P_j(\zeta;) \vert^2\vert w_1 \vert^{2j}.
\end{equation}
It follows from the definition of $\tau \big(\zeta, \theta J_{\zeta,t}(w) \big)$ that 
\begin{equation*} 
\vert P_j(\zeta;) \vert^2 \leq \frac{\theta^2 (J_{\zeta,t}(w))^2}{\Big( \tau \big( \zeta, \theta J_{\zeta,t}(w) \big) \Big)^{2j}}.
\end{equation*}
But we know from (\ref{Pdefn}) that  $\vert x_1 - w_1 \vert^{2j} \leq  \Big( \tau \big( \zeta, \theta J_{\zeta,t}(w) \big) \Big)^{2j}$. Therefore, 
\begin{equation} \label{f1}
2^{2j-1} \vert P_j(\zeta;) \vert^2 \vert x_1 - w_1 \vert^{2j} \leq 2^{2j-1} \theta^2 \big( J_{\zeta,t}(w) \big)^2 \\
 \lesssim \theta J_{\zeta,t}(w).
\end{equation}
Finally, since $\tau \big(\zeta,  \theta J_{\zeta,t}(w) \big)  \gtrsim  \big( \theta J_{\zeta,t}(w) \big)^{1/2}$,  
\[
\vert P_j(\zeta;) \vert^2 \lesssim \frac{\theta^2(J_{\zeta,t}(w))^2}{\theta^j(J_{\zeta,t}(w))^j}.
\]
Moreover, $\vert w_1 \vert^{2j} \lesssim \theta^j J_{\zeta,t}(w)^j$ and this shows that the second summand in (\ref{Pest}) is also 
bounded above by $\theta J_{\zeta,t}(w)$. Also note that the denominator in (\ref{Jest}) is bounded and $ |w| $ and $ |x| $ are small. 
The above analysis yields an estimate that is analogous to (4.3) of \cite{Her} in our setting, i.e.,
\begin{equation}\label{Jest2}
\vert J_{\zeta,t}(x) - J_{\zeta,t}(w) \vert \leq C_6 \theta J_{\zeta,t}(w)
\end{equation}
or equivalently that
\begin{align*}
\vert J_{\zeta,t}(x) \vert \leq (1+C_6 \theta) J_{\zeta,t}(w) \; \; \text{and} \; \;
\vert J_{\zeta,t}(x) \vert  \geq (1 -C_6 \theta) J_{\zeta,t}(w) 
\end{align*}
for some constant $C_6>0$. Here, $\theta$ is chosen small enough so as to ensure $1-C_6 \theta >0$. \\

\noindent  Then, it follows from Lemma \ref{maxhess}(i) and Lemma \ref{Ppolydisc}(b) that
\begin{align*}
\vert \rho^{\zeta,t}(x) \vert &\geq \epsilon_0 \vert E_{\zeta,t}(x) \vert - \vert \rho^\zeta(x) \vert \\
& \geq \epsilon/C_3 \; J_{\zeta,t} (x) - M_0 \theta J_{\zeta,t}(w) \\
&\geq \epsilon_0/C_3  (1- C_6 \theta)  J_{\zeta,t}(w) - M_0 \theta J_{\zeta,t}(w)\\
&\geq \epsilon_0/2 C_3 \; J_{\zeta,t}(w)
\end{align*}
and
\begin{align*}
\vert \rho^{\zeta,t} (x)  \vert &\leq \epsilon_0 \vert E_{\zeta,t}(x) \vert + \vert \rho^\zeta(x) \vert \\
& \leq \epsilon_0 C_3 \; J_{\zeta,t}(x) + M_0 \theta J_{\zeta,t}(w) \\
&\leq C_3 \epsilon_0 (1+ C_6\theta) J_{\zeta,w}(w) + M_0 \theta J_{\zeta,w}(w)\\
& \lesssim J_{\zeta,t}(w) 
\end{align*}
provided $\theta$ and $\epsilon_0$ are sufficiently small. Putting the above two observations together, we see that
\begin{equation} \label{roJapprox}
(\epsilon_0/2C_3) J_{\zeta,t}(w) \leq \vert \rho^{\zeta,t}(x) \vert \leq J_{\zeta,t}(w).
\end{equation}
Furthermore,
\begin{equation} \label{Eest}
\vert E_{\zeta,t}(x) \vert \leq C_3 J_{\zeta,t}(x) \leq C_3(1+C_6 \theta)J_{\zeta,t}(w) \leq C_7 \vert \rho^{\zeta,t}(x) \vert,
\end{equation}
which implies that
\[
\vert \rho^{\zeta,t}(x) \vert \approx \vert E_{\zeta,t}(x) \vert \approx J_{\zeta,t}(w).
\]
for all $x \in P(w, t, \theta)$. Since $\rho^{\zeta,t}(x) \approx {\rm dist}(x, \partial D_t^\zeta)$, it follows that 
\[
\vert {\rm dist}(x, \partial D_t^\zeta) - {\rm dist}(x,\partial D^{t,\zeta}) \vert \approx \hat{J}_{\zeta,t}(x)
\]
where $ \hat{J}_{\zeta,t}(x) = J_{\zeta,t}(x) -t$, and  $D^{t,\zeta} = \big\{ w \in \mathbb{C}^n: 
\rho_{t,\zeta} := \rho^\zeta (w) -t <0 \big \}$. Thus 
the pseudo-norm $\hat{J}_{\zeta,t}$ measures the difference between the trivial bumpings 
given by $\{\rho_{t,\zeta}\}$ and the bumpings $\{ \rho^{\zeta,t} \}$ engineered taking into account 
the CR geometry of $\partial D_\zeta$. It follows from (\ref{roJapprox}) that 
$\psi(x) \geq - \vert \rho^{\zeta,t}(x) \vert^{\eta_0} \geq - \big( J_{\zeta,t}(w) \big)^{\eta_0}$ for all $x \in P(w,t,\theta)$ 
and consequently that
\[
\psi^w_{\zeta,t}(x) = \frac{\psi(x)}{\big( J_{\zeta,t}(w) \big)^{\eta_0}} \geq -1.
\]

\noindent To establish (ii), let us examine the complex Hessian of $\rho^{\zeta,t}$:
\begin{align} \label{A}
\mathcal{L}_{\rho^{\zeta,t}}(x,Y) &= \mathcal{L}_{\rho^\zeta}(x,Y) + \epsilon_0 \mathcal{L}_{E_{\zeta,t}}(x,Y)  \nonumber \\
&\geq -K_1 \big( \vert \rho^\zeta (x) \vert \vert Y_1 \vert^2 + \vert \langle \partial \rho^\zeta(x),Y \rangle \vert \vert Y \vert \big) + 
\epsilon_0 \mathcal{L}_{E_{\zeta,t}}(x,Y) \;\; \text{ by } (\ref{pscvxity}) \nonumber \\
&\geq -K_1 \big( \vert \rho^{\zeta,t}(x) \vert + \epsilon_0 \vert E_{\zeta,t}(x) \vert \big) \vert Y \vert^2 - 
K_1 \vert \langle \partial \rho^{\zeta,t}(x), Y \rangle \vert \vert Y \vert \nonumber\\
&\;\;\; - K_1 \epsilon_0 \vert \langle \partial E_{\zeta,t}(x),Y \rangle \vert \vert Y \vert + \epsilon_0 \mathcal{L}_{E_{\zeta,t}}(x,Y)
\end{align}
for $ x \in P(w, t, \theta) $. Note that, by (\ref{Eest}), the first term 
$ -K_1 \big( \vert \rho^{\zeta,t}(x) \vert + \epsilon_0 \vert E_{\zeta,t}(x) \vert \big) \vert Y \vert^2 $ in (\ref{A}) is 
at least
\begin{equation} \label{Iest}
-K_1 \big( \vert \rho^{\zeta,t}(x) \vert + \epsilon_0 C_7 \vert \rho^{\zeta,t} (x) \vert \big) \vert Y \vert^2.
\end{equation}
Now, consider the third term in (\ref{A}):
\begin{equation}\label{IIIest}
- \epsilon_0 K_1 \vert \langle \partial E_{\zeta,t}(x),Y \rangle \vert \vert Y \vert \geq 
- \epsilon_0 K_1 \sqrt{C_3} \sqrt{J_{\zeta,t}(x)} \sqrt{\mathcal{L}_{E_{\zeta,t}}(x,Y)}\vert Y \vert
\end{equation}
by Lemma \ref{maxhess}(iii). The part (ii) of the same lemma gives
\begin{align} \label{Lest}
\sqrt{\mathcal{L}_{E_{\zeta,t}}(x,Y)} &\lesssim \frac{\sqrt{J_{\zeta,t}(x)}}{\sqrt{C_3}} \left( 
\frac{\vert Y_1 \vert}{ \big( \tau( \zeta, J_{\zeta,t}(x) \big)^2} + \sum \limits_{\alpha=2}^{n-1}
\frac{\vert Y_\alpha \vert}{J_{\zeta,t}(x)} + \frac{\vert Y_n \vert^2}{\big( J_{\zeta,t}(x) \big)^2} \right)^{1/2} \nonumber \\
&=\frac{1}{\sqrt{C_3}} \left( \frac{J_{\zeta,t}(x)}{ \tau \big( \zeta, J_{\zeta,t}(x) \big)^2}\vert Y_1 \vert^2 \; 
+ \sum \limits_{\alpha=2}^{n-1}\vert Y_\alpha \vert^2 + \frac{\vert Y_n \vert^2}{J_{\zeta,t}(x)} \right)^{1/2}.
\end{align}
Using the fact that $  \tau \big(\zeta,J_{\zeta,t}(x) \big) \gtrsim \big( J_{\zeta,t}(x) \big)^{1/2} $ in (\ref{Lest}), we get 
\[
 \sqrt{J_{\zeta,t}(x)}\sqrt{C_3} \sqrt{\mathcal{L}_{E_{\zeta,t}}(x,Y)} \lesssim \Big( \vert Y_1 \vert^2 + 
J_{\zeta,t}(x) \sum \limits_{\alpha=2}^{n-1}\vert Y_\alpha \vert^2 + \vert Y_n \vert^2 \Big)^{1/2}.
\]
Since $J_{\zeta,t}(x) $ is small, it follows that
\begin{align} \label{IIIFest}
- \epsilon_0 K_1 \vert \langle \partial E_{\zeta,t}(x),Y \rangle \vert \vert Y \vert \gtrsim 
- \epsilon_0 K_1  \vert Y \vert^2  \gtrsim - \epsilon_0 K_1 \vert \rho^{\zeta,t}(x)\vert  \vert Y \vert^2.
\end{align}
Here we used that $\vert \rho^{\zeta,t}(x) \vert$ is bounded away from zero on the polydiscs $P(w,t,\theta) \subset D_t^\zeta$. To 
estimate the second term in (\ref{A}), we use the following version of the Cauchy-Schwarz inequality: 
\[
xy \leq \epsilon/2 \; x^2 + 1/2\epsilon \; y^2
\]
where $x,y$ and $\epsilon$ are positive numbers. Applying this inequality for $x = \vert Y \vert$, 
$y= \vert \langle \partial \rho^{\zeta,t}(x),Y \rangle \vert$ and  $\epsilon=\vert \rho^{\zeta,t}(x) \vert$, we see that 
\begin{equation} \label{IInd}
- K_1 \vert \langle \partial \rho^{\zeta,t}(x), Y \rangle \vert \vert Y \vert \geq
 -K_1/2 \; \vert \rho^{\zeta,t}(x) \vert \vert Y \vert^2 - K_1/2 \; \vert \rho^{\zeta,t}(x) \vert \; \vert \langle \partial \rho^{\zeta,t}(x),Y \rangle \vert^2.
\end{equation} 
The inequalities (\ref{Iest}), (\ref{IInd}) and (\ref{IIIFest}) together with (\ref{A}) imply that
\[
\frac{\mathcal{L}_{\rho^{\zeta,t}}(x,Y) }{\vert \rho^{\zeta,t}(x) \vert} \geq -K_2 \vert Y \vert^2 - 
\frac{1}{K_2} \frac{\vert \langle \partial \rho^{\zeta,t}(x),Y \rangle \vert^2}{ \vert \rho^{\zeta,t}(x) \vert^2 } + 
\frac{\epsilon_0}{2} \frac{\mathcal{L}_{E_{\zeta,t}}(x,y)}{\vert \rho^{\zeta,t}(x) \vert}.
\]
\noindent Using the above inequality and (\ref{psiest}), and adjusting constants as in \cite{Her}, we finally get
\begin{equation}\label{Lpsiest}
\mathcal{L}_{\psi}(x,Y) \geq \frac{\eta_0 \epsilon_0}{2}\vert\psi(x)\vert  \frac{\mathcal{L}_{E_{\zeta,t}}(x,Y)}{\vert \rho^{\zeta,t}(x) \vert} \geq C_8  \vert \psi(x) \vert \frac{\mathcal{L}_{E_{\zeta,t}}(x,Y)}{J_{\zeta,t}(w)} .
\end{equation}
\noindent This estimate in conjunction with Lemma \ref{maxhess}(ii) yields the required estimate 
for the complex Hessian of $\psi$ and consequently, for $\psi_{\zeta,t}^{w}$ as well. Indeed, 
\[
\vert \psi(x) \vert = \vert \rho^{\zeta,t}(x) \vert^{\eta_0} e^{-K \eta_0 \vert x \vert^2}.
\]
So that, $\vert \psi(x) \vert \gtrsim \vert J_{\zeta,t}(w) \vert^{\eta_0}$ by using (\ref{roJapprox}). Therefore, we get from (\ref{Lpsiest}) that
\begin{align*} \label{psizetest}
\mathcal{L}_{\psi_{\zeta,t}^w}(x,Y) = \frac{\mathcal{L}_\psi(x,Y)}{\big( J_{\zeta,t}(w) \big)^{\eta_0}} 
&\geq C_8 \frac{\vert \psi(x) \vert}{\big( J_{\zeta,t}(w) \big)^{\eta_0}} \frac{\mathcal{L}_{E_{\zeta,t}}(x,Y)}{J_{\zeta,t}(w)} \\
&\geq C_8 \frac{\mathcal{L}_{E_{\zeta,t}}(x,Y)}{\big(J_{\zeta,t}(w)\big)^{\eta_0}} \\
&\geq C_{10} \frac{J_{\zeta,t}(x)}{J_{\zeta,t}(w)} \left( \frac{\vert Y_1 \vert^2}{\big( \tau(\zeta, J_{\zeta,t}(w) \big)^2} + 
\sum\limits_{k=2}^{n-1} \frac{\vert Y_k \vert^2}{J_{\zeta,t}(w)} +  \frac{\vert Y_n \vert^2}{\big( J_{\zeta,t}(w) \big)^2} \right)\\
&\geq C_{10}(1-C_6 \theta) \left( \frac{\vert Y_1 \vert^2}{\big( \tau(\zeta, J_{\zeta,t}(w) \big)^2} + 
\sum\limits_{k=2}^{n-1} \frac{\vert Y_k \vert^2}{J_{\zeta,t}(w)} +  \frac{\vert Y_n \vert^2}{\big( J_{\zeta,t}(w) \big)^2} \right).
\end{align*}
This proves the lemma.
\end{proof}

\subsection{Separation properties of the pluri-complex Green's function.} \label{third}
\noindent  Another key ingredient for constructing the weight functions for setting up the appropriate $\overline{\partial}$-problem 
 is the pluri-complex Green's function introduced by Klimek \cite{K}. For any domain $\Omega$, these are given by
\[
\mathcal{G}_{\Omega}(z,w) = \sup \big\{ u(z) \; : \; u \in PSH_w(\Omega) \big \}.
\]
Here for $w \in \Omega$, $PSH_w(\Omega)$ denotes the family of all plurisubharmonic functions that are 
negative on $\Omega$ and which have the property that the function $u(z) - \log \vert z - w \vert$ is 
bounded from above near $w$. The Green's function itself is again a member of $PSH_w(\Omega)$. The following lemma 
provides a link between the separation properties of the sub-level sets of the pluri-complex Green's function 
associated with the bumped domain $D_t^\zeta$ and the polydiscs $P(w,t, \theta)$. Before going further, note 
that the polydiscs $P(w,t,\theta)$ are balls in the metric defined by 
\[
\hat{V}_w(z) = \Big \vert \Delta_{\zeta}^{\theta J_{\zeta,t}(w)}(z - w) \Big \vert =  
 \left( \frac{\vert z_1 - w_1 \vert^2 }{ \tau(\zeta, \theta J_{\zeta,t}(w))^2} + 
\sum\limits_{\alpha=2}^{n-1} \frac{\vert z_\alpha - w_\alpha \vert^2}{ \theta J_{\zeta,t}(w)}  
+  \frac{\vert z_n - w_n \vert^2}{\theta J_{\zeta,t}(w)^2} \right)^{1/2}.
\]

\begin{lem} \label{Psep} \noindent
\begin{itemize}
\item[(a)] There is a bound $M_1 \gg 1$ (depending on $\theta$) such that given $\sigma \in (0,1)$ one has for all $w \in D_t^\zeta \cap D_\zeta$ that
\[
P(w,t,\sigma\theta) \supset \big \{ \mathcal{G}_{D_t^\zeta}(\cdot,w) < \log \sigma - M_1 \big\}.
\]
\item[(b)] Further, one can find $\sigma_0$ in such a way that for any two points $w, w' \in D_t^\zeta$ but with $w' \not \in P(w,t,\theta)$ one has
\[
P(w',t,\sigma_0 \theta) \cap P(w,t,\sigma_0 \theta) = \phi.
\]
In particular, we have
\[
\big \{ \mathcal{G}_{D_t^\zeta} (\cdot,w') < \log \sigma_0 - M_1 \big\} \cap \{ \mathcal{G}_{D_t^\zeta} (\cdot,w) < \log \sigma_0 - M_1 \big\} = \phi
\]
for all $w,w' \in D_\zeta$.
\end{itemize}
\end{lem}
\begin{proof}
Let $\xi: \mathbb{R} \to \mathbb{R}$ be a smooth increasing function with $\xi(s)=s$ for all $s \leq 1/4$ and with $\xi(s)=3/4$ if $s \geq 7/8$. Thereafter, let $V_w(z) = (\hat{V}_w(z))^2$ and choose a convex increasing function $\chi$ on $\mathbb{R}$ satisfying $\chi(s)=-7/4$ for $s 
\leq -2$ and $\chi(s) = s$ if $s \geq -3/2$. Then, for a large enough $M'$ (depending on $\theta$),  the function
\[
\phi_w = 1/2\; \log \xi \circ V_w + M' \chi \circ \psi_{\zeta,t}^w
\]
becomes plurisubharmonic on $D_t^\zeta$ and hence a candidate for the supremum that defines $\mathcal{G}_{D_t^\zeta}(\cdot,w)$. We have 
\[
\phi_w \geq 1/2 \; \log \xi \circ V_w - 7/4 \; M'.
\]
\noindent If now $z \in D_t^\zeta$ is a point for which $\mathcal{G}_{D_t^\zeta}(z,w) < \log \sigma - 7/4 \; M'$, then 
\[
\log \xi \circ V_w(z) \leq 2 \log \sigma,
\]
which implies $V_w(z) \leq \sigma^2$, provided $\sigma < 1/2$. But then this precisely means that $z \in P(w,t,\sigma \theta)$ for $V_w$ is the square of the metric that defines the polydisc $P$. So we may let $M_1=7/4 \; M'$, to obtain the assertion of part (a) of the lemma.\\

\noindent For part (b) of the lemma, we argue by contradiction. So, assume that (b) fails to hold for any choice of $\sigma_0$ and pick $w' \not \in P(w,t, \sigma_0 \theta)$ where $\sigma_0$ will be chosen appropriately in a moment, to contradict our assumption. Now $w' \not \in P(w,t, \sigma_0 \theta)$ means that atleast one of the following holds:
\begin{itemize}
\item[(1)] $\vert w_1 -w'_1 \vert > \sigma_0  \tau \big( \zeta,\theta J_{\zeta,t}(w) \big)$ or, \\
\item[($\alpha$)] $\vert w_\alpha - w'_\alpha \vert > \sigma_0 \sqrt{\theta  J_{\zeta,t}(w)} $  for $2 \leq \alpha \leq n-1$ or, \\
\item[(n)]  $\vert w_n - w'_n \vert > \sigma_0 \theta J_{\zeta,t}(w)$.
\end{itemize}

\noindent Suppose that the first holds and assume to obtain a contradiction that $x \in P(w', t, \sigma_0 \theta) \cap P(w,t, \sigma_0 \theta)$. Then a combined application of the facts that $J_{\zeta,t}(x) \approx J_{\zeta,t}(w) \approx J_{\zeta,t}(w')$ and that $\tau(w,\delta) \approx \tau(w', \delta)$ gives
\[
\vert w_1 - w'_1 \vert \leq \vert w_1 - x_1 \vert + \vert w'_1 - x_1 \vert \leq  \sigma_0 C_{11}  \tau(\zeta, \theta J_{\zeta,t}(w))
\]
with some constant $C_{11}$ that does not depend on $\sigma_0, w,w'$ and $\theta$. Taking $\sigma_0$ small enough furnishes a contradiction. Similar arguments take care of the remaining cases as well.
\end{proof}

\subsection{Localization lemmas} \label{fourth}
\begin{lem} \label{lem5.1}
There exists $L>0$ with the following property
\begin{itemize} 
\item[(a)] If $t$ is small and $w' \in P(w,t, \theta/2) \cap D_\zeta $ and $f \in H^\infty \big( P(w,t,\theta) \big)$. Then there exists $\tilde{f} \in H^2(D_t^\zeta)$ such that $\tilde{f}(w) = f(w)$ and $\tilde{f}(w') = f(w')$. Moreover,
\[
\parallel \tilde{f} \parallel_{L^2(D_t^\zeta)} \leq L \; J_{\zeta,t}(w) \Big( \sqrt{J_{\zeta,t}(w)} \Big)^{n-2} \tau \big(\zeta,J_{\zeta,t}(w) \big) \parallel f \parallel_{L^\infty}.
\]
\item[(b)] Suppose $w' \not \in P(w,t,\theta/2) \cap D_\zeta$ and $f \in H^\infty \big( P(w,t,\theta) \big)$. Then there exists $\tilde{f} \in H^2(D_t^\zeta)$ such that  $\tilde{f}(w) = f(w)$ and $\tilde{f}(w')=0$ and 
\[
\parallel \tilde{f} \parallel_{L^2(D_t^\zeta)} \leq L \; J_{\zeta,t}(w) \Big( \sqrt{J_{\zeta,t}(w)} \Big)^{n-2} \tau \big(\zeta,J_{\zeta,t}(w) \big) \parallel f \parallel_{L^\infty}.
\]
\end{itemize}
\end{lem}

\begin{proof}
Let $f$ be a function from $H^\infty \big( P(w,t, \theta) \big)$. For part (a), we choose a non-negative cut-off function $\xi \in C^\infty(\mathbb{R})$ such that $\xi(s)=1$ for $s \leq 1/3$ and $\xi(s)=0$, if $s \geq 7/8$. Now, define 
\begin{equation}
v = \overline{\partial} \Big[ \xi \big( \frac{ \vert z_1 - w_1 \vert^2 }{ \tau(\zeta, \theta J_{\zeta,t}(w))^2} \big) \xi \big( \frac{\vert z_2 - w_2 \vert^2 }{ \theta^2 J_{\zeta,t}(w) }  \big) \ldots \xi \big( \frac{\vert z_{n-1} - w_{n-1} \vert^2 }{ \theta^2 J_{\zeta,t}(w) }  \big) 
\xi \big( \frac{ \vert z_{n} - w_{n} \vert^2 }{ \theta^2 \big( J_{\zeta,t}(w) \big)^2 } \big)  \; f \Big]
\end{equation}
which is a smooth $\overline{\partial}$-closed $(0,1)$-form on $D_t^\zeta$ with support
\[
\text{supp}(v) \subset P(w,t,\theta) \setminus P(w,t,\theta/\sqrt{3}).
\]
By Lemma \ref{Psep} we have
\[
\mathcal{G}_{D_t^\zeta}(\cdot,w) \geq \log(1/\sqrt{3}) - M_1
\]
on ${\rm supp}(v)$ and then since  $w'$ lies in $P(w,t, \theta/2)$ we similarly have
\[
\mathcal{G}_{D_t^\zeta}(\cdot,w') \geq \log(1/\sqrt{3} - 1/2) \; - \; \log(1+2C_{12}) - M_1
\]
on ${\rm supp}(v)$, because $P\big( ( w' , t , \frac{1/\sqrt{3}-1/2}{1+2C_{12}} \theta) \big) \subset P(w,t, \theta/\sqrt{3} )$.
\noindent The plurisubharmonic function 
\[
\Phi = \psi_{\zeta,t}^w + 4 \mathcal{G}_{D_t^\zeta}(\cdot,w) + 4 \mathcal{G}_{D_t^\zeta}(\cdot,w')
\]
is bounded from below on ${\rm supp}(v)$ by some constant $-T<0$. From theorem 5 of \cite{Her} we obtain a solution $u \in C^\infty(D_t^\zeta)$ to the equation $\overline{\partial} u =v$ such that 
\begin{equation} \label{dbarest}
\int_{D_t^\zeta} \vert u \vert^2 e^{-\Phi} d^{2n} z \leq 2 \int_{D_t^\zeta} \vert v \vert^2_{\partial \overline{\partial} \psi_{\zeta,t}^w} e^{-\Phi} d^{2n}z
\end{equation}
where $d^{2n}z$ denotes the standard Lesbegue measure on $\mathbb{C}^n$ identified with $\mathbb{R}^{2n}$.
Now, by Lemma \ref{Greencandidate}, it follows that $\vert v \vert^2_{\partial \overline{\partial} \psi_{\zeta,t}^w} \leq L_1\vert f \vert_{L^\infty} $ for some unimportant constant $L_1$. Indeed, note first by the holomorphicity of $f$ that we have
\[
v(z_1, \ldots, z_n) = \overline{\partial} \big( W_1 \ldots W_n \big) f
\]
where 
\[
W_j =  \xi \Big( \big\vert \frac{ z_j - w_j  }{ \tau_j( \zeta,\theta J_{\zeta,t}(w) )^2} \big\vert^2 \Big).
\]
Let \[
\hat{W}_j = \Pi_{k=1, k \neq j}^{n} W_k /W_j.
\]
Then observe that $\vert \hat{W_j} \vert \leq 1$. Choose some constant $C$ with 
\[
\big\vert \frac{d \xi}{ds}(s) \big\vert^2 \leq C
\]
for $s \in [1/3,7/8]$. Note next that the $(0,1)$-form $v$ can be written as 
\[
v= f \sum\limits_{j=1}^{n} \hat{W}_j \frac{\partial W_j}{\partial \overline{z}_j} d \overline{z}_j
\]
where 
\[
\frac{\partial W_n}{\partial \overline{z}_n}= \frac{\partial}{\partial \overline{z}_n} \big( W_n (z) \big)^2 = \frac{d \xi}{ds}\big( W_n(z) \big) \frac{z_n - w_n}{ \big( \theta J_{\zeta,t}(w) \big)^2}
\]
So
\[
\Big \vert \frac{ \partial W_n }{ \partial \overline{z}_n } \Big \vert^2 
\leq C \Big \vert  \frac{ z_n - w_n }{ \big( \theta J_{\zeta,t}(w) \big)^2} \Big \vert^2 \lesssim \frac{1}{\vert \theta J_{\zeta,t}(w) \vert^2}  
\]
where the second inequality follows from $\vert z_n - w_n \vert \leq \big( \theta J_{\zeta,t}(w) \big)^2$ since $z \in P(w,t,\theta)$ which in turn comes from the fact ${\rm supp}(v) \subset P(w,t,\theta)$. Similarly, we have for each $2 \leq \alpha \leq n-1$ that
\[
\Big \vert  \frac{ \partial W_\alpha }{ \partial \overline{z}_\alpha } \Big \vert^2 \lesssim \Big \vert \frac{ z_ \alpha - v_\alpha}{ \big( \sqrt{\theta J_{\zeta,t}(w)} \big)^2} \Big \vert^2 \lesssim \frac{1}{\theta J_{\zeta,t}(w)}
\]
and finally 
\[
\Big \vert \frac{ \partial W_1 }{ \partial \overline{z}_1 } \Big \vert^2 \leq C \Big \vert  \frac{ z_n - w_n }
{ \big( \tau \big(w, \theta J_{\zeta,t}(w)\big)  \big)^2 } \Big \vert^2 \lesssim \frac{ \tau \big(w, \theta J_{\zeta,t}(w)  \big)^2  }{ \tau \big(w, \theta J_{\zeta,t}(w)  \big)^4} = \frac{1}{\tau \big(w, \theta J_{\zeta,t}(w)  \big)^2}
\]
as well, so that an application of part (ii) of Lemma \ref{Greencandidate} to the form $\vert v \vert^2_{\partial \overline{\partial} \psi_{\zeta,t}^w}$ (which will read as an upper bound) evaluated at the pair 
\[
(x,Y)= \Big(  z,  f(z) \big(\hat{W}_1(z), \ldots, \hat{W}_n(z)\big)  \Big)
\]
actually cancels the scaling factors in the definition of the form $\vert v \vert^2_{\partial \overline{\partial} \psi_{\zeta,t}^w}$ and yields our afore-mentioned claim that   
$ \vert v \vert^2_{\partial \overline{\partial} \psi_{\zeta,t}^w} \lesssim \vert f \vert_{L^\infty}$.
This subsequently extends the estimate (\ref{dbarest}) as
\begin{align*}
\int_{D_t^\zeta} \vert u \vert^2 d^{2n}z &\leq \int_{D_t^\zeta} \vert u \vert^2 e^{-\Phi} d^{2n}z\\
&\leq 2 L_1 e^T \text{Vol(supp $v$) } \vert f \vert_{L^\infty}^2 \\
&\leq 2 L_1 e^T J_{\zeta,t}(w)^2 \Big( \sqrt{J_{\zeta,t}(w)} \Big)^{2(n-2)} \tau \big( \zeta, J_{\zeta,t}(w) \big)^2 \vert f \vert_{L^\infty}^2.
\end{align*}
The function
\[
\tilde{f}(z) = \Big[ \xi \big( \frac{\vert z_1 - w_1 \vert^2 }{ \tau(\zeta,\theta J_{\zeta,t}(w))^2} \big) \xi \big( \frac{\vert z_2 - w_2 \vert^2 }{\theta^2 J_{\zeta,t}(w)}  \big) \ldots \xi \big( \frac{\vert z_{n-1} - w_{n-1} \vert^2 }{\theta^2 J_{\zeta,t}(w)}  \big) 
\xi \big( \frac{\vert z_{n} - w_{n} \vert^2 }{\theta^2 \big(J_{\zeta,t}(w) \big)^2 }\big) \Big] \cdot f \; - \; u(z)
\]
now becomes holomorphic and has the desired properties, as  $u(w)=u(w')=0$ while the $L^2$-estimate for $\tilde{f}$ follows immediately from that for the function $u$.\\

\noindent Next we do part (b) in a manner similar (a): Firstly, if $w' \not \in P(w,t, \theta/2)$ then with a number $\sigma_0 >0$ (independent of $w,w',t$) we have
\[
P(w,t,\sigma_0 \theta/2) \cap P(w',t, \sigma_0 \theta/2) = \phi.
\]
We put
\[
f_1(z) = \Xi^{\theta,t}_\zeta (z,w) \cdot f - u(z)
\]
where $\Xi^{\theta,t}_\zeta (z,w) $ denotes the product
\[
\xi \big( \frac{\vert z_1 - w_1 \vert^2 }{(\sigma_0 \theta/2)^2 \tau(\zeta, J_{\zeta,t}(w))^2} \big) \xi \big( \frac{\vert z_2 - w_2 \vert^2 }{(\sigma_0 \theta/2)^2 J_{\zeta,t}(w)}  \big) \ldots \xi \big( \frac{\vert z_{n-1} - w_{n-1} \vert^2 }{(\sigma_0 \theta/2)^2 J_{\zeta,t}(w)}  \big) 
\xi \big( \frac{\vert z_{n} - w_{n} \vert^2 }{(\sigma_0 \theta/2)^2 \big( J_{\zeta,t}(w)\big)^2 } \big) 
\]
and solve  $\overline{\partial} u =v$ with $v= \overline{\partial}(f_1f)$. Now, 
\[
{\rm supp}(v) \subset P(w,t, \sigma_0 \theta/2) \setminus P \big(w,t, (\sigma_0 \theta/2 \sqrt{3}) \big).
\]
Then, just as in part (a), we see that 
\[
\mathcal{G}_{D_t^\zeta}(\cdot,w) \geq \log \big( \sigma_0/(2 \sqrt{3}) \big)  - M_1
\]
on ${\rm supp}(v)$. Furthermore, ${\rm supp}(v)$ is disjoint from $P(w',t, \sigma_0 \theta/2)$ and so 
\[
\mathcal{G}_{D_t^\zeta}(\cdot,w') \geq \log \big( \sigma_0/(2 \sqrt{3}) \big) - M_1 -\log(1+2C_{12}).
\]
Proceeding exactly as in part (a) from here on, we arrive at the statement in (b) of our Lemma.
\end{proof}

\begin{lem} \label{lem5.2}
Let $x \in D$ and denote by $\zeta$ the point closest to $x \in \partial D$. Let $y$ be a point such that $w_y= \Phi^\zeta(y) \in D_t^\zeta $. Let $w_x= \Phi^\zeta(x)$. Then, given  $f \in H^\infty\big( P(w_x,t, \theta) \big)$ it is possible to find a function $\hat{f} \in H^\infty(D)$ such that $\parallel \hat{f} \parallel_{L^\infty} \leq L^* \parallel f \parallel_{L^\infty}$ and $\hat{f}(x) = f(w_x)$. Further, $\hat{f}(y) = f(w_y)$ in case $w_y \in P(w_x,t, \theta/2)$ and $\hat{f}(y) =0$ if $w_y \not \in P(w_x,t, \theta/2)$.
\end{lem}

\begin{proof}
We apply Lemma \ref{lem5.1} to the pair $w=w_x$ and $w'=w_y$ and the function $f$. This yields a function $\tilde{f} \in H^2(D_t^\zeta)$ with $\tilde{f}(w_x) = f(w_x)$ and $\tilde{f}(y) = f(w_y)$ in case $w_y \in P(w_x,t,\theta/2)$ and $\tilde{f}(y) =0$, if $w_y \not \in P(w_x,t,\theta/2)$. It satisfies the $L^2$-estimate 
\begin{equation} \label{eq5.1}
\vert \tilde{f} \vert_{L^2 \big({D_t^\zeta}\big)} \leq L J_{\zeta,t}(w_x) \tau(\zeta, J_{\zeta,t}(w_x)) \leq L' t \tau( \zeta,t) \vert f \vert_{L^\infty}.
\end{equation}
Let $\lambda \geq 0$ be a smooth function on $\mathbb{R}$ such that $\lambda(x) = 1$ for $x \leq (3/4)^2$ and $\lambda(x) =0 $ for $x \geq (7/8)^2$. Then there exists $\delta_0$ (independent of $\zeta,t,x,y$ and $f$) such that
\[
\hat{v} =
\Bigg \{
	\begin{array}{l l}
		\overline{\partial} \Big( \lambda(\frac{\vert \Phi^\zeta(z) \vert^2 }{r_0^2} \cdot \tilde{f} \circ \Phi^\zeta \Big) & \mbox{ ; on } (\Phi^\zeta)^{-1}(D_t^\zeta) \cap \{ r< \delta_0 \} \\
		0  & \mbox{ ; on } \{ r< \delta_0 \} \cap \big( \{ \vert \Phi^\zeta \vert \geq 7 r_0/8 \} \cup \{ \vert \Phi^\zeta \vert \leq 3r_0/4 \} \big)
	\end{array}
\]
defines a smooth $\overline{\partial}$-closed $(0,1)$-form on $\{r < \delta_0 \}$. This follows from the fact that 
\[
(\Phi^\zeta)^{-1}(D_t^\zeta) \cap \Big( \mathbb{B}(\zeta, 7/8 r_0) \setminus \mathbb{B}(\zeta,3/4 r_0) \Big) \supset  \{r < \delta_0 \} \cap \Big( \mathbb{B}(\zeta,7/8 r_0) \setminus \mathbb{B}(\zeta,3/4 r_0) \Big)
\]
where $r_0$ is the radius from Lemma \ref{Ppolydisc}.\\

\noindent Consulting now the discussion on the existence of a Stein-neighbourhood basis in  \cite{S} and as noted prior to Lemma \ref{Greencandidate}, we ascertain for ourselves the possibility of being able to choose a Stein neighbourhood $\Omega$ of $\overline{D}$ such that 
\[
\overline{D} \subset \{ r< \delta_2 \} \subset \Omega \subset \{r < \delta_0 \}.
\]
On $\Omega$, using results from \cite{H} -- where we work with the elementary weight $4\log \vert \cdot -x \vert +4 \log\vert \cdot -y \vert$ -- we can solve the equation $\overline{\partial} \hat{u} = \hat{v}$ with a smooth function $\hat{u}$ such that $\hat{u}(x) = \hat{u}(y)=0$ and
\begin{equation} \label{eq5.2}
\vert u \vert_{L^2(\Omega)} \leq C_{13} \vert \tilde{f} \circ \Phi^\zeta \vert_{L^2(\{r < \delta_0\}} \leq C_{14} \vert \tilde{f} \vert_{L^2(D_t^\zeta)}
\end{equation}
for  some positive constants $C_{13},C_{14}$. Then certainly the function
\[
\hat{f} = \lambda( \frac{\vert \Phi^\zeta(z) \vert^2}{r_0^2} ) \cdot \tilde{f} \circ \Phi^\zeta - \hat{u}
\]
is holomorphic on $D$. We need to estimate the $L^\infty$-norm of $\hat{f}$.\\

\noindent Let $z \in D$ and suppose first that $w_z = \Phi^\zeta(z) \in \mathbb{B}(0,r_0) \cap \{ \rho^\zeta <0 \}$. Then we know by part $(a)$ Lemma \ref{Ppolydisc} that 
\[
P(w_z,t,\theta) \subset D_t^\zeta.
\]
This gives 
\begin{align}
\vert f \vert &= \vert \hat{f} \circ (\Phi^\zeta)^{-1}(w_z) \vert \nonumber\\
& \leq \frac{\vert  \tilde{f} \circ \Phi^\zeta \vert_{L^2(D_t^\zeta)} }{\pi \tau(\zeta, J_{\zeta,t}(w_z)) \big(\sqrt{J_{\zeta,t}(w_z)}\big)^{n-2} J_{\zeta,t}(w_z) }\label{elest} \\
& \leq \frac{\vert \tilde{f} \vert_{L^2(D_t^\zeta)} +  \vert  u \circ \Phi^\zeta \vert_{L^2(D_t^\zeta)} }{\pi \tau(\zeta, J_{\zeta,t}(w_z)) \big(\sqrt{J_{\zeta,t}(w_z)}\big)^{n-2} J_{\zeta,t}(w_z) } \nonumber \\
& \leq C_{15} \frac{\vert \tilde{f} \vert_{L^2(D_t^\zeta)}}{\pi \tau(\zeta, J_{\zeta,t}(w_z)) \big(\sqrt{J_{\zeta,t}(w_z)}\big)^{n-2} J_{\zeta,t}(w_z)}\nonumber\\
& \leq C_{15} L \vert f \vert_{L^\infty}. \nonumber
\end{align}
The last estimate comes from (\ref{eq5.1}) in conjunction with $t \leq J_{\zeta,t}(w_z), \tau(\zeta,t) \leq \tau \big(\zeta, J_{\zeta,t}(w_z) \big)$.
Now assume that $w_z \in \{ \rho^\zeta<0 \} \cap \{ \vert \Phi^\zeta \vert \geq 0.9 r_0 \}$. Then we see that $\vert \hat{f}(z) \vert = \vert \hat{u}(z) \vert$. But $\hat{u}$ is defined on $\{ r< \delta_2\}$ and holomorphic on $\{r < \delta_2 \} \cap \{ \Phi^\zeta \vert \geq 0.9 r_0 \}$. After possibly shrinking $\delta_2 >0$ we find using the mean value inequality that 
\[
\vert \hat{u}(z) \vert \leq \delta_2^{-n} \vert  \hat{u} \vert_{L^2 (\{ r<\delta_2 \})} \leq L^* \vert f \vert_{L^\infty}
\]
because of (\ref{eq5.1}) and (\ref{eq5.2}), proving the desired $L^\infty$-estimate for $\hat{f}$.
\end{proof}
\noindent Next, as in \cite{Her} again we have the following separation of points Lemma and given the foregoing lemmas, the proof is verbatim as in \cite{Her} and this time we shall not repeat it.
\begin{lem} \label{cardist}
There is a uniform constant $c_0>0$ such that for any $a,b \in D \cap U$ such that if $b \not\in Q_{2 \delta_D(a)}(\zeta)$ where $\zeta=a^*$, one has
\[
d_D^c(a,b) \geq c_0.
\]
\end{lem}

\subsection{Estimation of the inner Caratheodory distance from below} \label{fifth}
Let $U$, a tubular neighbourhood of $\partial D$ be as before, $U'$ a relatively compact neighbourhood of $\partial D$ inside $U$. We intend to estimate $d_D^c(A,B)$ for two points $A,B \in U'$. We split the procedure into two cases
\begin{align}
d'(B,A) &> 4 C_e \delta_D(A) \label{far} \\
d'(B,A) &\leq 4 C_e \delta_D(A) \label{near}
\end{align}

Before we begin, we put down $2$ elementary inequalities which will be of recurrent use in the sequel. The first is the following simple 
\begin{equation} \label{CS}
\vert z_1 + \ldots + z_N \vert^2 \leq N \big( \vert z_1 \vert^2 + \ldots \vert z_N \vert^2 \big)
\end{equation} 
where $N \in \mathbb{N}$ for any set of complex numbers $\{ z_j \}_{j=1}^{N}$. Second, is the following logarithmic inequality
\begin{equation} \label{log}
\log(1 +r x) \geq r \log(1+x)
\end{equation}
where $x$ and $r$ are positive reals with $r<1$.\\

\noindent We now begin with case (\ref{far}).

\begin{lem}\label{lemfar}
Assume that (\ref{far}) holds. Then with some universal constant $C_* >0$ we have
\[
d_D^c(A,B) \geq C_* \log \Big( 1 +  \frac{d(B,A)}{\delta_D(A)}  \Big).
\]
\end{lem}

\begin{proof}
Let $A,B \in D \cap U'$ be points such that $d'(B,A) \geq 4 C_e \delta_D(A)$. Choose a smooth path $\gamma: [0,1] \to D$ from $A$ to $B$ satisfying 
\[
2 d_D^c(A,B) \geq L^{c}(\gamma)
\]
where $L^c(\gamma)$ refers to the length of $\gamma$ in the inner Caratheodory metric.\\

\noindent We shall split again into two cases:
\item[(i)] If $\gamma([0,1]) \not \subset U \cap D$, then we can find an exit time, i.e., 
a number $t_1' \in (0,1)$ with $\gamma([0,t'_1)) \subset D \cap U$ and $\gamma(t'_1) \in \partial U \cap D$. Now we apply Cho's lower bound on the differential metric as in \cite{Cho1} and find
\begin{align}
L^{c}(\gamma) &\geq L^{c}(\gamma_{\vert_{[0,t'_1]}}) \nonumber \\
& \geq C_1 \log \frac{\delta_D(\partial U)}{\delta_D(A)}. \label{gamlb} 
\end{align}
Indeed, the lower bound from \cite{Cho1} reads
\[
C_D(z,X) \gtrsim \frac{\vert \langle L_1(z), X \rangle \vert}{\tau \big( z, \delta_D(z) \big)} + \sum\limits_{\alpha=2}^{n-1} \frac{\vert \langle L_\alpha(z), X \rangle \vert}{ \sqrt{\delta_D(z)} } + \frac{\vert X_n \vert}{ \delta_D(z)} 
\]
-- recall that $L_n \equiv 1$; here we shall let $z$ vary in a small neighbourhood of $A^*= \pi(A)$, assumed to be the origin after a translation, on which such an estimate is guaranteed by \cite{Cho1}. Further we may also assume after a rotation that $\nu \big( A^* \big)= L_n = ('0,1)$; at this point we may also want to note that the hypothesis on $d'(A,B)$, of the case under consideration remains intact, since these transformations preserve $d'$ in the sense that they transform the $d'$ associated with the initial domain $D$ into the $d'$ of the transformed domain. In particular, we have for $z$ in a small ball $\mathbb{B}_\delta$ and $X \in \mathbb{C}^n$ the estimate
\[
C_D(z,X) \gtrsim \frac{\vert X_n \vert}{ \delta_D(z)}.
\] 
which contains in it the rate of blow-up of the Caratheodory metric along the normal direction, since $\nu (\pi(z))$ must have a non-zero component along $L_n$. To unravel this information precisely from the above inequality and in a more useful form for our purposes, we need to restrict ourselves to the cone
\[
\mathcal{C}^z_\alpha = \big\{ X \in T_{\pi(z)}(\mathbb{C}^n) \; : \; \big\vert \langle \nu \big(\pi(z)\big), X \rangle \big\vert > \alpha  \big\vert \nu \big(\pi(z)\big) \big\vert  \big\}
\]
where $\alpha \in (0,1]$. Let 
\begin{equation*}
\mathcal{C} = \bigcup_{ z \in D } \{ z \} \times \mathcal{C}^z_\alpha  
\end{equation*}
and consider the function defined on $\mathcal{C}$ by
\[
R(z,X) = \Big \vert \frac{\vert X_n \vert}{ \vert \langle \nu\big( \pi(z) \big), X \rangle \vert}   - 1 \Big \vert
\]
which is zero for all those $z$ whose $\pi(z)$ is the origin, i.e., the $z_n$-axis. Now note that $R(z,X)$ is continuous on $\mathcal{C}$ as $\vert\langle \nu\big( \pi(z) \big), X \rangle\vert$ is bounded away from $0$ and also that $R(A^*,X) = 0$. Therefore, given any $\epsilon$ ($1/2$, say) there is a $\delta_0$ (which we may take to be $< \delta$) such that 
\[
\vert R(z,X) - R(A^*,X) \vert < \epsilon
\]
for all $z \in \mathbb{B}_{\delta_0}$ which is to say, we have
\[
\Big \vert \frac{X_n}{\langle \nu\big( \pi(z) \big), X \rangle} \Big \vert > 1- \epsilon =1/2.
\]
or equivalently that 
\begin{equation} \label{X_neqn}
\vert X_n \vert > 1/2 \; \big \vert \langle \nu\big( \pi(z) \big), X \rangle \big\vert.
\end{equation}
Now getting to our setting, since the curve $\gamma(t)$ moves away from the boundary during the interval $I=[0,t_1')$ meaning, ${\rm dist} \big((\gamma(t_1'), \partial D)\big) = {\rm dist}(\partial U, \partial D) = \delta_D(U)$ is greater than $\gamma(0)=A \in U$, we must have that the `normal component' of the curve viz., $\langle \dot{\gamma}(t), \nu \big( \pi( \gamma(t)) \big) \rangle$, must be non-zero (bigger than some $\alpha >0$) for some non-trivial stretch of time, i.e., for a sub-interval of $I$ of non-zero length -- call this sub-interval $I$ again -- so that we may apply the fore-going considerations, in particular (\ref{X_neqn}) to pass to a further sub-interval of $I$ of non-zero length if necessary where 
\[
\vert \dot{\gamma}_n(t) \vert \geq 1/2 \; \vert \langle \dot{\gamma}(t), \nu \big( \pi( \gamma(t)) \big) \rangle \vert.
\]
Indeed as mentioned above, the existence of such an interval follows just from continuity and the fact that $\nu \big( \pi( \gamma(t)) \big)$ at $t=0$ is just $({}'0,1)$ so that 
\[
\frac{\vert \dot{\gamma}_n (0) \vert}{\vert \langle \dot{\gamma}(0), \nu \big( \pi( \gamma(0)) \big) \rangle \vert} =1
\]
so that we may take this sub-interval, which we shall denote again by $I$, to be of the form $[0,t_2')$. We then have on this sub-interval that 
\begin{align*}
C_D(\gamma(t), \dot{\gamma}(t)) &\gtrsim \frac{\vert \dot{\gamma}_n(t) \vert}{ \delta_D(\gamma(t))} \\
& \gtrsim \frac{\vert \langle \dot{\gamma}(t), \nu \big( \pi( \gamma(t)) \big) \rangle \vert}{\delta_D(\gamma(t)) } \\
& \gtrsim \frac{\vert \langle \dot{\gamma}(t), \nu \big( \pi( \gamma(t)) \big) \rangle \vert}{\vert r(\gamma(t)) \vert } \\
\end{align*}
We have elaborately presented the steps that lead to this lower bound because of its re-occurrence later in a more complicated setting. Integrating the final estimate in the above with respect to $t$, leads to (\ref{gamlb}) which subsequently yields, 
\begin{align*}
L^{c}(\gamma) & \geq C_1 \log \frac{\delta_D(\partial U)}{\delta_D(A)} \\
& \geq \frac{1}{2} C_1 \log \Big( 1+ \frac{\delta_D(\partial U)}{\delta_D(A)} \Big)\\
&\hspace{0.2in} \\
&\geq \frac{1}{2} \log \Big( 1 + \frac{\delta_D(\partial U)}{\text{diam}(D)} \Big) \frac{d(B,A)}{\delta_D(A)}  \\
&\geq C_2 \Big( 1 + \frac{d(B,A)}{\delta_D(A)} \Big) 
\end{align*}
where
\[
C_2 = \frac{1}{2}C_1 \frac{\delta_D(\partial U)}{\text{diam}(D)}.
\]
The second inequality in the foregoing string of inequalities, can be ensured by choosing $U_0$ such that  $\delta_D(A) \leq 1/2 \; \delta_D(\partial U)$, while the third one follows  because $d(B,A) \leq {\rm diam}(D)$. This finishes case (i).\\

\noindent The other sub-case is (ii): $\gamma([0,1]) \subset U \cap D$.\\
Since $d'(B,A) > 4 C_e \delta_D(A)$ we have 
\[
4C_e \delta_D(A) < \text{ inf } M(B,A)
\]
and therefore 
\[
B \not \in Q_{4C_e \delta_D(A)}(A) \supset Q_{2C_e \delta_D(A)}(A^*).
\]
Thus we can choose a number $t_1 \in (0,1)$ such that 
\[
\gamma(t_1) \in \partial Q_{2C_e \delta_D(A)}(A^*).
\]
This shows that the following set is not empty
\begin{align*}
S^{A,B} := \big\{ j \in \mathbb{Z}^{+} \; : \; &\exists \; 0=t_0<t_1< \ldots < t_j <1 \\
&\text{ such that } \gamma(t_\nu) \in \partial Q_{ 2 C_e \delta_D( \gamma(t_{\nu-1}) ) }\big( \gamma(t_{\nu-1})^* \big), \; 1\leq \nu \leq j \big\}.
\end{align*}
Since $\gamma([0,1]) \subset D \cap U$, the boundary point $\gamma(t)^*$ is well-defined for any $t \in [0,1]$. As in \cite{Her} again, $S^{A,B}$ can be ascertained to be a finite set and consequently we may define the number
\[
m : = {\rm max} \;S^{A,B}
\]
and choose numbers $0=t_0<t_1 \ldots <t_m<1$ such that 
\[
\gamma(t_\nu) \in \partial Q_{2 C_e \delta_D(\gamma(t_{\nu-1})}\big( \gamma(t_{\nu-1})^* \big), \; 1\leq \nu \leq m.
\]
Further following \cite{Her}, we get $2 d_D^{c}(A,B) \geq c_0 m$ and subsequently follow the steps therein to estimate $m$ from below, which uses the pseudo-distance property of $d$ and leads eventually to the estimate
\[
d(B,A) \leq C_*^{m+2} \delta_D(A)
\]
where $C_*$ is a constant bigger than $1$ (in fact bigger than $6$). This gives
\[
C_*^{2m} \geq \frac{d(B,A)}{\delta_D(A)}.
\]
From this it follows that
\[
(1 + C_*)^{2m}> 1+C_*^{2m} \geq 1+ \frac{d(B,A)}{\delta_D(A)}
\]
which subsequently leads to
\[
2m \log(1 + C_*) > \log \Big(  1+ \frac{d(B,A)}{\delta_D(A)} \Big)
\]
which gives
\[
m \gtrsim \log \Big(  1+ \frac{d(B,A)}{\delta_D(A)} \Big)
\]
But then recalling that 
\[
c_0 \cdot m \leq 2 d_D^{c}(A,B)
\]
we finally see that 
\[
d_D^{c}(A,B) \gtrsim \log \Big(  1+ \frac{d(B,A)}{\delta_D(A)} \Big)
\]
\end{proof}
\noindent Now we turn to the other possibility:
\begin{lem}\label{lemnear}
Assume that (\ref{near}) holds. Then with some universal constant $C_{21}>0$ we have
\begin{multline*}
d_D^{Cara}(A,B) \geq C_{21} \; \log \Big(1 \;+\; \vert \Phi^A(B)_n \vert^2/ \delta_D(A)^2 \\
+\; \sum\limits_{\alpha=2}^{n-1} \vert \Phi^A(B)_\alpha \vert^2/ \delta_D(A) \;+\; \vert \Phi^A(B)_1 \vert^2/ \tau(A^*, \delta_D(A))^2 \Big)
\end{multline*}
\end{lem}

\begin{proof}
Suppose that $A,B$ are points in $D \cap U'$ which satisfy (\ref{near}). Then clearly
\[
B \in Q_{4C_e \delta_D(A) }(A)
\]
and 
\begin{align*}
\Phi^A(B) &\in \Delta\big( 0, \tau(A, 4C_e\delta_D(A)) \big) \times \Delta\big( 0, \sqrt{4C_e\delta_D(A)} \big) \times \ldots \times \Delta\big( 0, \sqrt{4C_e\delta_D(A)} \big) \times \Delta \big( 0,4C_e\delta_D(A) \big) \\
& \; \; \;\subset P(w_A,t, \theta)
\end{align*}
 with a number $\theta>0$ independent of $A,B$ and $t:= 4C_e\delta_D(A)$. Here we put $w_A = \Phi^{A^*}(A)$. According to Lemma \ref{lem5.2} applied to the point $x=A$ and $y=B$, there exists for any function $f \in H^\infty\big( P(w,t,\theta) \big)$ having norm equal to one, a function $\hat{f} \in H^\infty(D)$ with $\vert \hat{f} \vert_{L^\infty} \leq L^*$ for some constant $L^*$ such that 
\begin{align*}
 \hat{f}(A) &= f(w_A)\\
 \hat{f}(B) &= f(w_B)
\end{align*}
where $w_B = \Phi^{A^*}(B)$. This implies
\[
d_D^{Cara}(A,B) \geq d^P \Big( \frac{1}{L^*} f(w_A), \frac{1}{L^*} f(w_B) \Big).
\]
We now make our choice of the function $f$ namely, put
\begin{equation}
f(v) = f_n(v_1,v_2, \ldots,v_n) = \frac{1}{C' \; \delta_D(A)} \Big( \big( \Phi^A \circ (\Phi^{A^*})^{-1} \big) (v_1, v_2,\ldots,v_n)\Big)_n
\end{equation}
where $C'$ is a constant chosen so that $\vert f \vert_{L^\infty} \leq 1$ and is independent of $A,B$ -- to see that this can indeed be done, notice that 
\[
\Phi^A \circ (\Phi^{A^*})^{-1}(v) = \big( \Phi^A - \Phi^{A^*} \big) \circ (\Phi^{A^*})^{-1}(v) + v
\]
and hence
\begin{align*}
\vert f(v) \vert &\leq \frac{ \left\vert \big( \Phi^A - \Phi^{A^*} \big) \circ (\Phi^{A^*})^{-1}(v) \right\vert  + \vert v_n \vert }{C' \delta_D(A)}\\
& \leq C + \frac{\vert v_n \vert}{\delta_D(A)}.
\end{align*}
But then on $P(w_A,t, \theta)$ we have
\[
\vert v_n \vert \leq \vert v_n - (w_A)_n \vert + \vert (w_A)_n \vert \leq \theta J_{A^*,t}(w_A) + t \leq C'' t \leq C''' \delta_D(A).
\]
Certainly $f(w_A)=0$. Together with a basic estimate concerning the Poincar\'{e} distance $d_\Delta^p$ on the unit disc -- estimate (6.6) in \cite{Her} -- we get
\begin{align*}
d_D^{Cara}(A,B) \geq d_\Delta^p \Big(0, \frac{1}{L^*} f(w_B) \Big)
\geq \frac{1}{2} \log \Bigg( 1 + \frac{\vert \Phi^A(B)_n \vert^2  }{(L^*C')^2 (\delta(A))^2 } \Bigg) 
\end{align*}
with some suitable constant $C'>0$. Similarly choosing next the function $f$ to be 
\[
f(v) = f_\alpha(v) = \frac{1}{C' \sqrt{\delta(A)}} \big( \Phi^A \circ (\Phi^{A^*})^{-1} \big)(v)_\alpha,
\]
for all $2 \leq \alpha \leq n-1$ -- which also has $L^\infty$-norm not bigger than $1$, viewed as a function on $P(w_A,t, \theta)$ --  we also obtain (since again $f(w_A)=0$) that 
\[
d_D^{Cara}(A,B) \geq d_\Delta^p \big(0, \frac{1}{L^*}f(w_B)
 \big) \geq \frac{1}{2} \log \Bigg( 1+ \frac{\vert \Phi^A(B)_\alpha \vert^2}{(L^*C')^2 ( \sqrt{\delta(A)})^2} \Bigg)
\]
and similarly again, choosing 
\[
f(v) = f_1(v) = \frac{1}{C' \tau \big( A, \delta_D(A) \big)} \big( \Phi^A \circ (\Phi^{A^*})^{-1} \big)(v)_1
\]
with $C'$ a suitable constant adjusted so that $f_1 \in L^\infty \big( P(w_A,t, \theta) \big)$, we get 
\[
d_D^{Cara}(A,B) \geq d_\Delta^p \big(0, \frac{1}{L^*}f(w_B)
 \big) \geq \frac{1}{2} \log \Bigg( 1+ \frac{\vert \Phi^A(B)_1 \vert^2}{(L^*C')^2 \tau \big(A, \delta_D(A) \big)^2} \Bigg)
\]
To summarize then, we have for each $1 \leq j \leq n$, that
\[
e^{d_D^{Cara}(A,B)} \gtrsim 1 + c\frac{\vert \Phi^A(B)_j \vert }{\tau_j(A)}
\]
for some constant $c<1$. Adding together these inequalities over the index $j$ gives the estimate asserted in the Lemma. Let us note here for later purposes that the foregoing inequalities may also be rewritten as
\begin{equation} \label{maxbd}
d_D^{Cara}(A,B) \gtrsim \max_{1 \leq j \leq n} \log \Bigg( 1 + \frac{\vert \Phi^A(B)_j \vert }{\tau_j(A)}  \Bigg)
\end{equation}
as $c<1$.
\end{proof}

\noindent We now proceed to demonstrate that the estimates of the last two lemmas fit together well to yield Theorem \ref{thm1}.
For reasons of symmetry it is enough to verify that
\[
d_D^c(A,B) \gtrsim \eta(A,B).
\]

\noindent Before we begin, let us just record one useful fact which is the analogue of lemma 3.2 of \cite{Her} and follows by the very same line of proof therein. 

\begin{lem} \label{lem 3.2}
If $a,b \in D \cap U$ are points with $\vert a - b \vert<R_0$, we have
\begin{multline*}
\max \Big \{ \vert \big(\Phi^a(b)\big)_n\vert ,  \vert \big( \Phi^a(b) \big)_2 \vert^2 , \ldots, \vert \big( \Phi^a(b) \big)_{n-1} \vert^2 , \max\limits_{2 \leq l \leq 2m} \vert P_l(a) \vert \vert \big( \Phi_a(b) \big)_1 \vert^l \Big\} \leq d'(b,a)\\
\leq 2 \; \max \Big \{ \vert \big(\Phi^a(b)\big)_n\vert ,  \vert \big( \Phi^a(b) \big)_2 \vert^2  , \ldots, \vert \big( \Phi^a(b) \big)_{n-1}\vert^2 , \max\limits_{2 \leq l \leq 2m} \vert P_l(a) \vert \vert \big( \Phi_a(b) \big)_1 \vert^l \Big\}
\end{multline*}
\end{lem}

\noindent Now suppose (\ref{far}) holds. Then we claim that for some $c_1>0$ we have
\begin{multline} \label{eqn6.7}
d(B,A)/\delta_D(A) \geq C_{12} \Big( \vert \Phi^A (B)_n \vert / \delta_D(A)  \\
+ \; \sum\limits_{\alpha=2}^{n-1} \vert \Phi^A(B)_\alpha \vert / \big( \delta_D(A) \big)^{1/2} \;+ \; \vert \Phi^A(B)_1 \vert/ \tau(A, \delta_D(A)) \Big)
\end{multline} 

\noindent For the proof of this, let $0 < \epsilon < 2 d'(B,A)$ be a number such  that $B \in Q_\epsilon(A)$. Then
\[
\Phi^A(B) \in  \Delta \big( 0, \tau(A,\epsilon) \big) \times \Delta(0, \sqrt{\epsilon}) \times \ldots \times \Delta(0,\sqrt{\epsilon}) \times  \Delta(0, \epsilon).
\]
In particular, $\vert \Phi^A(B)_n \vert \leq \epsilon \leq 2 d'(B,A)$. But then we also know $\vert \Phi^A(B)\vert \leq c'_1  \vert A -B \vert$, which implies
\[
\vert \Phi^A(B)_n \vert \leq \text{min} \big\{ 2d'(B,A), c'_1 \vert A -B \vert \big\} \leq c'_2 d(B,A).
\]
with some constant $c'_2>1$. In particular,
\[
\frac{\vert \Phi^A(B)_n \vert}{ \delta_D(A)} \leq c'_2 \frac{d(B,A)}{\delta_D(A)}
\]
Next we estimate $\vert \Phi^A(B)_1 \vert / \tau \big(A, \delta_D(A) \big) $. Since the function $t \to t/\tau(A,t)$ is increasing and $\epsilon \geq d'(B,A) > 4 C_e \delta_D(A)$, we get 
\[
\frac{\vert \Phi^A(B)_1 \vert}{\tau \big(A, \delta_D(A) \big)} \leq \frac{\tau(A,\epsilon)}{\tau \big(A, \delta_D(A) \big)} \leq c'_3 \frac{\epsilon}{\delta_D(A)} \leq 2 c'_3 \frac{d'(B,A)}{\delta_D(A)}.
\]
\noindent Moreover since $\vert \Phi^A(B)_1\vert \leq c'_1  \vert A -B \vert$ and $\tau \big(A, \delta_D(A) \big) \geq c'_4\delta_D(A)$ we get
\[
\frac{\vert \Phi^A(B)_1\vert}{\tau \big(A, \delta_D(A) \big)} \leq c'_5 \frac{\vert A -B \vert}{\delta_D(A)}
\]
and subsequently that 
\[
\frac{\vert \Phi^A(B)_1\vert}{\tau \big(A, \delta_D(A) \big)} \leq c'_6 \frac{d(B,A)}{\delta_D(A)}.
\]
Also
\[
\frac{\vert \Phi^A(B)_\alpha \vert}{\sqrt{ \delta_D(A)}} \leq \frac{\sqrt{\epsilon}}{\sqrt{\delta_D(A)}} \lesssim \frac{\epsilon}{\delta_D(A)}
\leq \frac{2 d'(B,A)}{\delta_D(A)} < \frac{2L d(B,A)}{\delta_D(A)}.
\]
This completes the verification of the claim (\ref{eqn6.7}) and then Lemma \ref{lemfar} proves
\[
d_D^c(A,B) \geq C_* \rho(A,B)
\]
for those points $A,B$ that satisfy (\ref{far}).\\

\noindent Next we move on to the case when $A,B$ satisfy (\ref{near}). In this case we claim that for some positive constant $c_2>0$ we have
\begin{equation} \label{nearest}
\frac{d(B,A)}{\delta_D(A)} \leq c_2 \Bigg( \frac{\vert \Phi^A(B)_n \vert}{\delta_D(A)} + \sum\limits_{\alpha=2}^{n-1} \frac{\vert \Phi^A(B)_\alpha \vert}{\sqrt{\delta_D(A)}} + \frac{\vert \Phi^A(B)_1 \vert}{\tau \big( A, \delta_D(A)\big)} \Bigg)
\end{equation}

\noindent First, we note that we have $d'(B,A) \leq \epsilon$ where we now let
\[
\epsilon = 2 \; \max \Big \{ \vert \big(\Phi^A(B)\big)_n\vert ,  \vert \big( \Phi^A(B) \big)_2 \vert^2  , \ldots, \vert \big( \Phi^A(B) \big)_{n-1}\vert^2 , \max\limits_{2 \leq l \leq 2m} \vert P_l(A) \vert \vert \big( \Phi_A(B) \big)_1 \vert^l \Big\}
\]
We split into the various possible cases for the value of $\epsilon$ and deal with them one by one. First, suppose that $\epsilon = 2 \vert \Phi^A(B)_n \vert$. Then we get
\[
\frac{d(B,A)}{\delta_D(A)} \leq \frac{d'(B,A)}{\delta_D(A)} \leq \frac{\epsilon}{\delta_D(A)} = 2 \frac{\vert \Phi^A(B)_n \vert}{\delta_D(A)}.
\]
Next we look at what happens when $\epsilon$ happens to be $ 2 \text{ max } \{ \vert P_l(A) \vert \vert \Phi^A(B)_1 \vert^l \} $. In this case note that 
\begin{align*}
\frac{d(B,A)}{\delta_D(A)} &\leq  \frac{\epsilon}{\delta_D(A)} \\
&= 2 \frac{ \text{ max } \{ \vert P_l(A) \vert \vert \Phi^A(B)_1 \vert^l \} }{\delta_D(A)} \\
&\lesssim \text{ max } \Big( \frac{\vert \Phi^A(B)_1 \vert}{\tau \big( A, \delta_D(A)\big) } \Big)^l \\
&\lesssim \frac{\vert \Phi^A(B)_1 \vert}{\tau \big( A, \delta_D(A)\big) } 
\end{align*}
provided we assure ourselves that $\vert \Phi^A(B)_1 \vert/\tau \big( A, \delta_D(A)\big)   \leq C$ for some constant $C$. To see this, choose any sequence $\eta_j \to d'(B,A)$ from above. Then by definition of $d'(B,A)$ we have $B \in Q_{\eta_j}(A)$ for all $j$. This gives $\vert \Phi^A(B)_1 \vert \leq \tau(A, \eta_j)$. Letting $j \to \infty$ we find $\vert \Phi^A(B)_1 \vert \leq \tau \big( A, d'(B,A)\big) $. Then since we are in the case (\ref{near}) i.e., $d'(B,A) \leq 4 C_e \delta_D(A)$, we obtain
\[
\tau \big( A, d'(B,A)\big) \leq C' \tau \big( A, \delta_D(A)\big)
\]
Therefore, $\vert \Phi^A(B)_1 \vert \leq  C' \tau \big( A, \delta_D(A)\big)$, finishing this case.\\
If it happens that $\epsilon = 2 \vert \Phi^A(B)_\alpha \vert^2$   for some $2 \leq \alpha \leq n-1$, then similar arguments with $\tau_1(A) = \tau_1 \big( A, \delta(A) \big)$ replaced by $\tau_\alpha(A) = \sqrt{\delta}$ gives 
\[
\frac{d(B,A)}{\delta_D(A)} \leq 2 \frac{\vert \Phi^A(B)_\alpha \vert^2}{\delta_D(A)} \lesssim \frac{\vert \Phi^A(B)_\alpha \vert}{\sqrt{\delta_D(A)}}
\]
\noindent Summarizing the results of the various cases depending on the value of $\epsilon$, we thus get that for some $1 \leq j \leq n$ the inequality
\begin{equation} \label{ddeluppbd}
\frac{d(B,A)}{\delta(A)} \lesssim \frac{\vert \Phi^A(B)_j \vert}{\tau_j(A)}
\end{equation}
must hold. This will be used in the sequel.\\

\noindent Putting together what we inferred for each of the possible values that $\epsilon$ may take, we may now also assert that (\ref{nearest}) holds, from which it in-turn follows from Lemma \ref{lemnear} that 
\begin{equation}
d_D^{c}(A,B) \geq d_D^{Cara}(A,B) \geq C \log \Big( 1 + \big(d(B,A)/\delta_D(A) \big)^2 \Big).
\end{equation}

\noindent Now we get to the end result, the lower bound as stated in Theorem \ref{thm1}; but we wish to first summarise for convenience, a couple of results from our variety of estimates encountered in course of our dealings of the two cases (\ref{far}) and (\ref{near}) which will be useful in the sequel -- to be precise, parts (i) and (ii) of part (a) of the following proposition, come from the discussion of Lemmas \ref{lemfar} and \ref{lemnear} respectively. We then conclude by showing how the desired end, re-written in part (b) of this proposition, follows from its previous parts.

\begin{prop} \label{lowbd}
\noindent
\begin{itemize} 
\item[(a)] For points $A,B \in D \cap U'$, depending on whether they are far or near, as measured by the pseudo-distance $d'$, we have two cases and correspondingly various estimates as in the first two statements below:
\begin{itemize} 
\item[(i)] $d'(B,A) > 4 C_e \delta_D(A)$. Then there for some constants $K_{11}, K_{12}>0$ we have
\begin{equation*}
d_D^{c}(A,B) \geq K_{11} \; \log \Big( 1 \;+\; \big( d(B,A) / \delta_D(A) \big) \Big)
\end{equation*}
\noindent and
\begin{multline*}
d(B,A)/\delta_D(A) \geq K_{12} \Big( \vert \Phi^A (B)_n \vert / \delta_D(A)  \\
+ \; \sum\limits_{\alpha=2}^{n-1} \vert \Phi^A(B)_\alpha \vert / \big( \delta_D(A) \big)^{1/2} \;+ \; \vert \Phi^A(B)_1 \vert/ \tau(A, \delta_D(A)) \Big)
\end{multline*} 
\item[(ii)] $d'(B,A) \leq 4 C_e \delta_D(A)$. Then for some constants $K_{21}, K_{22}, K_{23}>0$ we have
\begin{multline*}
d_D^{Cara}(A,B) \geq K_{21} \; \log \Big(1 \;+\; \vert \Phi^A(B)_n \vert^2/ \delta_D(A)^2 \\
+\; \sum\limits_{\alpha=2}^{n-1} \vert \Phi^A(B)_\alpha \vert^2/ \delta_D(A) \;+\; \vert \Phi^A(B)_1 \vert^2/ \tau(A, \delta_D(A))^2 \Big)
\end{multline*}
\noindent and
\begin{multline*}
d(B,A)/\delta_D(A) \leq K_{22} \Big( \vert \Phi^A(B)_n \vert/ \delta_D(A) \\
 +\; \sum\limits_{\alpha=2}^{n-1} \vert \Phi^A(B)_\alpha \vert/ \big( \delta_D(A) \big)^{1/2} \;+\; \vert \Phi^A(B)_1 \vert/ \tau \big(A, \delta_D(A) \big) \Big)
\end{multline*}
\noindent from which it was seen to follow in the foregoing lemmas, that
\[
d_D^{c}(A,B) \geq d_D^{Cara}(A,B) \geq K_{23} \log \Big( 1 \;+\; \big( d(B,A)/\delta_D(A) \big)^2 \Big)
\]
\end{itemize}
\item[(b)] Finally, we have the lower bound valid for all $A,B \in D \cap U$ and some constant $K>0$:
\begin{equation*}
d_D^{c}(A,B) \geq K \; \log \Big(1 \;+\; \frac{\vert \Phi^A(B)_n \vert}{\delta_D(A)} \\ 
 +  \sum\limits_{\alpha=2}^{n-2} \frac{ \vert \Phi^A(B)_\alpha \vert}{\big( \delta_D(A) \big)^{1/2}} + \frac{\vert \Phi^A(B)_1 \vert}{\tau\big(A, \delta_D(A) \big)} \Big)   - \frac{K}{n} \; \log 2
\end{equation*}
obtained by combining the first two inequalities in (ii) and (i) of part(a).
\end{itemize}
\end{prop}

\begin{proof} As noted earlier, it remains only to combine the two cases to get the final inequality as stated in (b). We may assume that the constant $K_{12}<1$ in the second inequality in (i). Then an application of the inequality (\ref{log}) namely, $\log(1+rx) \geq r \log(1+x)$ valid for $0<r<1$ here, gives
\[
K_{12} \log \big(1 + Q \big) \leq \log(1+K_{12}Q) \leq \log \big( 1+ d(B,A)/\delta_D(A) \big)
\]
where $Q$ is the quantity
\[
\vert \Phi^A (B)_n \vert / \delta_D(A)  \\
+ \; \sum\limits_{\alpha=2}^{n-1} \vert \Phi^A(B)_\alpha \vert / \big( \delta_D(A) \big)^{1/2} \;+ \; \vert \Phi^A(B)_1 \vert/ \tau(A, \delta_D(A)).
\]
We therefore have 
\[
d_D^{c}(A,B) \geq  K_{11}\log \big( 1+ d(B,A)/\delta_D(A) \big) \geq 
K_{11}K_{12} \log (1 + Q),
\]
when in the first case of (a) of the proposition. In particular
\begin{equation} \label{Case1}
d_D^{c}(A,B) \gtrsim \log (1+Q) - \frac{1}{n}\log 2
\end{equation}
To deal with case (ii), let us denote by $E$ the expression
\[
\vert \Phi^A(B)_n \vert^2/ \delta_D(A)^2 \\
+\; \sum\limits_{\alpha=2}^{n-1} \vert \Phi^A(B)_\alpha \vert^2/ \delta_D(A) \;+\; \vert \Phi^A(B)_1 \vert^2/ \tau(A, \delta_D(A))^2 .
\]
\noindent Next we use the inequality (\ref{CS}) namely,
\[
\vert z_1 + \ldots + z_N \vert^2 \leq N \big( \vert z_1 \vert^2 + \ldots \vert z_N \vert^2 \big)
\]
with $N=n$, to convert the inequality in (ii) of the proposition and express it in terms of $E$ to get
\begin{align} \label{Case2}
d_D^{Cara}(A,B) \geq K_{21} \log (1 + E) &\geq K_{21} \log  (1 + Q^2/n) \nonumber \\
& \geq \frac{K_{21}}{n} \log  (1 + Q^2) \nonumber \\
& \geq \frac{K_{21}}{n} \log \frac{1}{2} (1 + Q)^2 \nonumber \\
&=  \frac{2K_{21}}{n} \log (1 + Q) - \frac{K_{21}}{n}\log 2
\end{align}
Combining (\ref{Case1}) and (\ref{Case2}) gives the final inequality of the proposition, which also gives the lower bound stated in theorem \ref{thm1}, once we take into account the permutation that we had made at the start at (\ref{nrmlfrm0}) in the introductory section.
\end{proof}

\noindent Finally, we can also get from part (b) of the last Proposition that
\begin{align}
d_D^{c}(A,B) &\gtrsim \log  \Bigg(1 + \frac{\abs{ d(B,A)}}{\delta_D(A)} + \sum_{\al = 2}^{n-2} \frac{\abs{\Phi^A(B)_{\al}}}{\sqrt{\delta_D(A)}} + \frac{\abs{\Phi^A(B)_1}}{\tau \big((A, \delta_D(A)\big)}  \Bigg) - l. \label{dclowbd}
\end{align}
for some positive constant $l$.

\noindent To see this, first let us finish the easy case namely, when we are in the case (a)(i) of the last Proposition. Then, recall from our arguments for the inequality (\ref{eqn6.7}) that we had 
\[
\frac{d(B,A)}{\delta_D(A)} \geq c \frac{\vert \Phi^A(B)_j \vert }{\tau_j(A)}
\]
for all $1 \leq j \leq n$ and for some constant $c <1$.  Now from the first inequality in (a)(i) of Proposition \ref{lowbd} we have
\begin{align*}
d_D^c(A,B) &\geq \log \Bigg( 1 + \frac{1}{2} \frac{d(B,A)}{\delta_D(A)} + \frac{1}{2} \frac{d(B,A)}{\delta_D(A)}  \Bigg)\\
&\geq \log \Bigg( 1 + \frac{1}{2} \frac{d(B,A)}{\delta_D(A)} + \frac{c}{2} \sum\limits_{j=1}^{n-1} \frac{\vert \Phi^A(B)_j \vert }{\tau_j(A)} \Bigg) \\
&\geq \log \Bigg( 1 +  \frac{d(B,A)}{\delta_D(A)} +  \sum\limits_{j=1}^{n-1} \frac{\vert \Phi^A(B)_j \vert }{\tau_j(A)} \Bigg)
\end{align*}
as required.\\

\noindent The other case to deal with is when we are in the situation of (a)(ii) Proposition \ref{lowbd}. Again to finish off the easy sub-case first, suppose that the maximum on the right hand inequality in Lemma \ref{lem 3.2} happens to be $\vert \Phi^A(B)_n \vert$; we then have the following chain of inequalities giving the claim:
\begin{align}
d_D^{c}(A,B) &\geq K \log  \Bigg(1 + \frac{\abs{\Phi^A(B)_n}}{\delta_D(A)} + \sum_{\al = 2}^{n-2} \frac{\abs{\Phi^A(B)_{\al}}}{\sqrt{\delta_D(A)}} + \frac{\abs{\Phi^A(B)_1}}{\tau \big((A, \delta_D(A)\big)}  \Bigg) - K \log n/2 \nonumber   \\
&\geq K \log  \Bigg(1 + \frac{1}{2}\frac{\abs{ d'(B,A)}}{\delta_D(A)} + \sum_{\al = 2}^{n-2} \frac{\abs{\Phi^A(B)_{\al}}}{\sqrt{\delta_D(A)}} + \frac{\abs{\Phi^A(B)_1}}{\tau \big((A, \delta_D(A)\big)}  \Bigg) - K \log n/2 \nonumber \\
&\geq  \frac{K}{2} \log  \Bigg(1 + \frac{\abs{ d'(B,A)}}{\delta_D(A)} + \sum_{\al = 2}^{n-2} \frac{\abs{\Phi^A(B)_{\al}}}{\sqrt{\delta_D(A)}} + \frac{\abs{\Phi^A(B)_1}}{\tau \big((A, \delta_D(A)\big)}  \Bigg) - K \log n/2   \nonumber \\
&\geq  \frac{K}{2} \log  \Bigg(1 + \frac{\abs{ d(B,A)}}{\delta_D(A)} + \sum_{\al = 2}^{n-2} \frac{\abs{\Phi^A(B)_{\al}}}{\sqrt{\delta_D(A)}} + \frac{\abs{\Phi^A(B)_1}}{\tau \big((A, \delta_D(A)\big)}  \Bigg) - K \log n/2. \label{dclowbd1}
\end{align}
The remaining possibilities are when the maximum on the right inequality in the Lemma \ref{lem 3.2} is attained by $\vert \Phi^A(B)_k \vert ^2$ for some $2 \leq k \leq n-1$ or by $\max_{2 \leq l \leq 2m} \vert p_l(A) \vert \Phi^A(B)_k \vert^l$ with $k=1$. In this case, we first appeal to the inequality (\ref{ddeluppbd}) and put it in the from
\begin{equation} \label{I}
\frac{1}{2}\frac{\vert \Phi^A(B)\vert }{\tau_k(A)} \gtrsim  \frac{1}{2} \frac{d(B,A)}{\delta_D(A)}
\end{equation}
Next recall from (\ref{maxbd}) that for all $1 \leq j \leq n$ we have that 
\[
e^{d_D^{Cara}(A,B)} \gtrsim 1 + \Big \vert \frac{\Phi^A(B)_j}{\tau_j(A)} \Big \vert^2
\]
Summing over the first $n-1$ indices gives
\begin{align*}
e^{d_D^{Cara}(A,B)} &\gtrsim (n-1) + \sum \limits_{j=1}^{n-1} \Big \vert \frac{\Phi^A(B)_j}{\tau_j(A)} \Big \vert^2 \\
&\geq (n-1) + \frac{1}{2} \sum \limits_{j=1}^{n-1} \Big \vert \frac{\Phi^A(B)_j \vert}{\tau_j(A)} \Big \vert^2 + \frac{1}{2} \frac{\vert \Phi^A(B) \vert }{\tau_k(A)} 
\end{align*}
which can be re-written using (\ref{I}) as
\begin{align}
d_D^{Cara}(A,B) &\gtrsim \log \Bigg( 1 + \frac{1}{2(n-1)} \sum \limits_{j=1}^{n-1} \Big \vert \frac{\Phi^A(B)_j}{\tau_j(A)} \Big \vert^2 + \frac{1}{2(n-1)} \frac{1}{2} \frac{\vert \Phi^A(B)}{\tau_k(A)} \Bigg) \nonumber \\
&\gtrsim \log \Bigg( 1 + \frac{1}{2(n-1)} \sum \limits_{j=1}^{n-1} \Big \vert \frac{\Phi^A(B)_j}{\tau_j(A)} \Big \vert^2 + \frac{c}{(n-1)} 
 \frac{d(B,A)}{\delta_D(A)} \Bigg) \nonumber \\
&\gtrsim \log \Bigg( 1 + \frac{c}{(n-1)} \frac{d(B,A)}{\delta_D(A)} + \frac{1}{2(n-1)^2} \Big \vert \sum \limits_{j=1}^{n-1}   \frac{\Phi^A(B)_j}{\tau_j(A)} \Big \vert^2 \Bigg) \nonumber\\
&\gtrsim \log \Bigg( 1  +    \sum \limits_{j=1}^{n-1} \Big \vert \frac{\Phi^A(B)_j}{\tau_j(A)}\Big \vert^2    +   \frac{d(B,A)}{\delta_D(A)} \Bigg) \nonumber \\
&\gtrsim \log \Bigg( \frac{1}{n} \big( 1 + \sum \limits_{j=1}^{n-1} \Big \vert\frac{\Phi^A(B)_j}{\tau_j(A)} \Big \vert \big)^2 +  \frac{d(B,A)}{\delta_D(A)} \Bigg) \nonumber \\
&\geq \log \Bigg(  \frac{1}{n} \big( 1 + \sum \limits_{j=1}^{n-1} \Big \vert\frac{\Phi^A(B)_j}{\tau_j(A)} \Big \vert \big) + \frac{1}{n}\frac{d(B,A)}{\delta_D(A)} \Bigg) \nonumber \\
&= \log \Bigg(  1 + \sum \limits_{j=1}^{n-1} \Big \vert\frac{\Phi^A(B)_j}{\tau_j(A)} \Big \vert  + \frac{d(B,A)}{\delta_D(A)} \Bigg) - \log n \label{dclowbd2}.
\end{align}
finishing the proof of the inequality claimed at (\ref{dclowbd}).

\subsection{Upper bound on the invariant distances} \label{sixth}
In this section we shall establish the upper bound on the Kobayashi distance $d_D^k(A,B)$ between two points $A,B \in D \cap U$. Upper bounds for the {\it infinitesimal} Kobayashi metric were also obtained by Cho in \cite{Cho1} for Levi corank one domains following Catlin's methods \cite{Ca} where the upper bound was established in dimension $n=2$. Berteloot showed an elementary way of obtaining the same, which was again followed in \cite{TT} to obtain the upper bound of \cite{Cho1} for Levi corank one domains; but as in \cite{Ber} this time (in \cite{TT}) the estimate is expressed in terms of the canonical automorphisms associated with a Levi corank one domain $\Omega$, denoted therein by $M_\Omega(\zeta, X)$ and defined as
\begin{equation} \label{Mmetric}
M_{\Omega}(\zeta,X) = \sum\limits_{k=1}^{n} \vert \big(D\Phi^\zeta(\zeta) X \big)_k \vert / \tau_k(\zeta, \epsilon(\zeta)) 
=\big \vert D( B_\zeta  \circ \Phi^\zeta )(\zeta)(X) \big\vert_{l^1}
\end{equation}
with $B_\zeta = B_\zeta^{\epsilon(\zeta)}$ where $\epsilon(\zeta)>0$ is such that $\tilde{\zeta} = \zeta+(0, \ldots, 0,\epsilon(\zeta))$ lies on $\partial \Omega$ and
\[
B_\zeta^\delta (z_1, \ldots ,z_n) = \big( (\tau_1)^{-1} z_1, \ldots, (\tau_n)^{-1} z_n \big)
\]
where $\tau_1 = \tau(\zeta,\delta)$, $\tau_j = \delta^{1/2}$ for $2 \leq j \leq n-1$ and $\tau_n= \delta$. Let us also note here that $\Phi^{\tilde{\zeta}}(\zeta) = \big( {}'0, -\epsilon(\zeta)/b_n(\zeta) \big)$ where $b_n(\zeta) = \big( \partial r/\partial z_n (\zeta) \big)^{-1}$ so that $b_n(\zeta) \to 1$ and $B_{\tilde{\zeta}} \circ \Phi^{\tilde{\zeta}} \to ({}'0,-1)$ as $\zeta \to 0$, when we are in the normalization (\ref{nrmlfrm}) mentioned in the introduction. It was shown in \cite{TT} that near the boundary of $\Omega$, we have $K_{\Omega}(\zeta,X) \approx M_\Omega(\zeta, X)$. \\

\noindent We now begin the proof of the upper bound in theorem \ref{thm1}. We first get the following from \cite{Her} with almost no changes. 

\begin{lem} \label{lem7.1}
Suppose that $\tilde{A}, \tilde{B}$ are points in $D \cap U$ such that $\tilde{A}^* = \tilde{B}^*$, then 
\[
d_D^k(\tilde{A}, \tilde{B}) \leq \frac{1}{2} \log \Big( 1 + \tilde{C}\frac{\vert \delta_D(\tilde{A}) - \delta_D(\tilde{B})\vert}{{\min} \{ \delta_D(\tilde{A}), \delta_D(\tilde{B}) \} } \Big)
\]
for some positive constant $\tilde{C}$.
\end{lem}

\noindent We now proceed to use this to get an upper bound on the Kobayashi distance between the points $A$ and $B$, involving the canonical automorphisms and distances of $A,B$ to the boundary and the distance between them expressed in the pseudo-distance $d'(B,A)$ as expressed in Theorem \ref{thm1}.
So suppose that $A,B \in  D \cap U$; if $\vert A - B \vert \geq R_0$ then the claim follows from proposition 2.5 of \cite{FR}. So we have only to deal with the case when $A,B \in  D \cap U$ with $\vert A - B \vert <R_0$ in which case $d(A,B) \approx d'(A,B)$. Now we shall split-up again into two cases depending upon whether $A,B$ are near or far when their distance is measured by the pseudo-distance $d'$; to quantify the definition of nearness here we first choose constants $C_1>0, L>0, \eta>0$ with the following properties:\\

\noindent (a):  $d(x,y) \leq C_1 \big( d(x,z_1) + d(z_1,z_2) + d(z_2,y) \big)$ holds for all $x,y,z_1,z_2 \in U$ and \\

\noindent (b):  $d(x,y) \geq d'(x,y)/L$ whenever $d'(x,y)$ is finite.\\

\noindent If we choose an appropriate neighbourhood $U_0 \subset U$ of $\partial D$ we can achieve that, if $A,B \in U_0$ and $\eta \ll 1$, then the points $A - \eta d'(A,B) \nu_{A^*}$ and  $B - \eta d'(B,A) \nu_{B^*}$, still lie in $U$. \\

\noindent Define $M= 3LC_1/\eta$ and consider the following two cases, as mentioned above: \\

\noindent (i): The points $A,B$ satisfy $d'(A,B) \leq M \;\text{max} \{ \delta_D(A), \delta_D(B) \}$. In this case we begin with the upper bound on the infinitesimal Kobayashi metric obtained in \cite{TT}, which also gives the upper bound in the form expressed in \cite{Cho1}, namely:
\begin{equation} \label{eqn7.2}
K_D(z,X) \leq C \Big( \frac{\vert \langle L_1(z),X \rangle \vert }{\tau\big(z,\delta_D(z) \big)} + \sum\limits_{\alpha=2}^{n-1} \frac{\vert \langle L_j(z),X \rangle\vert}{\sqrt{\delta_D(z)}} + \frac{\vert X_n \vert }{\delta_D(z)}  \Big).
\end{equation}
Let us assume $\delta_D(A) \geq \delta_D(B)$. Let $B'$ be the point in $D$ such that $\delta_D(B') = \delta_D(A)$ and $(B')^*=B^*$. We shall connect $A$ and $B$ by a certain curve $c: [0,1] \to D$ which we shall specify in a moment before which we would like to upper bound the Kobayashi distance between $A$ and $B$ in terms of the Kobayashi length of $c$ and the Kobayashi distance between $B$ and $B'$ which lie on the same normal to $\partial D$ at $B^*$, allowing the use of Lemma \ref{lem7.1} to estimate this distance, so that we have
\begin{align}
d_D^k(A,B) &\leq d_D^k(A,B') + d_D^k(B',B) \nonumber \\
&\leq L_D^k(c) + d_D^k(B,B') \nonumber \\
&\leq L_D^{Kob}(c) + C' \log \Big( 1 + \tilde{C}\frac{\vert \delta_D(\tilde{B'}) - \delta_D(\tilde{B})\vert}{\text{min} \{ \delta_D(\tilde{B'}), \delta_D(\tilde{B}) \} } \Big) \nonumber \\
&= L_D^{Kob}(c) + C' \log \Big( 1 + \tilde{C} \frac{\vert \delta_D(B') - \delta_D(B) \vert }{\delta_D(B)} \Big) \nonumber \\
&\leq L_D^{Kob}(c) + C' \log \Big( 1 + \tilde{C} \frac{ d'(B,A)  }{\delta_D(B)} \Big).  \label{case1s} 
\end{align}
The last inequality follows  since $A\in Q_{2d'(A,B)}(B)$, so that  we have $\vert \delta_D(A) - \delta _D(B) \vert \leq d'(A,B)$.\\

\noindent We shall now define the path $c$, as follows. First, let $\gamma(t) = (\Phi^A)^{-1}\big( t \Phi^A(B) \big)$ denote the pull-back of the shortest path joining $\Phi^A(A)=0$ and the point $\Phi^A(B)$ in the pseudo-distance associated with the transformed domain $D_A=\Phi^A(D)$, which happens to be the straight line joining the origin and the point $\Phi^A(B)$. Then $\gamma(0)=A$, $\gamma(1)=B$ and since we are in case (i),
\[
\gamma \subset Q_{2d'(B,A)}(A) \subset Q_{2M \delta_D(A)} (A) \subset U.
\]
Let $\gamma^*(t) = \big( \gamma(t) \big)^*$ and $c(t) = \gamma^*(t) - \delta_D(A)\nu_{\gamma^*(t)}$. Then $\delta_D(c(t))=\delta_D(A)$ for all $t$ and $B'=c(1)$ satisfies
\[
B'= \gamma^*(1) - \delta_D(A) \nu_{\gamma^*(1)} = B^* - \delta_D(A) \nu_{B^*}
\]
proving $(B')^*=B^*$. \\

\noindent Next, to extend the chain of inequalities in (\ref{case1s}) and reach the required upper bound, we need to estimate the length of $c$ in the Kobayashi metric. Write $N(x) = \nu \big( \pi_{\partial D}(x) \big)$ and $\gamma^*(t) = \pi_{\partial D}(\gamma(t))$. Then,
\[
\dot{\gamma}^*(t) = D\big(\pi_{\partial D} \big) (\gamma(t)) \cdot \dot{\gamma}(t).
\]
Therefore,
\begin{align*}
\dot{c}(t) &=  D\big(\pi_{\partial D} \big)_{\vert_{\gamma(t)}} (\dot{\gamma}(t)) - \delta_D(A) \cdot \frac{d}{dt} \nu \big( \pi_{\partial D}(\gamma(t)) \big)\\
&= \Big( D(\pi _{\partial D}) (\gamma(t)) - \delta_D(A) D(N)(\gamma(t)) \Big) \dot{\gamma}(t).
\end{align*}
This implies 
\begin{equation}\label{cgamcompar}
\vert \dot{c}(t)_j \vert \lesssim \vert \dot{\gamma}(t)_j \vert 
\end{equation}
for all $1 \leq j \leq n$. Now in order to get an upper bound on $L_D^k(c)$, we shall make use of the estimate on the infinitesimal Kobayashi metric in the form expressed in \cite{TT} to upper bound $K_D\big(c(t), \dot{c}(t) \big)$, the Kobayashi-length of the tangent vector $\dot{c}(t)$ and integrate it. We intend to use (\ref{cgamcompar}) for this purpose and this entails estimating $K_D\big( \gamma (t), \dot{\gamma}(t) \big)$ and thereby subsequently, to getting a control on the size of the tangent to the curve $\gamma$ along each of the co-ordinate directions namely, $\vert \dot{\gamma}(t)_j \vert$ for each $1 \leq j \leq n$. Indeed, observe first that
\begin{equation}
K_D\big(c(t), \dot{c}(t) \big) \lesssim \Bigg( \frac{\vert \langle L_1\big( c(t) \big), \dot{c}(t) \rangle \vert }{\tau\big(c(t), \delta(c(t)) \big)} 
+ \sum\limits_{\alpha=2}^{n-1}  \frac{\vert\langle L_\alpha\big( c(t) \big), \dot{c}(t) \rangle \vert}{\sqrt{\delta(c(t))}} + \frac{\vert \dot{c}(t)_n \vert }{\delta(c(t))} \Bigg).
\end{equation}
Now notice for example that
\begin{align*}
\Big \vert \langle L_1\big( c(t) \big), \dot{c}(t) \rangle \Big\vert &=  \Big\vert \dot{c}(t)_1 - \big(\frac{\partial r/\partial z_1}{\partial r/\partial z_n}\big)\big( c(t)\big) \; \dot{c}(t)_n \Big\vert \\
& \lesssim \vert \dot{\gamma}(t)_1 \vert + \vert \dot{\gamma}(t)_n \vert
\end{align*}
and subsequently,
\begin{align*}
\frac{\vert \langle L_1\big( c(t) \big), \dot{c}(t) \rangle \vert }{\tau\big(c(t), \delta(c(t)) \big)} &\lesssim \frac{\vert \dot{\gamma}(t)_1 \vert + \vert \dot{\gamma}(t)_n \vert}{\tau\big(c(t), \delta(c(t)) \big)} \\
& \lesssim \frac{\vert \dot{\gamma}(t)_1 \vert }{\tau\big(A, \delta(A) \big)} + \frac{ \vert \dot{\gamma}(t)_n \vert}{\tau\big(A, \delta(A) \big)} 
\end{align*}
since $\delta \big( c(t) \big) \equiv \delta (A)$ and $\tau \big( c(t), \delta(A) \big) \approx \tau\big(A, \delta(A) \big)$ which follows from the uniform comparability of these distinguished radii at different points within a distorted polydisc and the fact that the path $c(t)$ remains at a fixed small distance from the boundary of $D$, small enough for the polydiscs to be defined and validate the application of the comparability of the radii by covering the path by finitely many such polydiscs. Next recalling that $1/ \tau\big(A, \delta(A)\big) \lesssim 1/\delta(A)$, we obtain
\begin{equation}
\frac{\vert \langle L_1\big( c(t) \big), \dot{c}(t) \rangle \vert }{\tau\big(c(t), \delta(c(t)) \big)} \lesssim \frac{\vert \dot{\gamma}(t)_1 \vert  }{\tau\big(A, \delta(A) \big)} + \frac{\vert  \dot{\gamma}(t)_n \vert}{\delta(A)}
\end{equation}
Similar calculations give
\begin{equation}
\frac{\vert \langle L_\alpha\big( c(t) \big), \dot{c}(t) \rangle \vert }{\sqrt{\delta(c(t))}} \lesssim \frac{\vert \dot{\gamma}(t)_\alpha \vert  }{\sqrt{\delta(A)}} + \frac{\vert \dot{\gamma}(t)_n \vert}{\delta(A)}
\end{equation}
so that altogether we get
\begin{equation} \label{Kcupbd}
K_D\big(c(t), \dot{c}(t) \big) \lesssim \Big( \frac{\vert  \dot{\gamma}_1 (t) \vert }{\tau\big(A, \delta(A) \big)} 
+ \sum\limits_{\alpha=2}^{n-1} \frac{\vert  \dot{\gamma}_\alpha (t) \vert}{\sqrt{\delta(A)}} + \frac{\vert \dot{\gamma}(t)_n \vert }{\delta(A)} \Big)
\end{equation}
which subsequently entails the estimation of the components $\dot{\gamma}(t)$ as mentioned above, towards which we now proceed. First let $\Psi_A = (\Phi^A)^{-1}$ and $\zeta = \Phi^A(B)$ and note that 
\[
\dot{\gamma}(t) = D \Psi_{A_{\vert_{t \zeta}}} (\zeta).
\]
The expression for the inverse map $\Psi_A$ and its derivative have been put down in the appendix, section \ref{app}, using which 
we may write down $\dot{\gamma}(t)$ more explicitly in the form
\[
\dot{\gamma}(t) = \Big( \zeta_1, H_A(\tilde{\zeta}) + H_A\big( \zeta_1 \frac{\partial Q_2}{\partial z_1}(t \zeta_1) \big), b_n^A \big( \zeta_n + \sum\limits_{j=1}^{n-1}\zeta_j \frac{\partial \tilde{Q}_1}{\partial z_j }(t '\zeta) \big) \Big)
\]
where $H_A= G_A^{-1}$ and 
\[
\tilde{Q}_1({}'z) = (b_n^A)^{-1} \Big(\langle \tilde{b}^A, H_A \big( \tilde{z} + Q_2(z_1)\big)\rangle + b_1^A z_1 \Big) + Q_1 \Big(z_1, H_A  \big( \tilde{z} + Q_2(z_1)\big) \Big).
\]
where $Q_1$ and $Q_2$ are the same polynomials that occur in the expression for $\Phi^\zeta$ as in (\ref{E45}). We expand $\tilde{Q}_1$ a bit more explicitly, in the coordinates $z_1, \ldots,z_{n-1}$, to put it in the form
\begin{equation} \label{Qhatfrm}
\tilde{Q}_1({}'z) =\sum\limits_{k=2}^{2m} d_k z_1^k + \sum\limits_{\alpha=2}^{n-1}\sum\limits_{k=1}^{m}d_{\alpha,k} z_1^k \big(z_\alpha + P_2^\alpha(z_1) \big)
+ \sum
\limits_{\alpha=2}^{n-1}c_\alpha\big( z_\alpha + P_2^\alpha(z_1) \big)^2
\end{equation}
where we recall that 
\[
P_2^\alpha(z_1)= \sum\limits_{l=1}^{m}e^\alpha_l z_1^l.
\]
Now (\ref{Kcupbd}) extends to
\begin{align} 
K_D\big( \gamma (t), \dot{\gamma}(t) \big) &\lesssim \frac{\vert \dot{\gamma}(t)_n \vert }{\delta(\gamma(t))} + \sum\limits_{\alpha=2}^{n-1} \frac{\vert \dot{\gamma}(t)_\alpha \vert }{\sqrt{\delta(\gamma(t))}} + \frac{\vert \dot{\gamma}(t)_1 \vert}{ \tau \big( \delta(\gamma(t)) \big) }  \nonumber\\
&\lesssim \frac{  \vert b_n^A \vert \Big\vert \zeta_n + \sum\limits_{j=1}^{n-1}\zeta_j \frac{\partial \tilde{Q}_1}{\partial z_j}(t {}' \zeta) \Big\vert  }{ \delta(A)} + \Bigg \vert \frac{ H_A(\tilde{\zeta}) + H_A\big( \zeta_1 \big( \frac{\partial Q_2}{\partial z_1}(t \zeta_1) \big)\big)} { \sqrt{\delta(A)} }\Bigg \vert + \frac{\vert \zeta_1 \vert}{\tau\big( A, \delta(A) \big)}. \label{Chest}
\end{align}
Now we estimate the numerator in the first summand to which end, we first estimate
\[
\Big\vert \sum\limits_{j=1}^{n-1}\zeta_j \frac{\partial \tilde{Q}_1}{\partial z_j}(t \; {}'\zeta) \Big\vert.
\]
The quantity within the modulus is ofcourse a polynomial of the same form as $Q_1$ in $\zeta$ (i.e., upto multiplication by some factors which are either powers of $t$ (also recall here that $\vert t \vert \leq 1$) or some integers), since it is obtained by applying the differential operator
$ z_1 \partial/\partial z_1 + \ldots  + z_{n-1} \partial/\partial z_{n-1}$ which is of weight zero, to the polynomial $\tilde{Q}_1 ({}'z)$ and evaluated at $t \; '\zeta$. To put down one instance of the calculation explicitly, we first use the form of $\tilde{Q}_1$ as in (\ref{Qhatfrm}) and write down
\begin{align*}
\sum\limits_{j=1}^{n-1} \zeta_j \frac{\partial \tilde{Q}_1}{\partial z_j}(t {}' \zeta) & = \sum\limits_{k=2}^{2m} k d_k t^{k-1} \zeta_1^k + \sum\limits_{\alpha=2}^{n-1}\sum\limits_{k=1}^{m}\sum\limits_{l=1}^{m}(l+k)d_{\alpha,k}e^{\alpha}_l t^{l+k-1}\zeta_1^{l+k}\\
& \;\; + \sum\limits_{\alpha=2}^{n-1}2 c_\alpha^2 t\zeta_\alpha^2 + \sum\limits_{\alpha=2}^{n-1} \sum\limits_{l=1}^{m} c_\alpha e^{\alpha}_l l t^{l-1} \zeta_1^l \zeta_\alpha +\sum\limits_{\alpha=2}^{n-1} \sum\limits_{l=1}^{m} c_\alpha e^{\alpha}_l  t \zeta_1^l  \\
&\;\; + 2 \sum\limits_{\alpha=2}^{n-1} \big( \sum\limits_{l=1}^{m} e^{\alpha}_l t^l \zeta_1^l \big) \big( \sum\limits_{l=1}^{m} e^{\alpha}_l t^{l-1} l \zeta_1^l \big)
\end{align*}
Now using the estimates on the various coefficients from Lemma (3.4) of \cite{TT} -- for instance again, the coefficient $(l+k)d_{\alpha,k} e_{\alpha ,k} t^{l+k-1}$ occurring in the second summand in the above, which we shall denote by $\mathbf{c}_{\alpha,l,k,t}$, can estimated as
\begin{align*}
\mathbf{c}_{\alpha,l,k,t}  &\lesssim (l+k) \delta(A) \tau_1\big( A,\delta(A) \big)^{-k} \tau_\alpha \big( A,\delta(A) \big)^{-1} \; \delta(A)  \tau_1\big( A,\delta(A) \big)^{-l} \tau_\alpha \big( A,\delta(A) \big)^{-1} t^{l+k-1} \\
& \lesssim (\delta(A))^2 \tau_1\big( A,\delta(A) \big)^{-(k+l)} \tau_\alpha \big( A,\delta(A) \big)^{-2} \\
&\lesssim \delta(A) \tau_1\big( A,\delta(A) \big)^{-(k+l)}
\end{align*}
since $\tau_\alpha \big( A,\delta(A) \big) = \sqrt{\delta(A)}$ and $t<1$. Consequently, the monomial $\mathbf{c}_{\alpha,l,k,t} \zeta_1^{l+k} \lesssim \delta(A)$. Subsequent similar computations result finally in 
\[
 \Big\vert \sum\limits_{j=1}^{n-1}\zeta_j \frac{\partial \hat{Q}_1}{\partial z_j}(t {}' \zeta) \Big\vert \lesssim \delta(A)
\]
which means that the first term in (\ref{Chest}) above is bounded above by
\[
\frac{\vert \Phi^A(B)_n \vert}{\delta_D(A)} + \text{some constant}
\]
Next, to say a few words about the second term at (\ref{Chest}), we first observe that it is bounded above by
\[
\big \vert H_A \big \vert \Big( \frac{\vert \tilde{\zeta}\vert }{ \sqrt{\delta(A)} } + \frac{ \big \vert  \zeta_1 \big( \frac{\partial Q_2}{\partial z_1}(t \zeta_1) \big)\big \vert } { \sqrt{\delta(A)} } \Big)
\]
and then note that 
\begin{align*}
\Big \vert \zeta_1 \frac{\partial Q_2  }{\partial z_1}(t \zeta_1) \Big \vert  &= \Big \vert  \sum\limits_{k=1}^{m} k b^{\alpha}_k t^{k-1} \zeta_1^k \Big\vert  \\
&\leq \big \vert  \sum\limits_{k=1}^{m} k b^{\alpha}_k \zeta_1^k \big\vert \text{ since } t<1\\
& \lesssim \sum\limits_{k=1}^{m} \delta(A) \tau_1\big(A,\delta(A)\big)^{-k} \tau_\alpha^{-1} \tau_1\big(A,\delta(A) \big)^{k}, \text{ (by Lemma  3.4 of \cite{TT})} \\
&\lesssim \sqrt{\delta(A)} \; .
\end{align*}

\noindent Thus the upshot is that the second summand in (\ref{Chest}) is bounded above up to a constant by
\[
\sum\limits_{\alpha=2}^{n-1} \frac{\vert \Phi^A(B)_\alpha \vert}{\sqrt{\delta(A)}} + \text{ some constant}
\]
and in all (\ref{Chest}) transforms to the more concrete upper bound
\begin{align} 
K_D\big( \gamma (t), \dot{\gamma}(t) \big) &\lesssim \frac{d'(A,B)}{\delta(A)} + \frac{\vert \Phi^A(B)_1 \vert }{\tau \big(A,\delta(A) \big)}
+ \sum\limits_{\alpha=2}^{n-1} \frac{\vert \Phi^A(B)_\alpha \vert}{\sqrt{\delta(A)}} + L 
\end{align}
for some positive constant $L$. Finally (\ref{case1s}) becomes
\begin{align} 
d_D^k(A,B) &\lesssim \frac{\vert \Phi^A(B)_n \vert}{\delta(A)} + \frac{\vert \Phi^A(B)_1 \vert }{\tau \big(A,\delta(A) \big)}
                         + \sum\limits_{\alpha=2}^{n-1} \frac{\vert \Phi^A(B)_\alpha \vert}{\sqrt{\delta(A)}} + L_* \nonumber \\
&\lesssim \log \Bigg( 1 + \frac{\vert \Phi^A(B)_n \vert}{\delta(A)} + \frac{\vert \Phi^A(B)_1 \vert }{\tau \big(A,\delta(A) \big)}
+ \sum\limits_{\alpha=2}^{n-1} \frac{\vert \Phi^A(B)_\alpha \vert}{\sqrt{\delta(A)}} \Bigg) + L^* \label{uppbd}
\end{align}
since the quantity 
\[
\frac{\vert \Phi^A(B)_n \vert}{\delta(A)} + \frac{\vert \Phi^A(B)_1 \vert }{\tau \big(A,\delta(A) \big)}
+ \sum\limits_{\alpha=2}^{n-1} \frac{\vert \Phi^A(B)_\alpha \vert}{\sqrt{\delta(A)}}
\]
is uniformly bounded above as we are in case (i) i.e., $B \in Q_{2 M \delta(A)}(A)$ and that when written out explicitly implies the just-asserted uniform-boundedness; we then get (\ref{uppbd}) by simply using the fact that the function $(\log(1+x))/x$ is bounded below by a positive constant when $x$ varies over a compact interval of positive reals. Since $\vert \Phi^A(B)_n \vert \leq  d'(B,A)$ and $d'(B,A) \lesssim d(B,A)$ by part (iii) of Lemma \ref{dprop} we may also write the inequality (\ref{uppbd}) as 
\begin{equation} \label{kobdistuppbd}
d_D^k(A,B) \lesssim \log \Bigg( 1 + \frac{d(B,A)}{\delta(A)} + \frac{\vert \Phi^A(B)_1 \vert }{\tau \big(A,\delta(A) \big)} + \sum\limits_{\alpha=2}^{n-1} \frac{\vert \Phi^A(B)_\alpha \vert}{\sqrt{\delta(A)}} \Bigg) + L^*.
\end{equation}

\noindent The case (ii) can be reduced essentially to the first, as in section $7$ of \cite{Her} following the line of arguments therein and the above estimates (\ref{uppbd}) and (\ref{kobdistuppbd}) may be obtained in that case as well.


\section{Fridman's invariant function on Levi corank one domains} \label{FriL}

\noindent The purpose of this section is to prove Theorem \ref{thm2}. But before that, we gather some
interesting properties of Fridman's invariant function $ h_D( \cdot, \mathbb{B}^n ) $ that were proved in 
\cite{Fr}.

\begin{prop}\label{E43}
Let $ \Omega $ be a Kobayashi hyperbolic manifold of complex dimension $ n $. Then

\begin{itemize}

\item if there is a $ p^0 \in \Omega $ such that $ h_{\Omega} ( p^0, \mathbb{B}^n ) = 0 $,
then $ h_{\Omega} ( p^0, \mathbb{B}^n ) \equiv 0 $ and $ \Omega $ is biholomorphically 
equivalent to $ \mathbb{B}^n $.

\item $ p \mapsto h_{\Omega} (p, \mathbb{B}^n )  $ is continuous on $ \Omega $.

\end{itemize}

\end{prop}

\noindent To put things in perspective, we state the following result on the boundary 
behaviour of Fridman's invariant for strongly pseudoconvex domains.

\begin{thm}\label{E42}
Let $ D \subset \mathbb{B}^n $ be a bounded domain, $ p^0 \in \partial D $ and let $ \{ p^j \}
\subset D $ be a sequence that converges to $ p^0 $. If $ D $ is $ C^2$-smooth strongly pseudoconvex 
equipped with either the Kobayashi or the Carath\'{e}odory metric, then
$ h_D( p^j, \mathbb{B}^n ) \rightarrow 0 $ as $ j \rightarrow \infty $.
\end{thm}

\noindent The reader is referred to \cite{MV} for a proof. It should be noted that, for $ D $ a Levi 
corank one domain, as in Theorem \ref{thm2}, the limit $ h_{D_{\infty}} \big( ('0, -1), \mathbb{B}^n \big) $
can be strictly positive, unlike the strongly pseudoconvex case, and, in general, depends on the nature
of approach $ p^j \rightarrow p^0 \in \partial D $. Recall that, here, and in the sequel, for any $ z \in \mathbb{C}^n $, 
$ z = ('z, z_n ) $ and $ 'z $ will denote $ (z_1, \ldots, z_{n-1} ) $.

\medskip

\noindent Before going further, let us briefly recall the scaling technique (cf. \cite{TT}) for a smoothly 
bounded pseudoconvex domain $ D \subset \mathbb{C}^n $ of finite type when the Levi form 
of $ \partial D $ has rank at least $ (n-2) $ at $ p^0 \in \partial D $. Assume that 
$D $ is given by a smooth defining function $r$ and that $ p^0 $ is the origin. Consider a 
sequence $ p^j \in D $ that converges to the origin and denote by $ \zeta^j $, the point 
on $ \partial D $ chosen so that $ \zeta^j = p^j + ('0, \epsilon_j) $ for some 
$ \epsilon_j > 0 $. Also, $ \epsilon_j 
\approx \delta_D(p^j) $. 

\medskip

\noindent Let $ \Phi^{\zeta^j} $ be the polynomial automorphisms of 
$ \mathbb{C}^n $ corresponding to $ \zeta^j \in \partial D $ as described in (\ref{E45}). 
It can be checked from the explicit form of $ \Phi^{\zeta^j} $
that $ \Phi^{\zeta^j}(\zeta^j) = ('0, 0) $ and 
\[
 \Phi^{\zeta^j} (p^j) = 
\big('0, - \epsilon_j/ d_0(\zeta^j) \big) ,
\]
where $ d_0(\zeta^j ) = \Big( \partial r/\partial \overline{z}_n (\zeta^j) \Big)^{-1} \rightarrow 1 $ as $ j \rightarrow \infty $. Define a 
dilation of coordinates by
\[
\Delta^{\epsilon_j}_{\zeta^j} (z_1, z_2, \ldots, z_n ) = \Big( \frac{z_1}{\tau(\zeta^j, \epsilon_j)}, 
\frac{z_2}{\epsilon_j^{1/2}}, \ldots, \frac{z_{n-1}}{\epsilon_j^{1/2}}, \frac{z_n}{\epsilon_j} \Big),
\]
where $ \tau(\zeta^j, \epsilon_j) $ are as defined in (\ref{E46}). Note that $ \Delta^{\epsilon_j}_{\zeta^j} 
\circ \Phi^{\zeta^j} (p^j) = \big('0, - 1/ d_0(\zeta^j) \big) $. For brevity, we write $ \big('0, - 
1/ d_0(\zeta^j) \big) = z^j $ and $ ('0, -1) = z^0 $. It was shown in \cite{TT}
that the scaled domains $ D^j = \Delta^{\epsilon_j}_{\zeta^j} \circ \Phi^{\zeta^j} (D ) $ converge
in the Hausdorff sense to
\[
D_{\infty} = \big\{ z \in \mbb C^n : 2 \Re z_n + P_{2m}(z_1, \ov z_1) + \abs{z_2}^2 + \ldots + \abs{z_{n-1}}^2 < 0  \big\}
\]
where $P_{2m}(z_1, \ov z_1)$ is a subharmonic polynomial of degree at most $2m$ ($m \ge 1$) without harmonic terms, $2m$ being the $1$-type of $\pa D$ at $p^0$. Observe that $ D_{\infty} $ is complete hyperbolic (each point on $ \partial D_{\infty} $, including the point at
infinity, is a local holomorphic peak point -- cf. Lemma 1 of \cite{BP}) and hence $ D_{\infty} $ is taut.

\medskip

\noindent It is natural to investigate the stability of the Kobayashi metric at the infinitesimal level first. The following 
lemma can be proved using the same ideas as in Lemma 5.2 of \cite{MV}. The only requirement is to establish the normality of 
a scaled family of holomorphic mappings which follows from Theorem 3.11 of \cite{TT}.

\begin{lem} \label{E27} For $ (z, v) \in D_{\infty} \times \mathbb{C}^n, \lim_{j \rightarrow \infty} K_{D^j} (z, v) = K_{D_{\infty}} (z, v) $. Moreover, the convergence is uniform on compact sets of $ D_{\infty} \times \mathbb{C}^n $.
\end{lem}

\noindent\textit{Proof of Theorem \ref{thm2}:} There are two cases to be examined. After passing to a subsequence, if needed,

\begin{enumerate}

\item[(i)] $ \lim_{j \rightarrow \infty} h_D( p^j, \mathbb{B}^n ) = 0 $, or

\item[(ii)] $ \lim_{j \rightarrow \infty} h_D( p^j, \mathbb{B}^n ) > c $ for some positive constant $c$.

\end{enumerate}

\noindent In the first case, arguments similar to the ones employed in Theorem 5.1(i) of \cite{MV} together with Lemma \ref{E27} show that the limit domain $ D_{\infty} $ is biholomorphic to $ \mathbb{B}^n $. On the other hand, this will not be true in the second case. 

\medskip

\noindent To analyse case (ii), it will be useful to consider the stability of the Kobayashi balls on the scaled domains $ D^j $ around $ z^j \in D^j $ with a fixed radius $ R > 0 $. The proof of this is accomplished in two steps. In the the first part, we show that the sets $ B_{D^j} (z^j, R) $ do not accumulate at the point at infinity in $ \partial D_{\infty} $. The proof of this statement relies on Theorem \ref{thm1} and unravelling the definition of the `normalizing maps' $ \Phi^{\zeta^j} $ and the dilations $ \Delta^{\epsilon_j}_{\zeta^j} $. The second part is to show that the sets $ B_{D^j} (z^j, R) $ do not cluster at any finite boundary point of $\partial D_\infty$.

\begin{lem} \label{E28} For each $ R > 0 $ fixed, $ B_{D^j} (z^j, R) $ is uniformly compactly contained in $ D_{\infty} $ for all $j$ large.
\end{lem}

\begin{proof} Since the scaling maps $ \Delta^{\epsilon_j}_{\zeta^j} \circ \Phi^{\zeta^j} $ are biholomorphisms and therefore Kobayashi isometries, it immediately follows that
\begin{equation*} \label{same}
B_{D^j}(z^j, R) = \Delta^{\epsilon_j}_{\zeta^j} 
\circ \Phi^{\zeta^j} \big( B_D(p^j,R) \big).
\end{equation*}
\noindent We assert that $B_{D^j}(z^j,R)$ cannot accumulate at the point of infinity in $\partial D_\infty$. To establish this,
assume that $ q \in B_D(p^j, R) $ and consider the lower bound on the Kobayashi distance given by Proposition \ref{lowbd}: 
\[
C_*\log \left( 1 + \Big (\frac{ d(p^j,q)}{\delta_D(p^j)} \Big)^2 \right) \leq d_D^{k}(p^j,q) \leq R
\]
which yields that 
\[
d(p^j,q) \leq ({\exp}^{R/C_*}-1)^{1/2} \delta_D(p^j) < {\exp}^{R/C_*} \delta_D(p^j) 
\]
where $ C_* $ is a positive constant (uniform in $j$). Then the definition of the pseudodistance $ d $ 
quickly leads to the following two possibilities:
\begin{itemize}
\item either $\vert p^j - q  \vert_{l^\infty} < {\exp}^{R/C_*} \delta_D(p^j)$ or 
\item for each $j$, there exists $\delta_j \in \big( 0 , {\exp}^{R/C_*} \delta_D(p^j) \big)$ such that $p^j \in Q(q, \delta_j)$.
\end{itemize}

\noindent Now, Proposition 3.5 of \cite{TT} tells us that if $ p^j \in Q(q, \delta_j) $ then $ q \in Q(p^j, C \delta_j)$
for some uniform constant $ C> 0 $. Hence, the second statement above can be rephrased in the following manner -- for every $j$ there exists $ \delta_j \in \big( 0, {\exp}^{R/C_*} \delta_D(p^j) \big) $ such that $ q \in Q(p^j, C \delta_j)$
or equivalently that
\[
q \in \big( \Phi^{p^j} \big)^{-1} \Big( \Delta(0, \tau(p^j, C\delta_j) ) \times \Delta(  0, \sqrt{C \delta_j} ) \times \ldots \times \Delta(0, \sqrt{C \delta_j}) ) \times \Delta(0, C \delta_j) \Big)
\]
In other words, each Kobayashi ball $ B_D(p^j,R)$ is contained in the union $ G^j_1 \cup G^j_2 $ where 
\begin{eqnarray*} \label{union}
G^j_1 &= &\big\{ z \in \mathbb{C}^n : \vert z - p^j \vert_{l^{\infty}} <  {\exp}^{R/C_*} \delta_D(p^j) \big\} \mbox{ and} \\ 
G^j_2 &= & \big( \Phi^{p^j} \big)^{-1} \Big( \Delta(0, \tau(p^j, C \delta_j)) \times \Delta(0, \sqrt{C \delta_j}) \times \ldots \times \Delta(0, \sqrt{C \delta _j}) \times \Delta(0, C \delta_j) \Big).
\end{eqnarray*}
The idea is to verify that the sets $ \Delta^{\epsilon_j}_{\zeta^j} \circ \Phi^{\zeta^j} (G^j_1) $ and $ 
\Delta^{\epsilon_j}_{\zeta^j} \circ \Phi^{\zeta^j} (G^j_2) $ are uniformly bounded. For this, consider
\begin{align*}
\Phi^{\zeta^j}(G^j_1) &= \Phi^{\zeta^j} \Big( \big \{ z \in \mathbb{C}^n : \vert z - p^j \vert_{l^\infty} < {\exp}^{R/C_*} \delta_D(p^j) \big \} \Big) \\
&= \big \{ w \in \mathbb{C}^n : \vert \big( \Phi^{\zeta^j} \big)^{-1} (w) - p^j \vert_{l^\infty} < {\exp}^{R/C_*} \delta_D(p^j) \big \}.
\end{align*}
Now, write 
\[
w - \Phi^{\zeta^j}(p^j) = \Phi^{\zeta^j} \Big( \big( \Phi^{\zeta^j} \big)^{-1} (w) \Big) - \Phi^{\zeta^j}(p^j)
\]
and note that the derivatives $\{ D \Phi^{\zeta^j} \}$ are uniformly bounded in the operator norm by $L$, say. Therefore, for $ w \in \Phi^{\zeta^j}(G^j_1) $, we have that
\[
\big\vert w - \Phi^{\zeta^j}(p^j) \big\vert_{l^\infty} \leq  L  \big\vert \big( \Phi^{\zeta^j} \big)^{-1}(w) - p^j \big\vert_{l^{\infty}} < L {\exp}^{R/C_*} \delta_D(p^j),
\]
and consequently that,
\[
\Phi^{\zeta^j} (G^j_1) \subset \big\{ w \in \mathbb{C}^n : \big\vert w - \Phi^{\zeta^j}(p^j) \big\vert_{l^\infty} < L {\exp}^{R/C_*} \delta_D(p^j) \big\}.
\]
Since $\Phi^{\zeta^j}(p^j) = \big( '0, - \epsilon_j/d_0(\zeta^j) \big) $, the above inclusion can be rewritten as
\begin{multline*}
\Phi^{\zeta^j}(G^j_1) \subset \left\{ w : \vert w_k \vert < L {\exp}^{R/C_*} \delta_D(p^j) \mbox{ for }
1 \leq k \leq  n-1, \left\vert w_n + \frac{\epsilon_j}{d_0(\zeta^j)} \right\vert < L {\exp}^{R/C_*} \delta_D(p^j) \right\}.
\end{multline*}
Hence
\begin{eqnarray} \label{E53}
\Delta_{\zeta^j}^{\epsilon_j} \circ \Phi^{\zeta^j} (G^j_1) \subset \left\{ w : \vert w_1 \vert < 
\frac{L{\exp}^{R/C_*} \delta_D (p^j) }{ \tau(\zeta^j, \epsilon_j) }, \vert w_\alpha \vert < \frac{L{\exp}^{R/C_*} \delta_D (p^j) }{{\epsilon_j}^{1/2}}  \text{ for } 2 \leq \alpha \leq n-1, \right.  \nonumber \\ 
\left. \left \vert w_n + \frac{1}{ d_0(\zeta^j) } \right\vert  < 
\frac{L{\exp}^{R/C_*} \delta_D (p^j) }{ \epsilon_j } \right\}. 
\end{eqnarray}
If $ w = (w_1, \ldots, w_n) $ belongs to the set described by (\ref{E53}) above, then
\begin{alignat}{3} \label{E55} 
\left\{ \begin{array}{lrl}
\vert w_1 \vert  <  \frac{L{\exp}^{R/C_*} \delta_D (p^j) }{ \tau(\zeta^j, \epsilon_j) } \lesssim  
\frac{L{\exp}^{R/C_*} \epsilon_j }{ \tau(\zeta^j, \epsilon_j) }, \\
\medskip
\vert w_\alpha \vert  < \frac{L{\exp}^{R/C_*} \delta_D (p^j) }{{\epsilon_j}^{1/2}} \lesssim  L{\exp}^{R/C_*}  {\epsilon_j}^{1/2} \text{ for } 2 \leq \alpha \leq n-1, \\
\vert w_n \vert  < \frac{L{\exp}^{R/C_*} \delta_D (p^j) }{ \epsilon_j } + \frac{1}{ d_0(\zeta^j) }
\lesssim L{\exp}^{R/C_*}  + 1
\end{array}
\right.
\end{alignat} 
since $ \epsilon_j \approx \delta_D(p^j) $ and $ d_0 (\zeta^j) \approx 1 $. Furthermore, to examine the sets $ \Delta^{\epsilon_j}_{\zeta^j} \circ \Phi^{\zeta^j} (G^j_2) $, first note that  
these sets are the images of the unit polydisc in $ \mathbb{C}^n $ under the maps $ 
\Delta_{\zeta^j}^{\epsilon_j} \circ \Phi^{\zeta^j} \circ \big( \Phi^{p^j} \big)^{-1} \circ \big( \Delta_{p^j}^{C \delta_j} \big)^{-1} $. Let $ K $ be a positive constant such that $\vert \det D \big( \Phi^{\zeta^j} \circ \big( \Phi^{p^j} \big)^{-1} \big) \vert <K$ for all $ j $ large. It follows that each set $ \Delta^{\epsilon_j}_{\zeta^j} \circ \Phi^{\zeta^j} (G^j_2) $ is contained in a polydisc centered at $ \big('0, -1/d_0(\zeta^j) \big) $, given by
\begin{multline} \label{E54}
\Delta \left( 0, K^2 \frac{ \tau(p^j, C \delta_j)} {\tau(\zeta^j, \epsilon_j) } \right) 
\times \Delta \left( 0, K^2 \left(\frac{C \delta_j}{ \epsilon_j } \right)^{1/2} \right) \times \ldots \\
\ldots \times \Delta \left( 0, K^2 \left(\frac{C \delta_j}{ \epsilon_j } \right)^{1/2} \right) \times \Delta 
\left( -\frac{1}{d_0(\zeta^j)} , K^2 \frac{ C \delta_j}{\epsilon_j} \right). 
\end{multline}
Observe that, if $ w$ belongs to the polydisc as defined above, then
\begin{alignat}{3} \label{E56}
\left\{ \begin{array}{lrl}
|w_1|  <  K^2 \frac{ \tau(p^j, C \delta_j)} {\tau(\zeta^j, \epsilon_j) }, \\
|w_{\alpha} |  <  K^2 \left(\frac{C \delta_j}{ \epsilon_j } \right)^{1/2} \lesssim K^2 \left( \frac{ \delta_D(p^j)} { \epsilon_j } \right)^{1/2} \lesssim K^2 \text{ for } 2 \leq \alpha \leq n-1, \\
\vert w_n \vert  <  K^2 \frac{ C \delta_j }{\epsilon_j} + \frac{1}{d_0(\zeta^j)} \lesssim K^2 
\frac{\delta_D(p^j) }{\epsilon_j} + 1 \lesssim K^2 + 1.
\end{array}
\right.
\end{alignat} 
It follows from \cite{Cho2} that $ \epsilon_j^{1/2} \lesssim \tau(\zeta^j, \epsilon_j) $ and $ 
\tau(p^j, C \delta_j) \lesssim \tau(p^j,\epsilon_j) \approx \tau(\zeta^j, \epsilon_j) $.
As a consequence, we see that if $ w $ belongs to either of the sets (\ref{E55}) or (\ref{E56}),
then $ |w| $ is uniformly bounded. Hence, by virtue of the inclusions (\ref{E53}) and (\ref{E54}), 
it is immediate that the sets $ \Delta^{\epsilon_j}_{\zeta^j} \circ \Phi^{\zeta^j} (G^j_1) $ and $ \Delta^{\epsilon_j}_{\zeta^j} \circ \Phi^{\zeta^j} (G^j_2) $ are uniformly bounded. This in turn 
implies that the sets $ B_{D^j}(z^j,R) $ cannot cluster at the point at infinity in $\partial D_\infty$.\\

\noindent It remains to show that the sets $B_{D^j}(z^j,R)$ do not cluster at any finite point of $\partial D_\infty$. Suppose that there is a sequence of points $q^j \in B_{D^j}(z^j,R)$ such that $q^j \to q^0 \in \partial D_\infty$
where $ q^0 $ is any finite boundary point. Now, Theorem 2.3 of \cite{BV} assures us that the estimates for the 
infinitesimal Kobayashi metric due to Thai-Thu remain stable for the family of the scaled domains $ D^j $, i.e., there is a neighbourhood $U$ of $q^0$ in $ \mathbb{C}^n $ such that 
\begin{alignat}{3} \label{E57}
K_{D^j}(z,v) \gtrsim \frac{ \vert \big( D^j \Phi^z(z)(v) \big)_1 \vert} { \tau \big(z, \delta_{D^j}(z) \big) } 
+ \sum\limits_{\alpha=2}^{n-1} \frac{ \vert \big( D {}^j \Phi^z(z)(v) \big)_\alpha \vert} { ( \delta_{D^j}
(z))^{1/2} } + \frac{ \vert \big( D^j \Phi^z(z)(v) \big)_n \vert } { \delta_{D^j}(z) } 
\end{alignat}
uniformly for all $ j $ large, $ z \in U \cap D_{\infty} $ and tangent vector $ v \in \mathbb{C}^n $. 
Here, the notation ${}^j \Phi^z(\cdot)$ is the special boundary chart (as described by (\ref{E45})) corresponding to $ z $,
when $z$ is viewed as a point in the scaled domain $D^j$. Evidently, the last component of $ D  \Phi^\zeta(z)(v)$ is given
by
\[
\big(D \Phi^\zeta(z)(v)\big)_n =  \langle\nu(\zeta), v \rangle - \sum_{j=1}^n \frac{\partial Q_1}{\partial z_j} ({}'z - {}'\zeta) v_j, 
\]
so that 
\[
\big( D {}^j \Phi^z (z)(v) \big)_n = \langle \nu(z), v \rangle.
\]
Now, consider $\tilde{U}$, a neighbourhood of $z^0$ disjoint from $U$ and which is compactly contained in $D_\infty$. 
Then $z^j \in \tilde{U}$ for all $j$ large. Let $\gamma^j$ be any piecewise $C^1$-path connecting $q^j=\gamma^j(0)$ and $z^j=\gamma^j(1)$. 
As we follow the trace of $\gamma^j$ starting from $z^j$, there is a last point $\alpha^j$ on the 
curve with $\alpha^j \in \partial U \cap D^j$. Let $\gamma^j(t_j) = \alpha^j$ and denote by $\sigma^j$, the sub-curve of 
$\gamma^j$ with endpoints $q^j$ and $\alpha^j$. Note that $\sigma^j$ is contained in an $\epsilon$-neighbourhood of $\partial D^j$ 
for some fixed uniform $\epsilon >0$ and for all $j $ large. It follows from (\ref{E57}) that
\begin{multline*}
\int\limits_0^1 K_{D^j} \big( \gamma^j(t), \dot{\gamma}^j(t) \big) dt  \geq \int_{t_j}^1 K_{D^j}
\big( \sigma^j(t), \dot{\sigma}^j(t) \big) dt
\gtrsim \int\limits_{t_j}^1 \frac{ \left\vert \left( D^j \Phi^{ \left( \sigma^j(t) \right) } \big(\sigma^j(t) \big) 
\big( \dot{\sigma}^j(t) \big) \right)_n \right\vert } { \delta_{D^j} \big( \sigma^j(t) \big) } dt \\
 = \int\limits_{t_j}^1 \frac{ \left\langle \nu \left( \sigma^j(t) \right), \dot{\sigma}^j(t)  \right\rangle}
{ \delta_{D^j} \big( \sigma^j(t) \big) } dt.
\end{multline*}
The fact that the last integrand is positive follows from the considerations as in the arguments following (\ref{gamlb}). To be more precise, we focus on a stretch of time
where the curve $\sigma^j$ has a non-zero component along the normal to $ \partial D^j $ at $ \pi_j \big( \sigma^j (t) \big) $ (as explained in Section \ref{pfofthm1}).
It can be checked that this stretch of time can be taken to be a non-zero constant, uniform in $ j$, since the domains $ D^j $ converge to $ D_{\infty} $. 
Furthermore, it can be checked that the last integrand is at least
\begin{alignat*}{3}
 \frac{ \Re \left\langle \nu \left( \sigma^j(t) \right), \dot{\sigma}^j(t)  \right\rangle}
{ 2 \; \delta_{D^j} \big( \sigma^j(t) \big) } =
\frac{1} { 2 \; \delta_{D^j} \big( \sigma^j(t) \big) } \; \frac{d}{dt} \delta_{D^j} \big( \sigma^j(t) \big) = \frac{1}{2} \frac{d}{dt} \log  \delta_{D^j} \big( \sigma^j(t) \big) 
\end{alignat*}
so that
\begin{alignat*}{3}
\int\limits_0^1 K_{D^j} \big( \gamma^j(t), \dot{\gamma}^j(t) \big) dt \gtrsim
 \frac{1}{2} \log \frac{1}{\delta_{D^j}(q^j)} + C
\end{alignat*}
for a uniform positive constant $C $. Finally, taking the infimum over all such $\gamma^j$, we see that
\[
d^k_{D^j} (q^j, z^j) \gtrsim \frac{1}{2} \log \frac{1}{\delta_{D^j}(q^j)} + C.
\]
Note that the left hand side here is bounded above by $R$ while the right hand sides becomes unbounded as 
$ q^j \rightarrow q^0 \in \partial D_{\infty} $. This contradiction completes the proof of the lemma.
\end{proof}

\noindent Once we are able to control the behaviour of Kobayashi balls $ B_{D^j} (z^j, R) $ as $ j 
\rightarrow \infty $, we intend to use the following comparison estimate due to K.T. Kim and 
D. Ma (\cite{KM}, \cite{KK}) to conclude the stability of 
the integrated Kobayashi distance under scaling. 

\begin{lem} \label{E29} Let $ D$ be a Kobayashi hyperbolic
domain in $ \mathbf{C}^n$ with a subdomain $ D' \subset D$. Let $
p,q \in D'$, $ d^k_{D}(p,q) = a$ and $ b > a$. If $ D'$ satisfies
the condition $ B_{D} (q,b) \subset D'$, then the following two
inequalities hold:
\begin{eqnarray*}
d_{D'}^k (p,q) \leq \frac{1}{\tanh (b-a) } d_{D}^k(p,q), \\
K_{D'} (p, v) \leq \frac{1}{\tanh (b-a) } K_{D}(p,v).
\end{eqnarray*}
\end{lem}

\noindent This statement compares the Kobayashi distance on the subdomain $ D' $ against its 
ambient domain $ D $. Recall that the estimate $ d_{D}^k \leq d_{D'}^k $ is always true. 

\medskip

\noindent The proofs of Lemmas 5.7, 5.8 and 5.9 of \cite{MV} go through verbatim in our setting, thereby, yielding the following two propositions -- which are stated here without proof.

\begin{prop} \label{E30} $ \displaystyle\lim_{j \rightarrow \infty} d^k_{D^j} ( z^j, \cdot) =  d^k_{D_{\infty}} ( z^0, \cdot) $ and $ \displaystyle\lim_{j \rightarrow \infty} d^k_{D^j} ( z^0, \cdot) =  d^k_{D_{\infty}} ( z^0, \cdot) $. Moreover, the convergence is uniform on compact sets of $ D_{\infty} $.
\end{prop}

\begin{prop} \label{E31} Fix $ R > 0$, then the sequence of domains $ B_{D^j} (z^j, R) $  converges 
in the Hausdorff sense to $ B_{D_{\infty}}(z^0, R) $. Moreover, for any $ \epsilon > 0 $ 

\begin{itemize}

\item $ B_{D_{\infty}}(z^0, R) \subset B_{D^j} (z^j, R + \epsilon),$ and

\item $ B_{D^j} (z^j, R - \epsilon)  \subset B_{D_{\infty}}(z^0, R)$

\end{itemize}
for all $j$ large.
\end{prop}

\noindent\textit{Proof of Theorem \ref{thm2}(ii):} By the biholomorphic invariance of the 
function $h$, it follows that $ h_D(p^j, \mathbb{B}^n ) = h_{D^j}(z^j, \mathbb{B}^n ) $ 
and therefore, it suffices to show that $ h_{D^j}(z^j, \mathbb{B}^n ) \rightarrow 
h_{D_{\infty}} (z^0, \mathbb{B}^n ) $. To verify this, let $ 1/R $ be a positive 
number that almost realizes $ h_{D_{\infty}} (z^0, \mathbb{B}^n ) $, i.e., 
$ 1/R < h_{D_{\infty}} (z^0, \mathbb{B}^n )  +  \epsilon $ for some $ \epsilon > 0 $ fixed.
Then there exists a biholomorphic imbedding $ F: \mathbb{B}^n \rightarrow D_{\infty} $ satisfying 
$ F(0) = z^0 $ and $ B_{D_{\infty}}(z^0, R) \subset F(\mathbb{B}^n) $. Pick $ \delta > 0 $ such
that $ B_{D_{\infty}}(z^0, R - \epsilon ) \subset F \left( B(0, 1 - \delta) \right) $. Since 
$ F\left( B(0, 1 - \delta) \right) $ is relatively compact in $ D_{\infty} $ and $ D^j \rightarrow D_{\infty} $, 
it follows that $ F\left( B(0, 1 - \delta) \right) $ is compactly contained in $ D^j $ for all large $j$.
Now, by Proposition \ref{E31}, we see that 
\[
B_{D^j} (z^j, R - 2 \epsilon)  \subset B_{D_{\infty}}(z^0, R - \epsilon), 
\]
and consequently that
\[
B_{D^j} (z^j, R - 2 \epsilon)  \subset F \left( B(0, 1 - \delta) \right) \subset D^j,
\]
which, in turn, implies that
\[
h_{D^j}(z^j, \mathbb{B}^n ) \leq \frac{1}{R - 2 \epsilon}
\]
for all $ j $ large. Therefore, by the choice of $ R $, it follows that 
\[
 \displaystyle\limsup_{j \rightarrow \infty} 
h_{D^j}(z^j, \mathbb{B}^n ) \leq h_{D_{\infty}} (z^0, \mathbb{B}^n ) .
\]
\noindent Conversely, consider a sequence of biholomorphic imbeddings $ F^j: \mathbb{B}^n \rightarrow D^j $ and positive numbers $ R^j $ with the property that $ F^j (0) = z^j$, $ B_{D^j} (z^j, R_j)  \subset F^j( \mathbb{B}^n) $ and $ 1/ R_j \leq h_{D^j}(z^j, \mathbb{B}^n ) + \epsilon $. Recall the scalings $ \Delta^{\epsilon_j}_{\zeta^j} \circ \Phi^{\zeta^j} $ associated with the sequence $ p^j $ and consider the mappings 
\[
\theta^j:= \left( \Delta^{\epsilon_j}_{\zeta^j} \circ \Phi^{\zeta^j} \right)^{-1} \circ F^j: \mathbb{B}^n \rightarrow D,
\]
and note that $ \theta^j(0) = p^j \rightarrow p^0 \in \partial D $. We claim that $ F^j $ admits a convergent subsequence.
Indeed, applying Theorem 3.11 from \cite{TT} to $ \theta^j $ assures us that the family $ \Delta^{\epsilon_j}_{\zeta^j} \circ \Phi^{\zeta^j} \circ \theta^j = F^j $ is normal and the uniform limit $ F $ is a holomorphic mapping from $ \mathbb{B}^n $ into $ D_{\infty} $. Note that by construction $ F(0) = \lim_{j \rightarrow \infty} F^j(0) = z^0 $. 
Further, by the first part of the proof and choice of the numbers $ R_j $, it follows that 
\[
1/{R_j} \leq h_{D^j}(z^j, \mathbb{B}^n ) + \epsilon \leq h_{D_{\infty}} (z^0, \mathbb{B}^n ) < + \infty 
\]
for all $ j $ large. On the other hand, we already know that, the largest possible radii admissible in the definition of Fridman's invariant function $ h_D(p^j, \mathbb{B}^n ) $ (which equals $h_{D^j}(z^j, \mathbb{B}^n )$) is at most $ 1/c $. Hence, we may assume that the sequence $ R_j $ converges to $ R_0 > 0 $. 

\medskip

\noindent The goal now is to show that $ F: \mathbb{B}^n \rightarrow D_{\infty} $ is an imbedding and that $ F( \mathbb{B}^n ) $ contains the Kobayashi ball on $ D_{\infty} $ around $ z^0 \in D_{\infty} $ with radius $ R_0 - 2 \epsilon $. To this end, it is clear that
\begin{alignat*}{3}
B_{D^j} (z^j, R_0 - \epsilon ) \subset B_{D^j} (z^j, R_j) \subset F^j( \mathbb{B}^n). 
\end{alignat*}
Moreover, by Proposition \ref{E31}, we see that
\begin{alignat*}{3}
B_{D_{\infty}} (z^0, R_0 - 2 \epsilon ) \subset B_{D^j} (z^j, R_0 - \epsilon) 
\end{alignat*}
and consequently, that $ B_{D_{\infty}} (z^0, R_0 - 2 \epsilon ) \subset F^j( \mathbb{B}^n) $ for all large $j $. It follows that $ B_{D_{\infty}} (z^0, R_0 - 2 \epsilon ) \subset F( \mathbb{B}^n) $ and in particular, that $ F $ is non-constant.

\medskip

\noindent To establish the injectivity of $ F $, consider any point $ a \in \mathbb{B}^n $. Each mapping $ F^j ( \cdot) - F^j (a) $ never vanishes in $ \mathbb{B}^n \setminus \{a\} $ because of the injectivity of $ F^j $ in $ \mathbb{B}^n $. Applying Hurwitz's theorem to the sequence $ F^j (\cdot) - F^j (a) \in \mathcal{O} \left( \mathbb{B}^n \setminus \{ a \}, \mathbb{C}^n \right) $, we have that $ F(z) \neq F(a) $ for all $ z \in \mathbb{B}^n \setminus \{a\} $. Since $ a $ is any arbitrary point of $ \mathbb{B}^n $, this exactly means that $ F $ is injective. 

\medskip

\noindent To conclude, observe that the above analysis shows that $ R_0 - 2 \epsilon $ is a candidate for the
infimum that defines $ h_{D_{\infty}} (z^0, \mathbb{B}^n )$, and hence $ h_{D_{\infty}} (z^0, \mathbb{B}^n ) \leq 1/ \left(R_0 - 2 \epsilon \right) $. This last observation, in turn, implies that 
\[
 h_{D_{\infty}} (z^0, \mathbb{B}^n ) \leq \displaystyle\liminf_{j \rightarrow \infty} h_{D^j}(z^j, \mathbb{B}^n ) 
\]
as desired. \qed

\section{A quantitative description of Kobayashi balls in terms of Euclidean parameters - Proof of Theorem \ref{thm3}}

\noindent It can be checked that 
\begin{alignat}{3} \label{E47}
Q \left(p, C_1(p, R) \delta_D(p) \right) \subset B_D(p, R) \subset Q \left(p, C_2(p, R) \delta_D(p) \right)
\end{alignat}
holds true for each $ p \in D $ and for some positive constants $ C_1(p, R), C_2(p, R) $ (which
depend on the point $ p $, the radius $ R $ and the domain $ D $). The main purpose of the Theorem \ref{thm3} is 
to show that these constants can be chosen independent of the point $p$. The proof is based on the fact that 
the integrated Kobayashi distance is stable under scaling (cf. Propositions \ref{E30} and \ref{E31}). 

\medskip

\noindent\textit{Proof of Theorem \ref{thm3}:} For $ p \in D $, let us first prove that 
$ Q \left(p, C_1 \delta_D(p) \right) \subset B_D(p, R) $ for some uniform constant $ C_1 $. Suppose that 
this is not the case. Then there are points $ p^j \in D $, $ p^0 \in \partial D $, $ p^j \rightarrow p^0 $ and 
a sequence of positive numbers $ C_j \rightarrow 0 $ with the property that -- for each $j$, the `polydisc' 
$ Q \left( p^j, C_j \delta_{D}(p^j) \right) $ is not entirely contained in the Kobayashi ball $ B_D(p^j, R) $.

\medskip

\noindent Applying a biholomorphic change of coordinates, if needed, we may assume that $ p^0 $ is the origin and 
the domain $ D $ near the origin is defined by 
\begin{alignat*}{3}
\Big \{ z \in \mathbb{C}^n : 2 \Re z_n + \sum_{l=2}^{2m} P_l(z_1) + |z_2|^2 + \cdots + |z_{n-1}|^2 
 + \sum_{\al = 2}^{n - 1} \sum_{ \substack{j + k \le m\\ j, k > 0}} \Re \Big( \big(b_{jk}^{\al} w_1^j \ov w_1^k \big) w_{\al} \Big) \Big. \\
 \Big.  + \text{ terms of higher weight}  \Big \}                                          
\end{alignat*}
as in (\ref{nrmlfrm}). Denote by $ \zeta^j $, the point on $ \partial D $ closest to $ p^j $ chosen such that 
$ \zeta^j = p^j + ('0, \epsilon_j) $. Then $ \epsilon_j \approx \delta_D(p^j) $ by construction. Furthermore, 
pick points $ q^j \in Q \left( p^j, C_j \delta_{D}(p^j) \right)$ that lie on the boundary of the 
Kobayashi ball $ B_D(p^j, R) $. The idea is to scale $ D $ with respect to the sequence $ p^j \rightarrow p^0$, 
and analyse the sets $  Q \left( p^j, C_j \delta_{D}(p^j) \right)  $ under the scalings $ \Delta^{\epsilon_j}_{\zeta^j} \circ \Phi^{\zeta^j} $. Recall from the proof of Lemma \ref{E28} that the images 
$ \Delta^{\epsilon_j}_{\zeta^j} \circ \Phi^{\zeta^j} \left( Q \left( p^j, C_j \delta_{D}(p^j) \right) 
\right) $ are contained in polydiscs centered at $ \left('0, -1/d_0(\zeta^j) \right) $, given by
\begin{alignat}{3} \label{E40}
\Delta \left( 0, K^2 \frac{\tau \left(p^j,C_j \delta_{D}(p^j)  \right) }{ \tau(\zeta^j, \epsilon_j) } 
\right) \times
\Delta \left(0, K^2  \left(\frac{C_j \delta_{D}(p^j)}{ \epsilon_j} \right)^{1/2} \right) 
\times \cdots \\
 \cdots \times 
\; \Delta \left(0, K^2  \left(\frac{C_j \delta_{D}(p^j)}{ \epsilon_j} \right)^{1/2} \right) 
\times 
\Delta \left( \frac{-1}{d_0(\zeta^j)}, \frac{K^2 C_j \delta_{D}(p^j) }{ \epsilon_j} \right), \nonumber
\end{alignat} 
where $ K > 0 $ is independent of $ j $. Among other things, S. Cho in \cite{Cho} proved that
$ \tau(\zeta^j, \epsilon_j) \approx \tau(p^j, \epsilon_j) $. But we know that $ \delta_D(p^j)
\approx \epsilon_j $ so that $ \tau \left(p^j,C_j \delta_{D}(p^j)  \right) \lesssim C_j 
\tau(\zeta^j, \epsilon_j) $. Also, $ d_0(\zeta^j) \approx 1 $ and $ C_j \rightarrow 0 $. 
These estimates show that the sets described in (\ref{E40}) are uniformly bounded, and consequently that,
$ \Delta^{\epsilon_j}_{\zeta^j} \circ \Phi^{\zeta^j} \left( Q \left( p^j, C_j \delta_{D}(p^j) \right) \right) 
$ are also uniformly bounded. In particular, the sequence $ \Delta^{\epsilon_j}_{\zeta^j} 
\circ \Phi^{\zeta^j}( q^j) $ is bounded and since $ C_j \rightarrow 0 $, we have 
$ \Delta^{\epsilon_j}_{\zeta^j} \circ \Phi^{\zeta^j}( q^j) \rightarrow ('0, -1)$.

\medskip

\noindent Observe that $ d^k_D(p^j, q^j) = R $ by construction. Since the scalings 
$ \Delta^{\epsilon_j}_{\zeta^j} \circ \Phi^{\zeta^j}: D \rightarrow D^j $ are isometries
in the Kobayashi metric on $ D $ and $ D^j $, it follows that
\begin{alignat*}{3}
d^k_{D^j} \left( \Delta^{\epsilon_j}_{\zeta^j} \circ \Phi^{\zeta^j}(p^j), 
\Delta^{\epsilon_j}_{\zeta^j} \circ \Phi^{\zeta^j}(q^j) \right) = 
d^k_D(p^j, q^j) = R
\end{alignat*}
or, equivalently that
\begin{alignat*}{3}
 d^k_{D^j} \left( ('0,-1/ d_0(\zeta^j)), \Delta^{\epsilon_j}_{\zeta^j} \circ \Phi^{\zeta^j}(q^j) \right) = R.
 \end{alignat*}
But we know that $ \Delta^{\epsilon_j}_{\zeta^j} \circ \Phi^{\zeta^j}(q^j) \rightarrow 
('0, -1) $. Hence, it follows from Proposition \ref{E30} that 
$ d^k_{D_{\infty}} \left( ('0, -1), ('0, -1) \right) = R $ which is not possible. This 
contradiction validates that there is a constant $ C_1 $ (uniform in $p$) such that
$ Q \big(p, C_1 \delta_D(p) \big) $ is contained in the Kobayashi ball $ B_D(p, R) $.

\medskip

\noindent To verify the inclusion $ B_D(p, R) \subset Q(p, C_2 \; \delta_D(p)) $ for some uniform 
constant $ C_2 $, we suppose, on the contrary, that this does not hold true. Evidently, 
in view of (\ref{E47}), 
there are points $ p^j \in D $, $ p^0 \in \partial D $, $ p^j \rightarrow p^0 $ and 
a sequence of positive numbers $ C_j \rightarrow + \infty $ such that -- for each $j$, 
$ Q \left( p^j, C_j \delta_{D}(p^j) \right) $ does not contain the Kobayashi ball $ B_D(p^j, R) $.
As before, pick points $ \zeta^j \in \partial D $ closest to $ p^j $ and define 
$ \epsilon_j, D^j $, $\Phi^{\zeta^j}, \Delta^{\epsilon_j}_{\zeta^j} $ and $ D_{\infty} $ 
analogously. Furthermore, choose $ q^j $ in the complement of the closure of 
$ Q \left( p^j, C_j \delta_D(p^j) \right) $ such that $ q^j \in B_D(p^j, R) $, so that, 
as before, 
\begin{alignat*}{3}
d^k_{D^j} \left( \Delta^{\epsilon_j}_{\zeta^j} \circ \Phi^{\zeta^j}(p^j), 
\Delta^{\epsilon_j}_{\zeta^j} \circ \Phi^{\zeta^j}(q^j) \right) = d_D(p^j, q^j) < R,
\end{alignat*}
and consequently, that
\begin{alignat*}{3}
 d^k_{D^j} \left( ('0,-1/ d_0(\zeta^j)), \Delta^{\epsilon_j}_{\zeta^j} 
 \circ \Phi^{\zeta^j}(q^j) \right) < R, 
 \end{alignat*}
which implies that
\begin{alignat*}{3}
\Delta^{\epsilon_j}_{\zeta^j} \circ \Phi^{\zeta^j}(q^j) \in B_{D^j}(z^j, R) \subset 
B_{D_{\infty}} (z^0, R + \epsilon)
\end{alignat*}
for all $j$ large. Here the last inclusion follows from Proposition \ref{E31} and 
as before, $ ('0, -1) $ and $ ('0, -1/ d_0(\zeta^j) ) $ are written as $ z^0 $ and $ z^j $ 
respectively for brevity. The above observation can be restated as 
\begin{alignat}{3} \label{E41}
 d^k_{D_{\infty}} \left(\Delta^{\epsilon_j}_{\zeta^j} \circ \Phi^{\zeta^j}(q^j), z^0 \right) < R. 
\end{alignat}
However, we claim that 
\[
 \abs{ \Delta^{\epsilon_j}_{\zeta^j} \circ \Phi^{\zeta^j}(q^j) - z^0 } 
\rightarrow + \infty ,
\]
which violates (\ref{E41}). Therefore, the theorem is 
completely proven once the claim is established. To prove the claim, recall that $ \Phi^{\zeta^j} (p^j) = 
\big('0, -\epsilon_j/d_0(\zeta^j) \big) $ and $ d_0(\zeta^j ) \approx 1 $. Therefore, 
we see that $ p^j \in Q (\zeta^j, \epsilon_j) $. Next, Proposition 3.5 of 
\cite{TT} quickly leads to the following statement: $ Q (\zeta^j, \epsilon_j) \subset Q (p^j, C \epsilon_j) $ for 
some uniform positive constant $ C $. Moreover, $ \delta_D(p^j) \approx \epsilon_j $ so that $ 
Q (\zeta^j, C C_j \delta_D(p^j)) \subset Q \left( p^j, C_j \delta_D(p^j) \right) $, where the
constant $ C $ is independent of $j$. Hence, $ q^j $ lies in the complement of $ 
Q (\zeta^j, C C_j \delta_D(p^j)) $, by construction, and therefore, the first component
\[
\abs{ \Big( \Delta^{\epsilon_j}_{\zeta^j} \circ \Phi^{\zeta^j}(q^j) \Big)_1 } \geq 
\frac{ \tau \big( \zeta^j, C C_j \delta_D(p^j)  \big)}{\tau(\zeta^j, \epsilon_j)} \gtrsim C_j \rightarrow + \infty.
\]
As a consequence, $ \abs{ \Delta^{\epsilon_j}_{\zeta^j} \circ \Phi^{\zeta^j}(q^j) - ('0, -1) } 
\rightarrow + \infty $, and hence the claim. \qed

\section{Proof of Theorem \ref{thm4}}

\noindent Suppose there exists a biholomorphism $ f $ from 
$ D_1 $ onto $ D_2 $ with the property that $ q^0 $ belongs to the cluster set of $f$ at $ p^0 $.  
To begin with, we assert that $ f $ extends as a continuous mapping to $ p^0 $. This
requires the fact that $ p^0 $ and $ q^0 $ are both \textit{plurisubharmonic barrier points}
(cf. \cite{Su}, \cite{CPS}). For the strongly pseudoconvex case, this is well known due to Fornaess 
and Sibony (\cite{FS}), and for smooth pseudoconvex finite type point $ q^0 $, the above statement
was proved in \cite{Cho}, \cite{Su}.

\medskip

\noindent Assume that both $ p^0 = 0 $ and $ q^0 = 0 $ and let $ q^j $ be a sequence of 
points in $ D_2 $ converging to $ q^0 $ along the inner normal to the origin, i.e., 
$ q^j = ('0, - \delta_j) $, each $ \delta_j > 0 $ and $ \delta_j \searrow 0 $. Since 
$ f : D_1 \rightarrow D_2 $ is a biholomorphism and $ 0 \in cl_f(0) $, there exists a 
sequence $ p^j \in D_1 $ with $ p^j \rightarrow 0 $ such that $ f(p^j) = q^j $. Now, 
scale $ D_1 $ with respect to $ \{ p^j \} $ and $ D_2 $ with respect to $ \{q^j \} $.

\medskip

\noindent To scale $ D_1 $, recall that by \cite{Pi}, for each $ \xi $ near $ p^0 \in \partial D_1 $, 
there is a unique automorphism $ h^{\xi} $ of $ \mathbb{C}^n $  with $ h^{\xi}(\xi) = 0 $
such that the domain $ h^{\xi} ( D_1 ) $ is given by
\[ 
\big \{ z \in \mathbf{C}^n: 2 \Re \big( z_2 + K^{\xi}(z)\big) + H^{\xi}(z) +
\alpha^{\xi}(z) < 0 \big \}
\] 
where 
$ K^{\xi}(z) = \displaystyle\sum_{i,j=1} ^n a_{ij} (\xi) z_i z_j, H^{\xi}(z)
= \displaystyle\sum_{i,j=1} ^n b_{ij} (\xi) z_i \bar{z_j}$ and $
\alpha^{\xi}(z)= o (|z|^2) $ with $ K^{\xi}(z_1,0) \equiv 0 $
and $ H^{\xi}(z_1,0) \equiv |z_1|^2 $. The automorphisms $ h^{\xi} $ converge to the identity uniformly
on compact subsets of $ \mathbb{C}^n $ as $ \xi \rightarrow p^0 $. For $ \xi = 
(\xi_1, \xi_2, \ldots, \xi_n) \in D_1 $ as above, consider the point 
$ \tilde{\xi} = ( \xi_1, \xi_2, \ldots, \xi_{n-1}, \xi_n + \epsilon) $ where $ \epsilon > 0 $ 
is chosen to ensure that $ \tilde{\xi} \in \partial D_1 $. Then the actual 
form of $ h^{\xi} $ shows that $ h^{ \tilde{\xi}} (\xi) = ('0, - \epsilon) $.

\medskip

\noindent In order to apply Pinchuk's scalings to the sequence $ p^j \rightarrow 
p^0 \in \partial D_1 $, choose $ \xi^j \in \partial D_1 $ such that if 
$ p^j = ('p^j, p^j_n) $, then $ \xi^j = 
('p^j, p^j_n + \epsilon_j) \in \partial D_1 $ for some $ \epsilon_j > 0 $. Then 
$ \epsilon_j \approx \delta_{D_1}(p^j) $ by construction. Now, define the dilations
\[
T^j ( z_1, z_2, \ldots, z_n ) = \left( { {\epsilon_j}^{-\frac{1}{2}}{z_1} } , \ldots,
 { \epsilon_j}^{-\frac{1}{2}} {z_{n-1}}, { \epsilon_j}^{-1} {z_n
 } \right)
\]
and the dilated domains $ D_1^j = T^j \circ h^{\xi^j} ( D_1 ) $. It was shown 
in \cite{Pi} that $ D_1^j $ converge to 
\[
 D_{1, \infty} = \left \{ z \in \mathbb{C}^n : 2 \Re z_n + | z_1 |^2 + 
 | z_2 |^2 + \cdots + |z_{n-1} |^2 < 0
\right \}
\]
which is the unbounded realization of the unit ball in $ \mathbb{C}^n $. 

\medskip

\noindent For clarity and completeness, we briefly describe the scalings for the domain 
$ D_2 $, which are simpler this time as the sequence $ q^j $ approaches $ q^0 $ normally. 
To start with, consider $ \Delta^j : \mathbb{C}^n \rightarrow \mathbb{C}^n $, a sequence
of dilations, defined by
\[
\Delta^j( w_1, w_2, \ldots, w_{n-1}, w_n ) = \left( \delta_j^{- \frac{1}{2m}} w_1, \delta_j^{-\frac{1}{2}} w_2, 
\cdots, \delta_j^{-\frac{1}{2}} w_{n-1}, \delta_j^{-1} w_n \right).
\]
Note that $ \Delta^j ('0, - \delta_j) = ('0, -1) $ for all $j$ and the domains $ D_1^j = \Delta^j (D_1) $ 
converge in the Hausdorff sense to
\[
D_{2, \infty} = \big\{ w \in \mbb C^n : 2 \Re w_n  + Q_{2m}(w_1, \overline{w}_1) + \abs{w_2}^2 + \cdots + 
\abs{w_{n-1}}^2 < 0 \big\},
\]
where $ Q_{2m} $ is the homogeneous polynomial of degree $ 2m $ that coincides with the polynomial
of same degree in the homogeneous Taylor expansion of the defining function for $ \partial D_2 $ 
near the origin.

\medskip

\noindent By the biholomorphic invariance of the function $h$, it follows that $ h_{D_1}
( p^j, \mathbb{B}^n ) = h_{D_2} (q^j, \mathbb{B}^n ) $. Then Theorem \ref{thm2} assures us 
that the right hand side above converges to $ h_{D_{2, \infty}} \big( ('0,-1), \mathbb{B}^n \big) $. 
Furthermore, since $ \partial D_1$ is strongly pseudoconvex near $ p^0 $, 
by Theorem \ref{E42}, we see that the left hand side above $ h_{D_1} ( p^j, \mathbb{B}^n ) \rightarrow 0 $
as $ j \rightarrow \infty $. It follows that $ h_{D_{2, \infty}} \big( ('0,-1), \mathbb{B}^n \big) = 0 $. 
As a result, $ D_{2, \infty} $ is biholomorphic to $ \mathbb{B}^n $ by virtue of Proposition \ref{E43}.
Therefore, the problem has been quickly reduced to investigation for the special case of model domains,
namely, $ D_{2,\infty} $ and $ \mathbb{B}^n $, which are algebraic.

\medskip

\noindent By composing with a suitable Cayley transform, if necessary, we may assume that there is a 
biholomorphism $ F $ from $ D_{1, \infty}$ (which is biholomorphic to $ \mathbb{B}^n $) onto 
$ D_{2, \infty} $ with the additional property that the cluster set of $ F^{-1} $ at some point
$ ('0, \iota a) \in \partial D_{2, \infty} $ (for $a \in \mathbb{R} $) contains a finite 
point of $ \partial D_{1, \infty} $. Then Theorem 2.1 of \cite{CP} tells us that $ F^{-1} $ extends 
holomorphically past the boundary to a neighbourhood  of $ ('0, \iota a) $. It turns out that
$ F^{-1} $ extends biholomorphically across some point $ ('0, \iota a^0 ) \in \partial D_{2, \infty} $.
To prove this claim, it suffices to show that the Jacobian of $ F^{-1} $ does not vanish
identically on the complex plane 
\[
L = \big \{ ('0, \iota a) : a \in \mathbb{R} \big \} \subset \partial D_{2, \infty}.
\]
If the claim were false, then the Jacobian of $ F^{-1} $ vanishes on the entire $ w_n $-axis, 
which intersects the domain $ D_{2, \infty} $. However, $ F^{-1} $ is injective on $ D_{2, \infty} $,
and consequently, has nowhere vanishing Jacobian determinant on $ D_{2, \infty} $. This contradiction 
proves the claim.

\medskip

\noindent Furthermore, it is evident that the translations in the imaginary $ w_n $-direction 
leave $ D_{2, \infty} $ invariant. Therefore, we may assume that $ ('0, \iota a^0) $ is the origin and 
that $ F $ preserves the origin. Now recall that the Levi form is preserved under local biholomorphisms 
around a boundary point, thereby yielding the strong pseudoconvexity of $ \partial D_{2, \infty} $.
In particular, $ Q_{2m}(w_1, \overline{w}_1) = |w_1|^2 $ which gives the strong pseudoconvexity 
of $ q^0 \in \partial D_2 $. This contradicts the assumption that the Levi form of $ \partial D_2 $
has rank exactly $ n - 2 $ at $ q^0 $. Hence the result. \qed

\section{Continuous extension of isometries -- Proof of Theorem \ref{thm5}}

\noindent The proof of Theorem \ref{thm5} will be accomplished in several steps. The first step is to analyse the behaviour of the Kobayashi 
metric on a smoothly bounded pseudoconvex Levi corank one domain.

\begin{prop} \label{E20}
Let $D $ be a bounded domain in $ \mathbb{C}^n $. Assume that $ \pa D $ is smooth pseudoconvex and of 
finite type near a point $ p^0  \in \partial D $. Suppose further that the Levi form of $\pa D $ has rank at least $n-2$ near $p^0$. 
Then for any $ \epsilon > 0 $, there exist positive numbers $ r_2 < r_1 < \epsilon $ and $ C $ (where $ r_1, r_2 $ and $ C $ depend on $ A $) such that
\begin{alignat*}{3}
d_D^k( A,B) \geq \frac{1}{2} \log \frac{1}{\delta_D(B)} - C, \qquad A \in D \setminus B(p^0, r_1), B \in B(p^0, r_2) \cap D.
\end{alignat*}
\end{prop}

\begin{proof} By Theorem 3.10 of \cite{TT}, there exists a neighbourhood $ U $ of $ p^0 $ in $ \mathbb{C}^n $ such that
\begin{alignat}{3} \label{E19}
K_D(z, v) \approx \frac{(D \Phi^z(z) v)_1}{\tau \left( z, \delta_D(z) \right)} + \frac{(D \Phi^z(z) v)_2}{ \left
( \delta_D(z) \right)^{1/2} } + \cdots + \frac{(D \Phi^z(z) v)_{n-1}}{ \left( \delta_D(z) \right)^{1/2}} + \frac{(D \Phi^z(z) v)_n}{\delta_D(z)}
\end{alignat}
for all $ z \in U \cap D $ and $ v $ a tangent vector at $ z $. The neighbourhood $ U $ is so chosen to avoid the point $A $.
Let $ \gamma $ be a piecewise $ C^1$-smooth 
curve in $ D $ joining $ A $ and $ B $, i.e., $ \gamma (0) = A $ and $ \gamma (0) = B $. As we travel along
 $ \gamma $ starting from $ A$, there is a last point  $ \alpha $ on the trace $ \gamma $ with $ \alpha 
\in \partial U \cap D $. Let $ \gamma (t_0) = \alpha $ and denote by $ \sigma $, the subcurve of $ 
\gamma $ with end-points $ \alpha $ and $ B $. Observe that the trace $ \sigma $ is contained in an $ 
\delta $-neighbourhood of $ \partial D $ for some fixed $ \delta > 0 $. Here we choose $\delta > 0$ in such a way that the $\delta$-neighbourhood of $\partial D$ does not 
contain the point $A$. Evidently,
\begin{alignat*}{3}
\int_0^1 K_D \left( \gamma(t), \dot{\gamma}(t) \right) dt \geq \int_{t_0}^1 K_D \left( \sigma(t), \dot{\sigma}(t) \right) dt.
\end{alignat*}
From this point, proceeding exactly as in the proof of Lemma \ref{E28}, it follows using (\ref{E19}) that
\begin{alignat*}{3}
\int_0^1 K_D \left( \gamma(t), \dot{\gamma}(t) \right) dt \gtrsim \frac{1}{2} \log \frac{1}{\delta_D(B)} - C
\end{alignat*}
for some positive constant $ C $. Now, taking the infimum over all such paths $ \gamma $ yields
\begin{alignat*}{3}
d_D^k( A,B) \geq \frac{1}{2} \log \frac{1}{\delta_D(B)} - C
\end{alignat*}
as desired.
\end{proof}

\noindent Finally, we recall a localization result (Lemma 4.3 of \cite{M}) for the Kobayashi balls. 
The proof relies on the localization property of the Kobayashi metric near holomorphic peak points (see \cite{G}, \cite{Ro}).

\begin{lem} \label{E22}
Let $ D \subset \mathbb{C}^n $ be a bounded domain and $ p^0 \in \partial D $  be a holomorphic peak point. 
Then for any fixed $ R > 0 $ and every neighbourhood $ U $ of $ p^0 $, there exists a smaller neighbourhood $ V \subset U $ of $ p^0 $ with $ V $ relatively compact in $ U $ such that for all $ z \in V \cap D $ we have
\[
B_D(z, c R ) \subset B_{U \cap D}(z, R ) \subset B_D(z, R )
\]
where $ c $ is a positive constant independent of $ z \in V \cap D $.
\end{lem}

\noindent\textit{Proof of Theorem \ref{thm5}:} The proof involves two steps. 
The first step is to show that $ f $ extends to $ D_1 \cup \{ p^0 \} $ as a 
continuous mapping. Once this claim is established, the second step is to 
verify that $ f $ is continuous on a neighbourhood of $ p^0 $ in $ \overline{D}_1 $.

\medskip

\noindent To prove the first claim, assume, on the contrary, that $ f $ does 
not extend continuously to $ p^0 $. It follows that there is a sequence of points $
s^j $ in $ D_1 $ with $ s^j \rightarrow p^0 \in \partial D_1 $ such
that the sequence $ f(s^j) $ does not converge to $ q^0 \in \partial D_2 $. 
Moreover, by hypothesis, there exists a sequence $ p^j \in D_1 $ with $ p^j 
\rightarrow p^0 $ and $ f(p^j) \rightarrow q^0 $. Consider polygonal paths $ \gamma ^j $ in $ D_1 $
joining $ p^j $ and $ s^j $ defined as follows: for each $j$,
choose $ p^{j0}, s^{j0} $ on $ \partial D_1 $ closest to $ p^j $ and
$ s^j $ respectively. Set $ p^{j'} = p^j - |p^{j}- s^{j} |
\nu(p^{j0})$ and $ s^{j'}= s^j - |p^j - s^j| \nu(s^{j0})$ where $ \nu(z)
$ denotes the outward unit normal to $ \partial D_1 $ at $ z \in
\partial D_1 $. Define $ \gamma^j = \gamma_1^j \cup \gamma_2^j \cup \gamma_3^j$ as the union of three segments, where

\begin{itemize}

\item $ \gamma_1^j $ is the straight line path joining $p^{j}$ and
$p^{j'}$ along the inner normal to $ \partial D_1 $ at the point
$p^{j0}$,

\item $ \gamma_2^j $ is the straight line joining
$p^{j'}$ and $ s^{j'}$, and 

\item $ \gamma_3^j $ is the straight line segment joining $ s^{j'} $ 
and $ s^{j} $ along the inward normal to $ \partial D_1 $ at the point $ s^{j0} $.

\end{itemize}
Evidently, the composition $ f \circ \gamma ^j $ yields a continuous path in 
$ D_2 $ joining $ f(p^j) $ and $ f(s^j) $. Fix $ \epsilon > 0 $ and for each $ j $, 
pick points $ u^j \in \partial {B}( q^0, \epsilon) \cap U_2 $
on the trace$(f \circ \gamma ^j)$. Further, let $ t^j \in D_1 $ be such that $
f(t^j) = u^j $. Observe that the points $ t^j $ lie on $\mbox{trace}(\gamma^j) $ and hence
$ t^j \rightarrow p^0 $ by construction. It is straightforward to check that the sequence 
$ f(t^j) = u^j $ converges to a point $ u^0 \in U_2 \cap \partial D_2 $, where $ u^0 $ is
necessarily different from $ q^0 $. Since strong pseudoconvexity is an open condition, we may assume
that a tiny neighbourhood $ \Gamma' \subset \partial D_2 $ of 
$ q^0 $ is strongly pseudoconvex. In particular, 
the same holds true for $ u^0 $ (this can be arranged by choosing $ \epsilon $ small so 
that $ u^0 \in \Gamma' $). Now, using the upper estimate (\cite{FR}) on the Kobayashi 
distance between two points converging to the same boundary point, we see that 
\begin{alignat}{3} \label{E23} 
d^k_{D_1}(p^j,t^j)  \leq  \frac{1}{2} \log \left( \delta_{D_1}(p^j)  + |p^j - t^j| \right) 
 + \frac{1}{2} \log \left( \delta_{D_1}(t^j) + |p^j - t^j| \right) \\ \nonumber
+ \frac{1}{2} \log { \frac{1}{\delta_{D_1}(p^j)}} + \frac{1}{2} \log { \frac{1}{\delta_{D_1}(t^j)} } +
 C_1
\end{alignat} 
for all $j$ large and uniform positive constant $ C_1 $. On the other hand, Corollary 2.4 of \cite{FR} implies that
\begin{equation} \label{E24}
d^k_{D_2} \big( f(p^j), f(t^j) \big) \geq \frac{1}{2} \log \frac{1}{ \delta_{D_2} 
\left( f(p^j) \right) } + \frac{1}{2} \log \frac{1}{ \delta_{D_2} 
\left( f(t^j) \right) } - C_2
\end{equation}
for all $j$ large and uniform positive constant $ C_2 $. But we know that $
d^k_{D_1}(p^j,t^j) = d^k_{D_2} \big( f(p^j), f(t^j) \big) $ for all $j$.

\medskip

\noindent \textbf{Assertion:} $ \delta_{D_2} \left( f(p^j) \right)
\leq C_3 \delta_{D_2} ( p^j ) $ and $ \delta_{D_2} \left( f(t^j) \right) \leq 
C_3 \delta_{D_2}( t^j) $ for some uniform positive
constant $ C_3 $.

\medskip

\noindent If the assertion were true, then, comparing the
inequalities (\ref{E23}) and (\ref{E24}) together with the above assertion 
gives that for all $j$ large
\[
- ( C_1 + C_2 + \log C_3 ) \leq  \frac{1}{2} \log \left( \delta_{D_1}(p^j) 
+ |p^j - t^j| \right) + \frac{1}{2} \log \left( \delta_{D_1}(t^j) +
|p^j - t^j| \right)
\]
which is impossible. This contradiction proves the claim.

\medskip

\noindent To establish the assertion, fix $ A
\in D_1 $. Now, by Proposition \ref{E20}, it follows that
\begin{equation} \label{E25}
d^k_{D_1} (p^j, A) \geq  \frac{1}{2} \log \frac{1} {\delta_{D_1}(p^j)} - C_4
\end{equation}
for some uniform positive constant $ C_4 $. On the other hand, the following upper estimate
for the Kobayashi distance is well known: 
\begin{equation} \label{E26}
d^k_{D_2} \left( f(p^j), f(A) \right) \leq \frac{1}{2} \log \frac{1} {\delta_{D_2} 
\left( f(p^j) \right)} + C_5
\end{equation}
for all $j$ large and a uniform constant $ C_5 > 0 $. Use the fact that
$  d^k_{D_1} ( p^j, A ) = d^k_{D_2} \left( f(p^j), f(A) \right) $, and 
compare the inequalities (\ref{E25}) and (\ref{E26}) to conclude that 
the assertion made above holds true. Hence, the first claim is completely proven. 

\medskip

\noindent Now, let $ a^j $ be a sequence of points in $ U_1 \cap D_1 $  
with $ a^j \rightarrow a^0 \in U_1 \cap \partial D_1 $. The goal now
is to show that $f$ extends continuously to the point $ a^0 $. To see this, pick 
$ a' \in U_1 \cap D_1 $ such that $ f(a') \in U_2 \cap D_2 $. This 
can be achieved using the continuity of the mapping $f$. There are 
two cases to be considered. After passing to a subsequence, if needed, 

\begin{itemize}

\item [(i)] $ f(a^j ) \rightarrow b^0 \in U_2 \cap \partial D_2 $,

\item [(ii)] $ f(a^j ) \rightarrow b^1 \in U_2 \cap D_2 $ as $ j 
\rightarrow \infty $.

\end{itemize}

\noindent In case (ii), observe that the quantity $ d^k_{U_2 \cap D_2}
\big( f(a^j), f(a') \big) $ is uniformly bounded (say by $ R$) because of the
completeness of $ U_2 \cap D_2 $. Therefore, for all $j$ large
\[
d^k_{D_2} \big( f(a^j), f(a') \big) \leq d^k_{U_2 \cap D_2} \big( 
f(a^j), f(a') \big) < R.
\]
Using the fact $ d^k_{D_1} ( a^j, a') = d^k_{ D_2} \big( 
f(a^j), f(a') \big) $, we see that $ a' \in B_{D_1} (a^j, R) $. 
Applying Lemma \ref{E22} gives that $ a' \in B_{U_1 \cap D_1} ( a^j, 
R/c) $ for some uniform constant $c$. This exactly means that 
\[
d^k_{U_1 \cap D_1} (a^j, a') < R/c.
\]
This is however a contradiction as $ d^k_{U_1 \cap D_1} (a^j, a') $
remains unbounded because of the completeness of $ U_1 \cap D_1 $.
Hence, the sequence $ f(a^j ) \rightarrow b^0 \in U_2 \cap \partial 
D_2 $ and consequently $ b^0 $ belongs to the cluster set of $ a^0 $
under $f$. As before, we may assume that $ \partial D_2 $ is strongly pseudoconvex at
 $ b^0 $. From this point, proceeding exactly as in the first part 
of the proof yields that $f$ extends continuously to the point $ a^0 $. 
Since $ a^0 \in U_1 \cap \partial D_1 $ was arbitrary, Theorem
\ref{thm5} is completely proven. \qed

\section{Kobayashi metric of a complex ellipsoid in $\mathbb{C}^n $ -- Proof of Theorem \ref{thm6}}

\noindent Let $ D $ be a domain in $ \mathbb{C}^n $, $z \in D $ and $ v \in \mathbb{C}^n $ a 
tangent vector at the point $z$. Recall that a mapping $ \phi \in \mathcal{O}(\Delta, D) $ 
is said to be a complex geodesic for $ (z, v) $ if $ \phi(0) = z $ and $ K_D (z,v) \phi '(0) = v $. A complex ellipsoid is a domain of the form 
\begin{equation*}
E (2m_1, \ldots, 2m_n) = \big\{ z \in \mathbb{C}^n : |z_1|^{2m_1} + \cdots + |z_n|^{2m_n} < 1 \big\}
\end{equation*}
where $ m_j > 0 $ for each $ j = 1, \ldots, n $. It is well known that complex ellipsoids are  convex if and only if $ m_j \geq 1/2 $ for $ j=1, \ldots, n $. Moreover, they are taut domains (i.e., $ \mathcal{O} ( \Delta, E(2m_1, \ldots, 2m_n) ) $ is a normal family) and hence, there always exist complex geodesics through a given point $ z \in E(2m_1, \ldots, 2m_n) $ and any tangent vector at the point $z$. The primary goal of this section is to describe the Kobayashi metric on the complex ellipsoids for the case $ m_1 \geq 1/2 $ and $ m_2 = \cdots = m_n = 1 $ -- notice that these are exactly the domains $ E_{2m} $ introduced in (1.1).  To understand the Kobayashi metric on $ E_{2m} $, we use the characterisation of all complex geodesics $ \phi : \Delta \rightarrow E (2m_1, \ldots, 2m_n) $, each $ m_j \geq 1/2 $, due to Jarnicki, Pflug and Zeinstra (\cite{JPZ}). Observe that it suffices to consider only those complex geodesics $ \phi = ( \phi_1, \ldots, \phi_n) : \Delta \rightarrow \mathbb{C}^n $  for which 
\begin{equation} \label{E1}
\phi_j \; \mbox{is not identically zero for any} \; j = 1, \ldots, n. 
\end{equation}
After a suitable permutation of variables, we may assume that for some $ 0 \leq s \leq n $, 
\begin{alignat}{2} \label{E2}
\left\{ \begin{array}{lrl}
0 \notin \phi_j(\Delta) \; \mbox{for} \; j = 1, \ldots, s \; \mbox{and} \\
0 \in \phi_j(\Delta) \; \mbox{for} \; j = s + 1, \ldots, n.
\end{array}
\right.
\end{alignat} 
The main result of (\cite{JPZ}) that is needed is:

\begin{thm} \label{E3}
A non-constant mapping $ \phi= (\phi_1, \ldots, \phi_n) : \Delta \rightarrow \mathbb{C}^n $ with (\ref{E1}) and 
(\ref{E2}) is a complex geodesic in $ E (2m_1, \ldots, 2m_n) $ if and only if $ \phi $ is of the form 
\begin{alignat*}{4}
\phi_j(\lambda) = 
\left\{ \begin{array}{lrl}
a_j \left( \frac{1 - \bar{\alpha}_j \lambda}{ 1 - \bar{\alpha}_0 \lambda }
\right)^{1/m_j} & \mbox{for} & j=1, \ldots, s, 
\\
a_j \left( \frac{\lambda - \alpha_j}{1 - \bar{\alpha}_j \lambda} \right)  \left(\frac{1 - \bar{\alpha}_j \lambda}{1 - \bar{\alpha}_0 \lambda }\right)^{1/m_j}
& \mbox{for} & j= s+1, \ldots, n,
\end{array}
\right.
\end{alignat*}
where
\begin{alignat*}{3}
                             a_j & \in \mathbb{C} \setminus \{0\} \; \mbox{for} \; j = 1, \ldots, n, \\
      \alpha_1, \ldots, \alpha_s & \in \bar{\Delta}, \alpha_0, \alpha_{s + 1}, \ldots, \alpha_n \in \Delta, \\
                        \alpha_0 & = \sum_{j=1}^n |a_j|^{2 m_j} \alpha_j, \\
                1 + |\alpha_0|^2 & = \sum_{j=1}^n |a_j|^{2 m_j} ( 1 + |\alpha_j|^2 ), \\
            \mbox{the case} \; s & = 0, \alpha_0 = \alpha_1 = \cdots = \alpha_n \; \mbox{is excluded and}, \\
  \mbox{the branches of } & \mbox{powers are such that } 1^{1/m_j} = 1, j = 1, \ldots, n.
\end{alignat*}
\end{thm}

\noindent\textit{Proof of Theorem \ref{thm6}:} We proceed by induction on the index $n$. First, note that the case $ n = 2 $ is Theorem 2 of \cite{BFKKMP}. Next, assume that the result holds for all integers between $ 3 $ and $ n-1 $. To prove the inductive step, fix $ ( p, 0, \ldots, 0 ) \in E_{2m} $ and $ (v_1, \ldots, v_n )  \in \mathbb{C}^n $. The main objective is to find an effective formula for $ \tau = K_{E_{2m}} \left( ( p, 0, \ldots, 0 ), (v_1, \ldots, v_n ) \right) $. It is well known (see, for example, Proposition 2.2.1 in \cite{JP}) that if $ p = 0 $, then
\[
\tau = K_{E_{2m}} \left( (0, 0, \ldots, 0), (v_1, \ldots, v_n) \right) = q_{E_{2m}} (v_1, \ldots, v_n ) 
\]
where $ q_{E_{2m}} $ denotes the \textit{Minkowski functional} of $ E_{2m} $. Furthermore, $ 1/ q_{E_{2m}} (v_1, \ldots, v_n ) $ is the only positive solution of the equation
\[
\frac{|v_1|^{2m} }{ \left( q_{E_{2m}} (v_1, \ldots, v_n ) \right)^{2m} } + \frac{|v_2|^{2} }{ \left( q_{E_{2m}} (v_1, \ldots, v_n ) \right)^{2} } + \cdots + \frac{|v_n|^{2} }{ \left( q_{E_{2m}} (v_1, \ldots, v_n ) \right)^{2} } = 1.
\]
For $ 0 < p < 1 $, if $ \hat{v} = ( v_2, \ldots, v_n ) = \hat{0} $, then
\[
\tau = K_{E_{2m}} \left( ( p, 0, \ldots, 0 ), (v_1, \ldots, 0 ) \right) = K_{\Delta} (p, v_1) = \frac{|v_1|}{1- p^2}.
\]
Hence, we may assume that $ \hat{v} \neq \hat{0} $ in the sequel.

\medskip

\noindent Let $ \phi : \Delta \rightarrow E_{2m} $ be a complex geodesic with $ \phi(0) = (p, \hat{0}) $ and $
\tau \phi'(0) = (v_1, \ldots, v_n ) $. Evidently, $ \phi_1 $ is not identically zero and $ \phi_2, \ldots, \phi_{n} $ cannot be identically zero simultaneously. Suppose that $ \phi_4 = \cdots = \phi_n \equiv 0 $, then the mapping $ \tilde{\phi} = ( \phi_1, \phi_2, \phi_3) : \Delta \rightarrow E(2m, 2, 2) \subset \mathbb{C}^3 $ is a complex geodesic through the point $ (p, 0, 0) $ in $ E(2m,2, 2) $ and consequently,
\[
\tau = K_{E_{2m}} \left( ( p, 0, \ldots, 0 ), (v_1, \ldots, v_n ) \right) = K_{E(2m, 2, 2) } \left( (p,0,0), (v_1,v_2, v_3) \right).
\]
The right hand side above is known explicitly by induction hypothesis, thereby, yielding an explicit formula for $ \tau $.
The above analysis shows that we may assume that $ \phi_j $ is not identically zero for any $j = 1, \ldots, n $. This assumption does not restrict the generality, since mappings with zero-components are exactly \textit{lower dimensional} complex geodesics, as observed above.
Then, $ \phi $ satisfies (\ref{E1}) and (\ref{E2}) (with $ s=0 $ or $ s = 1 $). 

\medskip

\noindent Consider the case $ s = 1 $ first. Applying Theorem \ref{E3} gives that
\begin{alignat*}{3}
\phi_1(\lambda) & = a_1 \left( \frac{1 - \bar{\alpha}_1 \lambda}{ 1 - \bar{\alpha}_0 \lambda }
\right)^{1/m}, \\
\phi_j(\lambda) & = a_j \left( \frac{\lambda - \alpha_j}{1 - \bar{\alpha}_j \lambda} \right)  \left(\frac{1 - \bar{\alpha}_j \lambda}{1 - \bar{\alpha}_0 \lambda }\right)
 \mbox{ for }  j= 2 , \ldots, n,
\end{alignat*}
where $a_j, \alpha_j $ are as stated in Theorem \ref{E3}. It follows that
\begin{alignat*}{3} 
\phi(0) & = ( a_1, - a_2 \alpha_2, \ldots, - a_n \alpha_n ) \; \mbox{and } \nonumber \\
\phi_1'(0) & = a_1 ( \bar{\alpha}_0 - \bar{\alpha}_1 )/m, \\ 
\phi_j'(0) & = a_j \left( (1 - | \alpha_j |^2 ) - \alpha_j ( \bar{\alpha}_0 - \bar{\alpha}_j ) \right)  \mbox{ for }  j= 2 , \ldots, n. \nonumber
\end{alignat*}
But we know that $ \phi(0) = (p, \hat{0}) $ and $
\tau \phi'(0) = (v_1, \ldots, v_n ) $. Therefore, 
\begin{eqnarray*}
a_1 = p, \; \alpha_2 = \cdots = \alpha_n = 0, \; \mbox{and} \\
\tau p ( \bar{\alpha}_0 - \bar{\alpha}_1) /m = v_1, \tau a_2 = v_2, \ldots, \tau a_n = v_n.
\end{eqnarray*}
These conditions, in turn, imply that 
\begin{alignat*}{3}
\alpha_0 & = p^{2m} \alpha_1, \mbox{ and} \\
1 + | \alpha_0|^2 & = p^{2m} \left( 1 + | \alpha_1|^2 \right) + |a_2|^2 + \cdots + |a_n|^2,
\end{alignat*}
and consequently,
\begin{alignat*}{3}
\tau  p \bar{\alpha}_1 (p^{2m} - 1 ) /m & = v_1, \mbox{ and} \\
1 + p^{4m} |\alpha_1|^2  & = p^{2m} \left( 1 + | \alpha_1|^2 \right) + |v_2|^2 / \tau ^2 + \cdots + |v_n |^2 / \tau^2,
\end{alignat*}
or equivalently that
\begin{eqnarray} \label{E10}
\tau = \left( \frac{m^2 p^{2m-2} |v_1|^2}{(1-p^{2m})^2}
+ \frac{|v_2|^2}{1-p^{2m}} + \cdots + \frac{|v_n|^2}{1-p^{2m}} \right)^{1/2} \mbox{ for } \; \frac{m^2 |v_1|^2}{|v_2|^2 + \cdots + |v_n|^2} \leq p^2.
\end{eqnarray}

\noindent For the second case, when $ s = 0 $, let $ u $ and $ t $ be parameters as defined by equations (\ref{E4}) and (\ref{E5}) respectively. Observe that $ u > p $ and the parameter $t$ is a solution of 
\begin{equation} \label{E12}
(m-1)^2 p^2 t^2 - \left( u^2 + 2m(m-1) p^2 \right) t + m^2 p^2 = 0,
\end{equation}
and hence, satisfies $ 0 < t < 1 $. As before, Theorem \ref{E3} applied once again gives 
\begin{alignat*}{3}
\phi_1(\lambda) & = a_1 \left( \frac{\lambda - \alpha_1}{1 - \bar{\alpha}_1 \lambda} \right) \left( \frac{1 - \bar{\alpha}_1 \lambda}{ 1 - \bar{\alpha}_0 \lambda }
\right)^{1/m}, \\
\phi_j(\lambda) & = a_j \left( \frac{\lambda - \alpha_j}{1 - \bar{\alpha}_j \lambda} \right)  \left(\frac{1 - \bar{\alpha}_j \lambda}{1 - \bar{\alpha}_0 \lambda }\right)
 \mbox{ for }  j= 2 , \ldots, n,
\end{alignat*}
where $ a_j, \alpha_j $ satisfy the conditions listed in Theorem \ref{E3}. It is immediate that
\begin{alignat*}{3} 
\phi(0) & = ( - a_1 \alpha_1, \ldots, - a_n \alpha_n ) \; \mbox{and } \nonumber \\
\phi_1'(0) & = a_1 \left( (1 - | \alpha_1 |^2 ) - \frac{\alpha_1}{m} ( \bar{\alpha}_0 - \bar{\alpha}_1 ) \right) , \\ 
\phi_j'(0) & = a_j \left( (1 - | \alpha_j |^2 ) - \alpha_j ( \bar{\alpha}_0 - \bar{\alpha}_j ) \right)  \mbox{ for }  j= 2 , \ldots, n. \nonumber
\end{alignat*}
It follows that
\begin{eqnarray*}
- a_1 \alpha_1 = p, \; \alpha_2 = \cdots = \alpha_n = 0, \; \mbox{and} \\
\tau a_1 \left( 1 - |\alpha_1|^2  - \frac{\alpha_1}{m} ( \bar{\alpha}_0 - \bar{\alpha}_1) \right) = v_1, \tau a_2 = v_2, \ldots, \tau a_n = v_n.
\end{eqnarray*}
As a consequence,
\begin{alignat*}{3}
\alpha_0 & = |a_1|^{2m} \alpha_1, \mbox{ and} \\
1 + | \alpha_0 |^2 & = |a_1|^{2m} \left( 1 + | \alpha_1|^2 \right) + |a_2|^2 + \cdots + |a_n|^2,
\end{alignat*}
so that
\begin{alignat}{3} \label{E7}
\frac{ |v_1| } { \tau }  = \frac{ p} { |\alpha_1| } \left( 1 - |\alpha_1|^2 - \frac{ |\alpha_1|^2 }{m} \Big( \frac{p^{2m}} { {|\alpha_1|}^{2m} } - 1 \Big) \right), \mbox{ and} \\ 
1 + \frac{p^{4m}}{ |\alpha_1|^{4m-2} } - \frac{p^{2m}}{ |\alpha_1|^{2m} }  \left( 1 + |\alpha_1|^2 \right)  = \frac{ |v_2 |^2 }{ {\tau}^2 } + \cdots + \frac{ |v_n|^2 }{ {\tau}^2 }. \label{E8}
\end{alignat}
Writing $ |\alpha_1| = \alpha $, the goal now is to solve the equations (\ref{E7}) and (\ref{E8}) for $ \alpha $. To achieve this, first eliminate $ \tau $ from the above two equations, so that
\begin{equation*}
\frac { \left( (m-1) {\alpha}^{2m} - m {\alpha}^{2m -2} + p^{2m}   \right) ^2 } { {\alpha}^{4m-2} - p^{2m} {\alpha}^{2m} - p^{2m} {\alpha}^{2m-2} + p^{4m} } =  \frac{ m^2 | v_1 |^2 }{ p^2  \left( |v_2|^2 + \cdots + |v_n|^2 \right) }.
\end{equation*}
But the right hand side above is exactly the quotient $ u^2 / {p^2 } $ and hence, 
\begin{eqnarray*}
p^2 \left( (m-1) {\alpha}^{2m} - m {\alpha}^{2m -2} + p^{2m}   \right) ^2  = \\
u^2 \left( {\alpha}^{4m-2} - p^{2m} {\alpha}^{2m} - p^{2m} {\alpha}^{2m-2} + p^{4m} \right).
\end{eqnarray*}
The above equality can be rewritten as (refer Example 8.4.7 of \cite{JP})
\begin{eqnarray} \label{E6}
 {\alpha}^{2m} - t {\alpha}^{2m-2} - (1-t) p^{2m} = 0, \mbox{ or} \\
 (m-1)^2 \frac{ p^2 } {u^2} {\alpha}^{2m} - \frac{ m ^2 p^2 }{ t u^2 } { \alpha}^{2m -2} + \frac{ \left( u^2 - p^2 \right) }{ (1-t)u^2 } p^{2m} = 0. \nonumber
\end{eqnarray}
Further observe that the equation (\ref{E6}) is equivalent to 
\begin{equation} \label{E9}
 \alpha = q_{E(2m,2)} \left( p (1-t)^{\frac{1}{2m}}, t^{\frac{1}{2}} \right),
\end{equation}
where $ q_{E(2m,2)} $ denotes the Minkowski functional of $ E(2m,2) = \left \{ z \in \mathbb{C}^2: |z_1|^{2m} + |z_2|^2 < 1    \right\} $, and therefore satisfies
\begin{equation*} 
\frac{|v_1|^{2m} }{ \left( q_{E(2m,2)} (v_1, v_2) \right)^{2m} } + \frac{|v_2|^{2} }{ \left( q_{E(2m,2)} (v_1, v_2) \right)^{2} } = 1.
\end{equation*}
It follows from the formulation (\ref{E9}) that the equation (\ref{E6}) has a unique solution $ \alpha $ in the open interval $ (0, 1) $. Once we know $ \alpha $ rather explicitly, it is easy to compute $ \tau $ -- Indeed, substituting (\ref{E6}) into the expression (\ref{E7}) yields
\begin{equation*}
\tau = \frac{m {\alpha}^{2m-1} |v_1|}{p \left( m(1-t) + t \right) \left( p^{2m} - {\alpha}^{2m-2} \right)}
\end{equation*}
which, in turn, equals
\begin{equation} \label{E11}
\tau = \frac{m {\alpha} (1-t) |v_1|}{p ( 1 - {\alpha}^2 ) \left( m(1-t) + t \right) }  \mbox{ whenever } u > p.
\end{equation}
The equations (\ref{E10}) and (\ref{E11}) together give a comprehensive formula for the infinitesimal 
Kobayashi metric of the ellipsoid $ E_{2m} $. 

\medskip

\noindent We assume that $ m > 1/2 $ for the rest of this section. To complete the proof, it remains to establish smoothness of $ K_{E_{2m}} $. To this end, we first show that away from the zero section of the tangent bundle of the domain $ E_{2m} $, both expressions (\ref{E10}) and (\ref{E11}) are $ C^1$-smooth in each of the variables $ p, v_1, \ldots, v_n $. While it is straightforward to infer smoothness from (\ref{E10}), to verify the claim for (\ref{E11}), the following observation will be needed: The \textit{Kobayashi indicatrix} of the complex ellipsoid $ E(2m,2) $ at the origin, i.e., $ \{ (v_1, v_2) \in \mathbb{C}^2 : K_{E(2m,2)} \left( (0,0), (v_1, v_2) \right) = 1 \} $ is given by the equation $ |v_1|^{2m} + |v_2|^2 = 1 $. The indicatrix is evidently $ C^1$ since $ m > 1/2$. As a consequence, the Kobayashi metric $ K_{E(2m,2)} \left( (0,0), (v_1, v_2) \right) $ must be a $ C^1$-function of the variables $ v_1 $ and $ v_2 $. Equivalently, the Minkowski functional $ q_{E(2m,2)} $ of the domain $ E(
2m,2) $ (which equals $ K_{E(2m,2)} \left( (0,0), (v_1, v_2) \right) $)
is $ C^1$. It follows from (\ref{E9}) that $ \alpha $ is $ C^1$-smooth with respect to the parameter $t$, which in turn, varies smoothly as a function of $ p, v_1, \ldots, v_n $. This proves the claim. 

\medskip

\noindent Recall that every point of $ E_{2m} $ is in the orbit of the point $ (p, \hat{0}) $ for $ 0 \leq p < 1 $. Moreover, since the action of the automorphism group of $ E_{2m} $ is real analytic, to conclude that $ K_{E_{2m}} $ is $ C^1$, it suffices to show that the Kobayashi metric is $C^1$-smooth at the point $ v = (v_1, \hat{v}) \neq (0, \hat{0}) $ in the tangent space $ T_{(p, \hat{0})} E_{2m} = \mathbb{C}^n $ with $ u = p $. We will show that for each $ j = 1, \ldots, n $,
\begin{alignat}{3} 
\lim_{p \geq u \rightarrow p } \frac{\partial K_{E_{2m} }} { \partial |v_j| } \left( (p, \hat{0}), (v_1, \hat{v}) \right)
& =  \lim_{ p < u \rightarrow p } \frac{\partial K_{E_{2m} }} { \partial |v_j| } \left( (p, \hat{0}), (v_1, \hat{v}) \right), \mbox{ and} \label{E49} \\
\label{E16}
\lim_{p \rightarrow u^{+}} \frac{\partial K_{E_{2m} }} { \partial p } \left( (p, \hat{0}), (v_1, \hat{v}) \right)
& = \lim_{p \rightarrow u^{-}} \frac{\partial K_{E_{2m}} }{ \partial p} \left( (p, \hat{0}), (v_1, \hat{v}) \right).
\end{alignat}
When $ u \leq p $, we have
\begin{equation*}
K_{E_{2m}} \left( (p, \hat{0}), (v_1, \hat{v}) \right) = \left( \frac{m^2 p^{2m-2} |v_1|^2}{(1-p^{2m})^2}
+ \frac{|v_2|^2}{1-p^{2m}} + \cdots + \frac{|v_n|^2}{1-p^{2m}} \right)^{1/2}
\end{equation*}
from which it follows that
\begin{alignat}{3} \label{E51}
\lim_{p \geq u \rightarrow p } \frac{\partial K_{E_{2m} }} { \partial |v_1| } \left( (p, \hat{0}), (v_1, \hat{v}) \right)
& = \frac{m p^{2m-1}} { 1 - p^{2m} }, \mbox{ and} \\
\label{E52}
\lim_{p \geq u \rightarrow p } \frac{\partial K_{E_{2m} }} { \partial |v_2| } \left( (p, \hat{0}), (v_1, \hat{v}) \right)
& = \left( \frac{ m^2 |v_1|^2 - p^2 |v_3|^2 - \cdots - p^2 |v_n|^2 }{ m^2 |v_1|^2}  \right)^{1/2}.
\end{alignat}
The computation in the second case, $ 0 < p < u $, will involve several steps. To begin with, differentiating (\ref{E12}) with respect to $ |v_1| $ gives
\begin{alignat*}{3}
\frac{\partial t}{ \partial |v_1| } = - \frac{2 u^2 t}{2(m-1) p^2  \left( m(1-t) + t \right) + u^2},
\end{alignat*}
which implies that
\begin{alignat}{3} \label{E13}
\lim_{p < u \rightarrow p} \frac{\partial t}{ \partial |v_1| } = \frac{2m}{(1 - 2m) p} \left( |v_2|^2 + \cdots + |v_n|^2 \right)^{-1/2}.
\end{alignat}
Differentiating (\ref{E6}) with respect to $ |v_1| $ quickly leads to
\begin{alignat}{3} \label{E48}
2 \left( m \alpha^{2m-1} - (m-1) t \alpha^{2m-3} \right)\frac{\partial \alpha}{ \partial |v_1| } = \left( \alpha^{2m-2} - p^{2m} \right) \frac{\partial t}{ \partial |v_1| }.
\end{alignat}
Taking the limit as $ u $ tends to $ p $ from above in (\ref{E48}), and using (\ref{E13}), we get
\begin{alignat}{3} \label{E14}
\lim_{p < u \rightarrow p} \frac{\partial \alpha}{ \partial |v_1| } = \frac{m \left(1 - p^{2m} \right)}{(1 - 2m) p} \left( |v_2|^2 + \cdots + |v_n|^2 \right)^{-1/2}.
\end{alignat}
Moreover, it is straightforward to check that
\begin{alignat}{3} \label{E15}
\lim_{p < u \rightarrow p} t = 1 \; \mbox{ and } \; \lim_{p < u \rightarrow p} \alpha = 1.
\end{alignat}
Also, observe that the equations (\ref{E6}) and (\ref{E12}) are equivalent to 
\begin{alignat*}{3}
\frac{1 - t}{1- {\alpha}^2} = \frac{ {\alpha}^{2m-2} }{ {\alpha}^{2m-2} - p^{2m}} \mbox{ and } 
u^2  = \frac{p^2 \left( m(1-t) + t \right)^2}{t}
\end{alignat*}
respectively, so that (\ref{E11}) can be rewritten as
\begin{alignat*}{3}
K_{E_{2m}} \left( (p, \hat{0}), (v_1, \hat{v}) \right)  = \frac{{\alpha}^{2m-1}}{{\alpha}^{2m-2} - p^{2m}} \left( \frac{|v_2|^2}{t} + \cdots + \frac{|v_n|^2}{t} \right)^{1/2},
\end{alignat*}
which upon differentiation turns out to be
\begin{multline*}
\frac{\partial K_{E_{2m} }} { \partial |v_1| } \left( (p, \hat{0}), (v_1, \hat{v}) \right)
= \left( \frac{(2m-1) \alpha^{2m-2}}{ \left( \alpha^{2m-2} - p^{2m} \right) } \frac{\partial \alpha}{ \partial |v_1|}
- \frac{\alpha^{2m-1}}{ 2 t^{3/2} \left( \alpha^{2m-2} - p^{2m} \right)  } \frac{\partial t}{ \partial |v_1| } \right.
\\
\left.
- \; \frac{ 2(m-1) \alpha^{4m-4} }{ t^{1/2} ( \alpha^{2m-2} - p^{2m})^2  } \frac{\partial \alpha}{ \partial |v_1|}
\right).
\end{multline*}
So that
\begin{alignat}{3} \label{E50}
\lim_{p < u \rightarrow p } \frac{\partial K_{E_{2m} }} { \partial |v_1| } \left( (p, \hat{0}), (v_1, \hat{v}) \right)
= \frac{m p^{2m-1}}{1 - p^{2m}}
\end{alignat}
owing to (\ref{E13}), (\ref{E14}) and (\ref{E15}). The expressions (\ref{E51}) and (\ref{E50}) together verify that 
(\ref{E49}) holds for $ j = 1 $. Furthermore, a similar computation yields that
\begin{alignat*}{3}
\lim_{p < u \rightarrow p} \frac{\partial t}{ \partial |v_2| } & = \frac{2 p}{ (2m -1) m^2 |v_1|^2 } \left(  m^2 |v_1|^2 - p^2 |v_3|^2 - \cdots - p^2 |v_n|^2   \right)^{1/2}, \mbox{ and} \\
\lim_{p < u \rightarrow p} \frac{\partial \alpha}{ \partial |v_2| } & = \frac{ p \left( 1 - p^{2m} \right) }{ (2m -1) m^2 |v_1|^2 } \left(  m^2 |v_1|^2 - p^2 |v_3|^2 - \cdots - p^2 |v_n|^2   \right)^{1/2}.
\end{alignat*}
Consequently, we find that, in agreement with the first case (cf. (\ref{E52}))
\begin{alignat*}{3}
\lim_{p < u \rightarrow p } \frac{\partial K_{E_{2m} }} { \partial |v_2| } \left( (p, \hat{0}), (v_1, \hat{v}) \right)
= \left( \frac{ m^2 |v_1|^2 - p^2 |v_3|^2 - \cdots - p^2 |v_n|^2 }{ m^2 |v_1|^2}  \right)^{1/2}.
\end{alignat*}
Finally, an argument similar to the one used above shows that
\begin{alignat*}{3}
\lim_{p \geq u \rightarrow p } \frac{\partial K_{E_{2m} }} { \partial |v_j| } \left( (p, \hat{0}), (v_1, \hat{v}) \right)
=  \lim_{ p < u \rightarrow p } \frac{\partial K_{E_{2m} }} { \partial |v_j| } \left( (p, \hat{0}), (v_1, \hat{v}) \right)
\end{alignat*}
for each $ j = 3, \ldots, n $.

\medskip

\noindent In the third case, $ 0 = p < u $, the Kobayashi indicatrix of $ E_{2m} $ at the origin, defined as
\[
\{ v \in \mathbb{C}^n : K_{E_{2m}} \left( (0, \ldots, 0), (v_1, \ldots, v_n) \right) = 1 \}, 
\]
is given by the equation $ |v_1|^{2m} + |v_2|^2 + \cdots + |v_n|^2 = 1 $. Since $ m > 1/2 $, the indicatrix is $ C^1 $ which, in turn, implies that the Kobayashi metric $ K_{E_{2m}} \left( (0, \ldots, 0), (v_1, \ldots, v_n) \right) $ is also $ C^1$. 

\medskip

\noindent To conclude that $ K_{E_{2m}} $ is $ C^1 $, it remains to verify that (\ref{E16}) holds. Evaluating the left hand side first, note that from (\ref{E10}), we obtain
\begin{alignat*}{3}
\lim_{p \rightarrow u^{+}} K_{E_{2m}} \left( (p, \hat{0}), (v_1, \hat{v}) \right) = \frac{m|v_1|}{u \left( 1 - |u|^{2m} \right)}.
\end{alignat*}
On differentiating (\ref{E10}) with respect to $p$, we have that
\begin{multline*}
2 K_{E_{2m}} \left( (p, \hat{0}), (v_1, \hat{v}) \right) \frac{\partial K_{E_{2m} }}
{ \partial p } \left( (p, \hat{0}), (v_1, \hat{v}) \right) = \frac{ 2m p^{2m-1} }
{ (1- p^{2m})^2 } \left( |v_2|^2 + \cdots + |v_n|^2 \right) \\
+ \frac{2 m^2 (m-1) p^{2m-3} }{ \left( 1 - p^{2m} \right)^2 } |v_1|^2 +  
\frac{4 m^3 p^{4m-3} }{ \left( 1 - p^{2m} \right)^3 } |v_1|^2. 
\end{multline*}
Letting $ p \rightarrow u^{+} $, and using the above observation, we get
\begin{alignat*}{3}
\lim_{p \rightarrow u^{+}} \frac{\partial K_{E_{2m} }} { \partial p } \left( (p, \hat{0}), (v_1, \hat{v}) \right) = 
\frac{m u^{2m-2} \left( u^{2m} + 2m - 1 \right)} { \left( 1 - u^{2m} \right)^2 } |v_1|.
\end{alignat*}
Working with the right hand side of (\ref{E16}), observe that
\begin{alignat*}{3}
\lim_{p \rightarrow u^{-}} t = 1 \mbox{ and } \lim_{p \rightarrow u^{-}} \alpha = 1.
\end{alignat*}
It follows from the definition of the parameter $ t $ that
\begin{alignat*}{3}
\frac{\partial t }{ \partial p} = \frac{2 p \left( m(1-t) + t \right)^2}{ 2 p^2 (m-1) + u^2}, 
\end{alignat*}
which implies that
\begin{alignat}{3} \label{E17}
\lim_{p \rightarrow u^{-}} \frac{\partial t }{ \partial p} = \frac{2}{(2m -1)u}.
\end{alignat}
To compute $ \partial \alpha/ \partial p $, we differentiate both sides of (\ref{E6}). After simplification, we get
\begin{alignat*}{3}
\left( 2m {\alpha}^{2m-1} - 2t(m-1) {\alpha}^{2m-3} \right) \frac{\partial \alpha}{ \partial p } = \left( {\alpha}^{2m-2} - p^{2m} \right) \frac{\partial t}{ \partial p } + 2m (1-t) p^{2m-1},
\end{alignat*}
so that
\begin{alignat}{3} \label{E18}
\lim_{p \rightarrow u^{-}} \frac{\partial \alpha}{ \partial p} = \frac{ 1 - u^{2m}}{(2m -1)u}.
\end{alignat}
Using (\ref{E11}) to write the derivative of $ K_{E_{2m}} $ with respect to $ p$, it follows, by virtue of (\ref{E17}) and (\ref{E18}), that
\begin{alignat*}{3}
\lim_{p \rightarrow u^{-}} \frac{\partial K_{E_{2m} }} { \partial p } \left( (p, \hat{0}), (v_1, \hat{v}) \right) = 
\frac{m u^{2m-2} \left( u^{2m} + 2m - 1 \right)} { \left( 1 - u^{2m} \right)^2 } |v_1|,
\end{alignat*}
as required. This finishes the proof of the theorem.
\qed

\section{Proof of Theorem \ref{thm7}}

\noindent Suppose that there is a Kobayashi isometry $ f $ from $ D_1 $ onto 
$ D_2 $ with $ q^0 \in cl_f(p^0) $, the cluster set of $ p^0 $. Firstly, from 
the explicit form of the defining function for $ U_1 \cap \partial D_1 $, 
it is clear that $ \partial D_1 $ near $ p^0 $ is smooth pseudoconvex and 
of finite type with Levi rank exactly $ (n-2) $.  

\medskip

\noindent Assume that both $ p^0 = 0 $ and $ q^0 = 0 $ and choose a sequence $ 
p^j = ('0, -\delta_j) \in U_1 \cap D_1 $ on the inner normal approaching the
origin. By Theorem \ref{thm5}, it readily follows that $ f $ extends continuously up
to $ p^0 $. As a consequence, $ q^j = f(p^j) $ converges to $ q^0 $ which is a 
strongly pseudoconvex point in $ \partial D_2 $. The idea is to apply the scaling 
technique to $ (D_1, D_2, f) $ as in the proof of Theorem \ref{thm4}. To scale 
$ D_1 $, we only consider dilations 
\[
\Delta^j( z_1, z_2, \ldots, z_{n-1}, z_n ) = \left( \delta_j^{- \frac{1}{2m}} z_1, 
\delta_j^{-\frac{1}{2}} z_2, \cdots, \delta_j^{-\frac{1}{2}} z_{n-1}, \delta_j^{-1} z_n \right).
\]
Note that $ \Delta^j ('0, - \delta_j) = ('0, -1) $ for all $j$ and the domains 
$ D_1^j = \Delta^j (D_1) $ converge in the Hausdorff sense to
\[
D_{1, \infty} = \big\{ z \in \mbb C^n :  2 \Re z_n + \abs{z_1}^{2m} + \abs{z_2}^2 + \cdots + 
\abs{z_{n-1}}^2  < 0 \big\},
\]
which is biholomorphic to $ E_{2m} $. 

\medskip

\noindent While for $ D_2 $, we use the composition $ T^j \circ h^{\xi^j} $ 
as in Section 5 - we include a brief exposition here for completeness: Consider 
points $ \xi^j \in \partial D_2 $ defined by $ \xi^j = q^j + ('0, \epsilon_j) $, 
for some $ \epsilon_j > 0 $. As before, $ h^{\xi^j} $ are the `centering maps' 
(cf. \cite{Pi}) corresponding to $ \xi^j \in \partial D_2 $, and $ T^j $ 
are the dilations 
\[
T^j ( w_1, w_2, \ldots, w_n ) = \left( { {\epsilon_j}^{-\frac{1}{2}}{w_1} } , \ldots,
 { \epsilon_j}^{-\frac{1}{2}} {w_{n-1}}, { \epsilon_j}^{-1} {w_n
 } \right).
\]
It was shown in \cite{Pi} that the dilated domains $ D_2^j = T^j \circ h^{\xi^j} ( D_2 ) $ 
converge to 
\[
 D_{2, \infty} = \left \{ w \in \mathbb{C}^n : 2 \Re w_n + | w_1 |^2 + 
 | w_2 |^2 + \cdots + | w_{n-1} |^2 < 0
\right \}
\]
which is the unbounded representation of the unit ball in $ \mathbb{C}^n $. 
Among other things, the following claim was verified in \cite{SV1}: For $ w \in D_{2, \infty} $, 
\begin{alignat}{3} \label{E39}
 d^k_{D_2^j} \left( w, \cdot \right) \rightarrow d^k_{D_{2,\infty}} \left( w, \cdot \right) 
\end{alignat} 
uniformly on compact sets of $ D_{2,\infty}$. As a consequence, the sequence of Kobayashi metric balls 
$ B_{D_2^j} \left( \cdot, R \right) \rightarrow B_{D_{2,\infty}} \left( \cdot, R \right) $,
and, for large $j$,
\begin{alignat}{3} \label{E32}
B_{D_{2,\infty}} \left( \cdot, R \right) \subset B_{D_2^j} \left( \cdot, R +
\epsilon \right),  \mbox{ and} \nonumber \\ 
B_{D_2^j} \left( \cdot, R - \epsilon \right)  \subset
B_{D_{2,\infty}} \left( \cdot, R \right).
\end{alignat}

\noindent The scaled maps $ f^j = T^j \circ h^{\xi^j} \circ f \circ (\Delta^j)^{-1} : D_1^j
\rightarrow D_2^j $ are isometries in the Kobayashi metric on $ D_1^j $ and $ D_2^j $ and 
note that $ f^j( '0,-1) = ('0,-1) $ for all $j$. Exhaust $ D_{1, \infty} $ by an increasing union 
$ \{ K_{\nu} \} $ of relatively compact domains, each containing $ ('0, -1) $. Fix a pair 
$ K_1 $ compactly contained in $ K_2 $ say, and let $ \omega(K_1) $ be a neighbourhood of $ K_1 $ 
such that $ \omega(K_1) \subset K_2 $. Since the domains $ D_1^j $ converge to $ D_{1,\infty} $, it 
follows that $ \omega(K_1) \subset K_2$ is relatively compact in $ D_1^j $ for all $ j $ large. 
Now, to establish that $ f^j $ admits a convergent subsequence, it will suffice to show that 
$ f^j $ restricted to $ \omega(K_1) $ is uniformly bounded and equicontinuous. For each $ z \in K_2 $,
note that for all $j$ large,
\[
d^k_{D^j_2 } \left( f^j(z), ('0,-1) \right) = d^k_{D^j_1} \left(z,('0,-1)
\right) \leq d^k_{D_{1,\infty}} \left(z, ('0,-1) \right) + \epsilon 
\]
where the last inequality follows from Proposition \ref{E30}. Observe that the right hand side above 
is bounded above by a uniform positive constant, say $ \tilde{R} > 0 $. Therefore, by
Proposition \ref{E31}, it follows that
\begin{alignat}{3} \label{E33}
f^j(K_2) \in B_{D^j_2} \left( ('0,-1), \tilde{R} \right) \subset B_{D_{2, \infty}}
\left( ('0,-1), \tilde{R} + \epsilon \right),
\end{alignat}
which exactly means that $ \{ f^j (K_2) \} $ is uniformly bounded. 

\medskip

\noindent The following observation will be needed to deduce the equicontinuity of $ f^j $ 
restricted to $ \omega(K_1) $. For each $ z \in \omega(K_1) $, there is a 
small ball $ B(z, r) $ around $ z $ with radius $ r > 0 $, which is compactly contained 
in $ \omega(K_1) $. For $ R' \gg 2 \tilde{R} $, we intend to apply Lemma \ref{E29} 
to the domain $ D_2^j $ with the Kobayashi ball $ B_{D_2^j} \left(('0, -1), R' \right) $ 
as the subdomain $ D'$. Let $ \tilde{z} \in B(z, r) $, then, for $ f^j(\tilde{z}), f^j(z) $ 
as $ p $ and $ q $ respectively, and $ R'/2 $ as $ b $, it can be checked that all the conditions 
of Lemma \ref{E29} are satisfied. So that
\begin{alignat}{3} \label{E36}
d^k_{B_{D_2^j} \left( ('0, -1), R' \right)} \left(  f^j(z), f^j(\tilde{z})  \right) \leq
\frac{d^k_{D_2^j} \left(  f^j(z), f^j(\tilde{z})  \right)} { \tanh \left( \frac{R'}{2} -
d^k_{D_2^j} \left(  f^j(z), f^j(\tilde{z})  \right) \right)} \leq \frac{d^k_{D_2^j} 
\left(  f^j(z), f^j(\tilde{z})  \right)} { \tanh \left( \frac{R'}{2} - 2 \tilde{R} \right)}
\end{alignat}
where the last inequality above is a simple consequence of triangle inequality and (\ref{E33}).

\medskip

\noindent Furthermore, it turns out that for any small neighbourhood $ W $ of 
$ q^0 \in \partial D_2 $ and for all large $ j $, $ B_{D_{2, \infty}} \left( 
('0, -1), R' + \epsilon \right) \subset T^j \circ h^{\xi^j} ( W \cap D_2 ) $. 
On the other hand, let $ R > 1 $ be such that 
\begin{alignat*}{3}
h^{\xi^j} ( W \cap D_2 ) \subset \{ w \in \mathbb{C}^n : |w_1|^2 + \cdots + |w_{n-1}|^2 + |w_n + R|^2 < R^2 \} \\
\subset \{ w \in \mathbb{C}^n : 2 R (\Re w_n) + |w_1|^2 + \cdots + |w_{n-1}|^2  < 0 \} = D_0 \approx \mathbb{B}^n.
\end{alignat*}
Observe that $ T ^j $ leaves the domain $ D_0 $ invariant for each $j$, therefore, we may conclude that
\[
B_{D_{2, \infty}} \left( ('0, -1), R' + \epsilon \right) \subset T^j \circ h^{\xi^j} ( W \cap D_2 ) 
\subset D_0 \approx \mathbb{B}^n.
\]
Now, using the explicit form of the Kobayashi metric on $ \mathbb{B}^n $, it follows that
\begin{alignat}{3} \label{E34}
|f^j(z) - f^j( \tilde{z}) | \lesssim d^k_{D_0} \left( f^j(z), f^j(\tilde{z}) \right) \leq d^k_{B_{D_{2, \infty}}
\left( ('0,-1), R' + \epsilon \right)} \left( f^j(z), f^j(\tilde{z}) \right). 
\end{alignat}
While from (\ref{E32}), we see that
\begin{alignat}{3} \label{E35}
d^k_{B_{D_{2, \infty}} \left( ('0,-1), R' + \epsilon \right)} \left( f^j(z), f^j(\tilde{z}) \right)  
\leq d^k_{B_{D_2^j} \left( ('0,-1), R' \right)} \left( f^j(z), f^j(\tilde{z}) \right). 
\end{alignat}
Finally, we deduce, using the expression for the Kobayashi metric (which equals the 
Poincar\'{e} metric) on $ B(z, r) $, that 
\begin{alignat}{3} \label{E37}
d^k_{D^j_2 } \left( f^j(z), f^j(\tilde{z}) \right) = d^k_{D^j_1}(z,\tilde{z}) \leq d^k_{B(
z,r)} (z,\tilde{z}) \lesssim  | z - \tilde{z} |.
\end{alignat}
so that
\begin{alignat*}{3}
|f^j(z) - f^j( \tilde{z}) | \lesssim | z - \tilde{z} |
\end{alignat*}
by virtue of (\ref{E34}), (\ref{E35}), (\ref{E36}) and (\ref{E37}). Hence, there is a 
well-defined continuous limit of some subsequence of $ f^j $. Denote this limit by 
$ \tilde{f}: D_{1, \infty} \rightarrow \overline{D}_{2, \infty} $. The proof now 
divides into two parts. The first part is to show that $ \tilde{f} $ is an isometry 
between $ D_{1, \infty} $ and $ D_{2, \infty} $ and the second step is to show that 
$ \tilde{f} $ is either holomorphic or anti-holomorphic.

\medskip

\noindent \textbf{I. $ \tilde{f} $ is an isometry}

\medskip

\noindent If $ \tilde{f} $ were 
known to be holomorphic, then the maximum principle would imply that $ \tilde{f}: 
D_{1, \infty} \rightarrow D_{2, \infty} $. However, $ \tilde{f} $ is known to be just 
continuous. To overcome this difficulty, consider $ \Omega_1 \subset D_{1,\infty} $, the 
set of all points $ z \in D_{1, \infty} $ such that $ \tilde{f}(z) \in D_{2, \infty} $.
Note that $ \tilde{f} \left( ('0, -1) \right) = ('0, -1) \in D_{2, \infty} $ and 
hence $ \Omega_1 $ is non-empty. Also, since $ \tilde{f} $ is continuous, it follows that
$ \Omega_1 $ is open in $ D_{1, \infty} $.

\medskip

\noindent \textbf{Assertion:} $ d^k_{D_{1, \infty}} (p, q) = d^k_{D_{2, \infty}} \left( \tilde{f}(p), 
\tilde{f}(q) \right) $ for all $ p, q \in \Omega_1 $.

\medskip

\noindent Grant this for now. Then, for $ s \in \partial \Omega_1 \cap D_{1, \infty} $, 
let $ s^j \in \Omega_1 $ such that $ s^j \rightarrow s $. Assuming that the assertion 
holds true, we have that
\begin{alignat}{3} \label{E38}
d^k_{D_{1, \infty}} \left( s^j, ('0, -1) \right) = d^k_{D_{2, \infty}} \left( \tilde{f}(s^j), ('0, -1) \right)
\end{alignat}
for all $ j $. Since $ s \in \partial \Omega_1 \cap D_{1, \infty} $, $ \tilde{f}(s^j) $ converges
to a point of $ \partial D_{2, \infty} $. Furthermore, $ D_{2, \infty} $ is complete in the 
Kobayashi metric, and hence, the right hand side in (\ref{E38}) is unbounded. However, the left hand side 
remains bounded again because of completeness of $ D_{1,\infty} $. This contradiction shows that 
$ \Omega_1 = D_{1, \infty} $. In other words, $ \tilde{f}: D_{1, \infty} \rightarrow D_{2, \infty} $
and $ d^k_{D_{1, \infty}} (p, q) = d^k_{D_{2, \infty}} \big( \tilde{f}(p), 
\tilde{f}(q) \big) $ for all $ p, q \in D_{1, \infty} $. 

\medskip

\noindent To verify the assertion, recall that $ d^k_{D_1^j} (p, q) = d^k_{D_2^j} \big( f^j(p), 
f^j(q) \big) $ for all $ j $. The statement $ d^k_{D_1^j} (p, q) \rightarrow d^k_{D_{1,\infty}} (p, q) $ 
follows from the proof of Proposition \ref{E30}. Hence, it remains to show that the right hand side above 
converges to $ d^k_{D_{2,\infty}} \left( \tilde{f}(p), \tilde{f}(q) \right) $. To achieve this, note that
\[
\big | d^k_{D^j_2} \big(f^j(p),f^j(q) \big) - d^k_{D^j_2} \big (
\tilde{f}(p), \tilde{f}(q) \big) \big| \leq   d^k_{D^j_2} \big(
f^j(p), \tilde{f}(p) \big) + d^k_{D^j_2} \big( \tilde{f}(q), f^j(q)
\big)
\]
by the triangle inequality. Since $ f^j(p) \rightarrow
\tilde{f}(p) $ and $ D^j_2 \rightarrow D_{2, \infty} $, 
it follows that there is a small ball $ B \big(\tilde{f}(p), r \big)
$ around $ \tilde{f}(p) $ which contains $ f^j(p) $ and
which is contained in $ D^j_2 $ for all large $j$, where 
$ r > 0 $ is independent of $j$. Thus
\[
d^k_{D^j_2} \big( f^j(p), \tilde{f}(p) \big) \lesssim \big| f^j(p) -
\tilde{f}(p) \big |.
\]
Similarly, it can be checked that $ d^k_{D^j_2} \big( \tilde{f}(q),f^j(q) \big) $ 
is arbitrarily small. Furthermore, it follows from (\ref{E39}) that
$ d^k_{D^j_2} \big ( \tilde{f}(p), \tilde{f}(q) \big) \rightarrow  d^k_{D_{2, \infty}} \big(
\tilde{f}(p), \tilde{f}(q) \big) $. Hence the assertion.

\medskip

\noindent Next, we claim that $ \tilde{f} $ is surjective. To see this, consider
any point $ t^0 \in \partial \big( \tilde{f}( D_{1, \infty} ) \big) \cap D_{2,\infty} $ and let 
$ t^j \in \tilde{f}(D_{1, \infty})$ satisfying $ t^j \rightarrow t^0 $. Pick
$ s^j \in D_{1, \infty} $ such that $ \tilde{f}(s^j)= t^j$. Then 
\begin{equation*}
d^k_{D_{1, \infty}} \big( ('0,-1), s^j \big) = d^k_{D_{2, \infty}} \big( \tilde{f}( ('0, -1)),
\tilde{f}(s^j) \big)
\end{equation*}
There are two cases to be considered depending on whether $ s^j \rightarrow 
s \in \partial D_{1, \infty} $ or $ s^j \rightarrow  s^0 \in D_{1, \infty} $ as $ j
\rightarrow \infty $. In case $ s^j \rightarrow s $, observe that the right hand 
side above remains bounded because of the completeness of $ D_{2, \infty}$. But 
left hand side is unbounded since $ D_{1, \infty}$ is complete in the Kobayashi metric. 
This contradiction shows that $ s^j \rightarrow  s^0 \in D_{1, \infty} $, which, in turn,
implies that $ \tilde{f}(s^0) = t^0 $. Now, consider the isometries $ (f^j)^{-1} : 
D_2^j \rightarrow D_1^j $. Exactly the same arguments as above show that some 
subsequence of $ (f^j)^{-1} $ converges uniformly on compact sets of 
$ D_{2, \infty} $ to $ \tilde{g}: D_{2, \infty} \rightarrow D_{1, \infty} $. It follows that
$ \tilde{f} \circ \tilde{g} \equiv id_{D_{2,\infty}} $. In
particular, $ \tilde{f}$ is surjective and $ \tilde{f} $ is a isometry
between $ D_{1, \infty} $ and $ D_{2, \infty} $ in the Kobayashi
metric. 
\medskip

\noindent \textbf{II. $ \tilde{f} $ is (anti)-holomorphic}

\medskip

\noindent The proof of the fact that $ \tilde{f} $ is a biholomorphic mapping follows 
exactly as in \cite{SV2}. To outline the key ingredients, 
the first step is to show that $ \tilde{f} $ is differentiable everywhere, which implies
that the Kobayashi metric $ K_{D_{1, \infty}} $ is Riemannian. By Theorem \ref{thm6}, 
we know that $K_{E_{2m}} $ or equivalently that $ K_{D_{1, \infty}} $ is $ C^1$-smooth. 
Recall that $ D_{1, \infty} \approx E_{2m} $ and $ D_{2, \infty} \approx \mathbb{B}^n $. 
So that $ \tilde{f} $ after composing with appropriate Cayley transforms, leads to a 
continuous isometry $ \tilde{F} $ between two $ C^1 $-smooth Riemannian manifolds 
$ \big( E_{2m}, K_{_{2m}} \big) $ and $ \big( \mathbb{B}^n, K_{\mathbb{B}^n } \big) $. 
Applying the Myers-Steenrod theorem repeatedly to $ \tilde{F} $, yields 
the desired result.

\medskip

\noindent Once we know that $ \tilde{F} : E_{2m} \rightarrow \mathbb{B}^n $
is holomorphic, which, additionally, may be assumed to preserve the origin, it 
follows that $ 2m=2$. Indeed, $ \tilde{F} $ is a biholomorphism between two
circular domains, $ E_{2m} $ and $ \mathbb{B}^n $ such that $ \tilde{F}(0) = 0 $
and is, hence, linear. In particular $ 2m=2 $. Said differently, there are holomorphic
coordinates at $ p^0 $ in which a tiny neighbourhood around $ p^0 $ can be written as 
\[
\big \{ z \in \mathbb{C}^n : 2 \Re z_n + |z_1|^2 + |z_2|^2 + \cdots + |z_{n-1}|^2 + 
\text{higher order terms} < 0 \big\},
\]
which violates the assumption that the Levi rank at $ p^0 \in \partial D_1 $ is exactly 
$ n-2 $. Alternatively, one can use Theorem \ref{thm4} or results from \cite{BV} 
to arrive at a contradiction. Hence the theorem.\qed

\section{Appendix} \label{app}
\noindent The group of all polynomial automorphisms of $\mathbb{C}^n$ is usually denoted by $GA_n(\mathbb{C})$ and two special subgroups of $GA_n(\mathbb{C})$ are the affine subgroup 
\[
Af_n(\mathbb{C}) = \{ F \in GA_n(\mathbb{C}) \; : \; {\rm deg}(F) \leq 1 \}
\]
and the triangular subgroup 
\[
BA_n(\mathbb{C}) = \{ F \in GA_n(\mathbb{C}) \; : \; F_j =a_jz_j + H_j \text{ where } a_j \in \mathbb{C}^* \text{ and } H_j  \in \mathbb{C}[z_1, \ldots, z_{j-1}] \}
\]
whose members are also called {\it elementary automorphisms}.
\medskip\\
\noindent The Jung-van der Kulk theorem says that every polynomial automorphism in dimension $n=2$ can be obtained as a finite composition of affine and elementary automorphisms; using this 
fact it can be derived (cf. \cite{FrMi}) that the degree of the inverse $\Phi^{-1}$ of a polynomial automorphism of $\mathbb{C}^2$ is the same as the degree of $\Phi$. However, these facts are known to be false in higher dimensions.\\

\noindent The weight of a polynomial automorphism -- or more generally a polynomial endomorphism of $\mathbb{C}^n$ -- is by definition the maximum of the weights of its components.\\

\noindent  Let us assign a weight of $1/2m$ to the variable $z_1$ where $m \in \mathbb{N}$, $1/2$ to the variables $z_\alpha$ for all $2 \leq \alpha \leq n-1$ and $1$ to the variable $z_n$ as in the introduction and consider the collection $\mathcal{E}_L = \mathcal{E}_{(1/2m,1/2,1)}$, of all weight preserving polynomial automorphisms of the form 
\[
F(z) = \big( a_1 z_1 + b_1, A({}''z) + P_2(z_1), a_n z_n + b_n + P_n('z) \big)
\]
where $''z=(z_2, \ldots, z_{n-1})$, $ 'z=(z_1, \ldots, z_{n-1})$, $A$ is an invertible affine transform on $\mathbb{C}^{n-2}$, $P_2$ is a vector 
valued polynomial all of whose components are polynomials of weight at-most $1/2$ and $P_n$ is a polynomial of weight at-most $1$ while
 $b_1,b_n \in \mathbb{C}$ and $a_1 , b_n \in \mathbb{C}^* = \mathbb{C} \setminus \{0\}$. In general, it is not true that the degree (resp. weight) 
of the inverse $\Phi^{-1}$ of a polynomial automorphism $\Phi$, is same as the degree (resp. weight) of $\Phi$. However, since $\mathcal{E}_L$ consists 
of `elementary like' polynomial automorphisms that preserve the weight of each component, we have that the weights of the components of the 
inverse of $F$ is same as those of $F$, for each $F \in \mathcal{E}_L $. Indeed, note that the inverse of $F$ is given by
\[
 F^{-1}(z) = \Big( a_1^{-1}(z_1 - b_1), A^{-1} \big( {}''z - P_2 \big( a_1^{-1}(z_1-b_1) \big) \big), a_n^{-1} \big( z_n-b_n - P'_n({}'z) \big) \Big)  
\]
where $P'_n({}'z) =  P_n \Big( a_1^{-1}(z_1-b_1), A^{-1} \big({} 'z - P_2(a_1^{-1}(z_1-b_1)) \big) \Big)$. It is easy to see that 
$F^{-1}$ lies in $\mathcal{E}_{L}$. Next note that if we pick another $G \in \mathcal{E}_L$ given by 
\[
G(z) = \big( c_1  z_1 + d_1, B({}''z) + Q_2(z_1), c_n z_n +d_n + Q_n('z) \big)
\]
say, then
\begin{multline*}
G(F(z)) = \Big( c_1 a_1 z_1 + c_1 b_1 + d_1,  BA (''z) + BP_2(z_1) + Q_2(a_1z_1 + b_1),\\
                                                                 c_n a_n z_n + c_n b_n + d_n + Q_n \big( c_1z_1+d_1, A(''z) + P_2(z_1) \big) \Big) 
\end{multline*}
which is again in $\mathcal{E}_{(1/2m,1/2, 1)}$ completing the verification that $\mathcal{E}_L$ is a group. In fact noting that $\mathcal{E}_L$ is in {\it bijection} with a product of finitely many copies of $\mathbb{C}^M \setminus \{0\}$, $\mathbb{C}^N$ for some $M,N \in \mathbb{N}$ and $GL_{n-2}(\mathbb{C})$ we see that $\mathcal{E}_L$ is a {\it non-singular} affine algebraic variety; next noting that the composition of maps in $\mathcal{E}_L$ when viewed as an operation on the various coefficients here, is a polynomial operation on these coefficients and likewise for taking inverses as well, we conclude that $\mathcal{E}$ is a complex algebraic group.\\

\noindent Recall the special reduction procedure for the Taylor expansion of any given smooth defining function for a 
piece of Levi corank one hypersurface. Let $\Sigma$ be a smooth pseudoconvex real-hypersurface in $\mathbb{C}^n$ of finite type 
with the property that the Levi-rank is at-least $n-2$ at each of its points. We assume that the origin lies in $\Sigma$ and that the D' Angelo $1$-type of the points of $\Sigma$ is bounded above by some integer $2m$. Let $r$ be a smooth defining function for $\Sigma$ with $\partial r /\partial z_n(z) \ne 0$ for all $z$ in a small neighbourhood $U$ in $\mathbb{C}^n$ of $\Sigma$ such that the vector fields
\begin{equation*} 
L_n=\partial /\partial z_n,\ L_j=\partial / \partial z_j + b_j(z,\bar{z})\partial / \partial z_n,
\end{equation*}
where $b_j = \big( \partial r/ \partial z_n \big)^{-1} \partial r / \partial z_j $, form a basis of $\mathbb{C}T^{(1,0)}(U)$ and satisfy $L_j r\equiv 0$ for $1 \leq j\leq n-1$ and 
for each $ z \in U $, all eigenvalues of $\partial\bar{\partial} r(z) (L_i,\bar{L}_j)_{2\leq i,j \leq n-1}$
are positive.

\medskip

\noindent The main objective of this reduction procedure is to obtain a certain normal form near $ \zeta \in U $ in which there are no harmonic
 monomials of weight less than one, when weights are taken with respect to the inverses of the multitype at $\zeta \in U$. Recall that Levi corank one hypersurfaces are $h$-extendible i.e., their Catlin multitype and D' Angelo multitype agree at every point.

\medskip

\noindent We digress a little here, to recall and introduce the term `weakly spherical' used in the introduction. First, we recall the notion of weak sphericity of Barletta -- Bedford from \cite{BaBe}: a smooth pseudoconvex hypersurface $M \subset \mathbb{C}^2$ of finite type $2m$ at $p \in M$ can after a change of coordinates centered at $p=0$, be defined by a function of the form
\[
2 \Re z_2 + P_{2m}(z_1, \bar{z}_1) + \Im z_2 \sum\limits_{l=1}^{k} Q_l(z_1) + \sigma_{2m+1}(z_1) + \sigma_2(\Im z_2) + (\Im z_2) \sigma_{m+1}(z_1)
\]
Here $P_{2m}(z_1, \bar{z}_1)$ is a non-zero homogeneous subharmonic polynomial of degree $2m$ without harmonic terms, the $Q_l$'s are homogeneous polynomials of degree $l$ and the $\sigma_j$'s vaish to order $j$ in $z_1$ or $\Im z_2$. To put it succintly, the lowest weight component in the weighted homogeneous expansion of the defining function with respect to the weight $(1,1/2m)$ given by the type, is of weight one and of the form $2 \Re z_2 + P_{2m}(z_1, \bar{z}_1)$. Now according to \cite{BaBe}, $M$ is weakly spherical at $p$ if $P_{2m}(z_1, \bar{z}_1) = \vert z_1 \vert^{2m} = \vert z_1^m \vert^2$ -- note that a `squared norm of a polynomial' of weight $1/2$ in the single variable $z_1$, is necessarily of the form $c \vert z_1 \vert^{2m}$ for some positive constant $c$. We may extend this notion to higher dimensions for $h$-extendible hypersurfaces as follows: call a smooth pseudoconvex hypersurface $M \subset \mathbb{C}^n$ which is $h$-extendible at a point $p \in M$, to be {\it weakly spherical} at $p$ 
if there is a change of coordinates that maps $p$ to the origin, in which the lowest weight component of the weighted homogeneous expansion of the defining function of $M$ about $p \in M$ which we may assume to be the origin, performed with respect to the weights given by the inverse of the multitype $(m_n,\ldots,m_1)$ of $M$ at $p=0$, is of the form
\[
2 \Re z_n + \vert P_1({}'z) \vert^2 + \ldots + \vert P_{n-1}({}'z) \vert^2
\]
which when expanded is of the form
\[
2 \Re z_n + c_1 \vert z_1 \vert^{m_n} + c_2 \vert z_2 \vert^{m_{n-1}} + \ldots + c_{n-1} \vert z_{n-1} \vert^{m_2} + {\rm mixed \; terms}
\]
where the phrase `mixed terms' denotes a sum of weight $1$ monomials annihilated by at-least one of the natural quotient maps 
$\mathbb{C}[ {}'z, {}' \overline{z}] \to \mathbb{C}[ {}'z, {}' \overline{z}] /(z_j \overline{z}_k)$ for $1 \leq j,k \leq n-1$, $j \neq k$.\\

\medskip

\noindent Now, the mixed terms involving the $z_\alpha$'s for $2 \leq \alpha \leq n-1$, must be of the form $z_\alpha \overline{z}_\beta$ where $2 \leq \beta \leq n-1$ with $\beta \neq \alpha$, since ${\rm wt}(z_\alpha)=1/2$. An application of the spectral theorem, removes the occurrence of such terms with $\alpha \neq \beta$. The remaining mixed terms must be of the form $z_\alpha \overline{z}_1^j$ where $ 2 \leq \alpha \leq n-1$ and $1 \leq j \leq m$ and constitutes the polynomial
\[
 \sum\limits_{\alpha=2}^{n-1} \sum \limits_{j=1}^{m} \Re (b_j^\alpha \overline{z}_1^j z_\alpha),
 \]
say. Then an application of the change of variables given by
\begin{align*}
w_1 &=z_1, \; w_n=z_n\\
w_\alpha &= z_\alpha - P(z_1) \; \text{ for } 2 \leq \alpha \leq n-1
\end{align*}
where $P(z_1) = \sum\limits_{j=1}^{m}b_j^\alpha z_1^j $, removes the occurrence of such terms as well and the transformed defining function when expanded about $p=0$ reads
\[
2 \Re z_n +  \vert z_1 \vert^{2m} + \vert z_2 \vert^2 + \ldots + \vert z_{n-1} \vert^2 + R(z, {}\overline{z})
\]
with the error function $R(z, {}\overline{z}) \to 0$ faster than atleast one of the monomials of weight $1$. The domain $D_1$ in the hypothesis of theorem \ref{thm7} has its defining function about a boundary point $p$  in the above form; we shall say that $p$ is {\it strictly weakly spherical} if the integer $m >1$. \\

\noindent Getting back from this digression about weakly spherical Levi corank one hypersurfaces, to the afore-mentioned reduction procedure for arbitrary Levi corank one hypersurfaces, we recall that it can be split up into five simpler steps -- for $\zeta\in U$, the map $\Phi^\zeta = \phi_5 \circ \phi_4 \circ \phi_3 \circ \phi_2 \circ \phi_1 $ where each $\phi_j$ is described below.\medskip \\
The first step is to normalize the linear part of the Taylor series as in (\ref{nrmlfrm0}). Recall that this was done via the affine map $\phi_1$ given by
\begin{align*}
\phi_1(z_1,\ldots,z_n) &=  \Big( z_1 - \zeta_1, \ldots, z_{n-1}-\zeta_{n-1},\big( z_n - \zeta_n - \sum\limits_{j=1}^{n-1} b_j^\zeta(z_j - \zeta_j) \big) (b_n^\zeta)^{-1} \Big)  \\
& = \Big( z_1 - \zeta_1, \ldots, z_{n-1}-\zeta_{n-1}, \langle\nu(\zeta),z-\zeta \rangle  \Big)  
 \end{align*}
where the coefficients $b_n^\zeta = \big(\partial r/\partial z_n(\zeta) \big)^{-1}$ and  $b_j^\zeta$ are clearly smooth functions of $\zeta$ on $U$. Therefore, $\phi_1$ translates $\zeta$ to the origin and 
\[
r(\phi_1^{-1}(z))=r(\zeta)+ 2 \Re z_n+ \text{terms of higher order}.
\]
where the constant term disappears when $\zeta\in \Sigma$.

\medskip

\noindent Now, since the Levi form restricted to the subspace 
\[ 
L_*=\textrm{span}_{\mathbb{C}^n} \langle L_2,\ldots ,L_{n-1} \rangle
\] 
of $T^{(1,0)}_{\zeta}(\partial \Omega)$ is positive definite, we may diagonalize it via a unitary transform $\phi_2$ and a dilation $\phi_3$ will then ensure that the Hermitian -- quadratic part involving only $z_2,z_3,\ldots,z_{n-2}$ in the Taylor expansion of $r$ is $\vert z_2 \vert^2 + \vert z_3 \vert^2 + \ldots + \vert z_{n-2}\vert^2$. The entries of the matrix that represents the composite of the last two linear transformations are smooth functions of $\zeta$ and in the new coordinates still denoted by $z_1,\ldots z_n$, the defining function is in the form
\begin{multline}\label{taylexp}
r(z)=r(\zeta) + 2 \Re z_n +  \sum\limits_{\alpha=2}^{n-1}\sum\limits_{j=1}^{m} 2\Re\big(( a_j^\alpha z_1^j + b_j^\alpha \bar{z}_1^j )z_\alpha \big)+ 2 \Re\sum\limits_{\alpha=2}^{n-1}c_\alpha z_\alpha^2 \\ 
+ \sum\limits_{2\leq j+k \leq 2m}a_{j,k}z_1^j \bar{z}_1^k + \sum\limits_{\alpha=2}^{n-1}\vert z_\alpha \vert^2 
+  \sum\limits_{\alpha=2}^{n-1}\sum_{\substack{j+k \leq m \\j,k>0}} 2 \Re\big(b_{j,k}^\alpha z_1^j \bar{z}_1^k z_\alpha \big)\\
+ O(\vert z_n \vert \vert z \vert+ \vert z_* \vert^2 \vert z \vert + \vert z_* \vert \vert z_1 \vert^{m+1}
+ \vert z_1 \vert^{2m+1} )
\end{multline}
This still does not reduce the quadratic component of the Taylor series as far as we can; more can be done: the pluriharmonic terms of weights upto $1$ here i.e., $z_\alpha ^2$ as also $z_1^k, \bar{z}_1^k, $, $z_1^k z_\alpha, \bar{z}_1^k \bar{z}_\alpha$ can all be removed by absorbing them into the normal variable $z_n$
by the following standard change of coordinates $\phi_4$ given by
\begin{equation*}
\begin{split}
z_j&=t_j \; \; (1\leq j\leq n-1), \\
z_n&=t_n- \hat{Q}_1(t_1,\ldots,t_{n-1})
\end{split}
\end{equation*}
where 
\[
\hat{Q}_1(t_1, \ldots , t_{n-1}) =\sum\limits_{k=2}^{2m}a_{k0} t_1^k - \sum\limits_{\alpha=2}^{n-1} \sum\limits_{k=1}^{m} a_k^{\alpha} t_\alpha t_1^k-\sum\limits_{\alpha=2}^{n-1} c_{\alpha}t_\alpha^2
\]
with coefficients that are smooth functions of $\zeta$.

\medskip

\noindent The final step removes all other harmonic monomials of weight upto one remaining in (\ref{taylexp}), rewritten in the $t$-coordinates, 
which are of the form $\bar{t}_1^j t_\alpha$ by applying the transform $\phi_5$ given by
\begin{align*}
t_1&=w_1,\ t_n=w_n,\\
t_\alpha &=w_\alpha - Q_2^\alpha(w_1) \; \; (2\leq\alpha \leq n-1)
\end{align*}
where
\[
\shoveleft Q_2^\alpha(w_1)=\sum\limits_{k=1}^{m} \ov{b_k^\alpha} w_1^k 
\]
with coefficients smooth in $\zeta$, as before (since all these coefficients are simply the derivatives of some order of the smooth defining function $r$ evaluated at $\zeta$). 

\medskip

\noindent We then have that the composite $\Phi^\zeta$ of these various simplifying maps is as given in (\ref{E45}) and the normal form for the Taylor expansion as given in (\ref{nrmlfrm}). Further, $\Phi^\zeta$ belongs to the group $\mathcal{E}_L$ for each fixed $\zeta \in U$. In particular, the collection $\{ Q_\zeta, \Phi^\zeta \}$, where $\zeta$ varies over points in the tubular neighbourhood $U$ of $\Sigma$ and $Q_\zeta$ is the biholomorphically distorted polydisc $Q(\zeta, \epsilon(\zeta))$ -- which is also the ball of radius $\epsilon(\zeta)$ about $\zeta$ in the pseudo-distance $d$ defined at (\ref{pseudistdefn}) -- forms an atlas of special charts, giving $U$ the structure of a $\mathcal{E}_L$-manifold i.e., the associated transition maps lie in the complex Lie group $\mathcal{E}$, the group of weight preserving, elementary-like, weight one polynomial automorphisms described above.

\medskip

\noindent We need to compute the inverse $\Psi_\zeta = (\Phi^\zeta)^{-1}$ in the last subsection of Section \ref{pfofthm1}. Let 
\begin{equation} \label{wdefn}
(w_1, \ldots, w_n) = \Phi^{\z}(z) = \Big(z_1 - \z_1, G_{\z}(\ti z - \ti \z) - Q_2(z_1 - \z_1), \langle \nu(\zeta), z- \zeta \rangle - Q_1({}'z - {}' \zeta) \Big)
\end{equation}
where $Q_2: \mathbb{C} \to \mathbb{C}^{n-2}$ is a polynomial map whose components are the polynomials $Q_2^\alpha$ as above. Now, we find out the components of $\Psi_\zeta$, the first component of which is  
\begin{equation}\label{1stcomp}
z_1=w_1 + \zeta_1
\end{equation}
Next, $\tilde{w} = G_\zeta(\tilde{z} - \tilde{\zeta}) - Q_2(z_1 - \zeta_1)$ i.e., $G_\zeta (\tilde{z} - \tilde{\zeta}) = \tilde{w} - Q_2(w_1)$ so that
\begin{equation} \label{midcomp}
\tilde{z} = H_\zeta\big( \tilde{w} + Q_2(w_1) \big) + \tilde{\zeta}
\end{equation}
and finally
\begin{align*}
w_n &= \langle \nu(\zeta), z- \zeta \rangle - Q_1({}'z - {}' \zeta) \\
&= \partial r/ \partial z_n(\zeta) (z_n - \zeta_n) + \sum\limits_{j=1}^{n-1} \partial r/\partial z_j(\zeta) (z_j - \zeta_j) - Q_1({}'z - {}' \zeta)
\end{align*}
Now note that $\tilde{z} - \tilde{\zeta} = H_\zeta \big(\tilde{w} + Q_2(w_1) \big)$. So $Q_1(z_1 -\zeta_1, \tilde{z} - \tilde{\zeta}) = Q_1 \big( w_1, H_\zeta(\tilde{w} + Q_2(w_1)) \big)$ and subsequently,
\begin{align*}
w_n &= \partial r/\partial z_n(\zeta) \Big( z_n - \zeta_n -\sum\limits_{\alpha=2}^{n-2} b_j^\zeta (z_j - \zeta_j) - b_1^\zeta(z_1 - \zeta_1) \Big) - Q_1 \big( w_1, H_\zeta(\tilde{w} + Q_2(w_1)) \big) \\
&= (b_n^\zeta)^{-1} \Big( z_n - \zeta_n - \langle \tilde{b}^{\zeta}, \tilde{z} - \tilde{\zeta} \rangle - b_1^\zeta w_1 \Big) - Q_1 \big( w_1, H_\zeta(\tilde{w} + Q_2(w_1)) \big) \\
\end{align*}
and subsequently,
\begin{align*}
b_n^\zeta w_n &= \Big( z_n - \zeta_n - \langle \tilde{b}^{\zeta}, \tilde{z} - \tilde{\zeta} \rangle - b_1^\zeta w_1 \Big) - b_n^\zeta Q_1 \big(w_1, H_\zeta(\tilde{w} + Q_2(w_1)) \big) 
\end{align*}
giving finally that the last component of $z=\Psi_\zeta(w)$ is of the form
\begin{equation} \label{nthcomp}
z_n = b_n^\zeta w_n + \zeta_n + \langle \tilde{b}^{\zeta}, \tilde{z} - \tilde{\zeta} \rangle + b_1^\zeta w_1 + b_n^\zeta Q_1 \big(w_1, H_\zeta(\tilde{w} + Q_2(w_1)) \big).
\end{equation}
which we shall also write more shortly as 
\[
z_n = b_n^\zeta w_n + b_n^\zeta \tilde{Q}_1({}'w) + \zeta_n
\]
where for some slight convenience in the section where it is used, we take $\tilde{Q}_1$ to be of the form
\[
\tilde{Q}_1({}'w) = (b_n^\zeta)^{-1} \Big(\langle \tilde{b}^\zeta, H_\zeta \big( \tilde{w} + Q_2(w_1)\big)\rangle + b_1^\zeta w_1 \Big) + Q_1 \Big(w_1, H_\zeta  \big( \tilde{w} + Q_2(w_1)\big) \Big).
\]
with $Q_1$ and $Q_2$ are the same very polynomials occurring in the expression for $\Phi^\zeta$ as in (\ref{wdefn}). Now, altogether, equations (\ref{1stcomp}), (\ref{midcomp}) and (\ref{nthcomp}) give the expressions for the various components that constitute the mapping $\Psi_\zeta(w)$. 
A straightforward computation shows that the derivative of the map $\Phi^\zeta(z)$ in standard co-ordinates, is represented by the matrix 
\begin{eqnarray}
D \Phi^\zeta(z) = 
\begin{pmatrix}
1                                                           & 0        & \cdots    &  0     & 0 \\
-\frac{\partial Q_2^2}{\partial z_1}(z_1 - \zeta_1)         &{}        & {}        & {}     & 0 \\
\vdots                                                      &{}        & G_\zeta   & {}     & \vdots \\
-\frac{\partial Q_2^{n-1}}{\partial z_1}(z_1 - \zeta_1)     &{}        & {}        & {}     & 0 \\
\frac{\partial r}{\partial z_1}(\zeta) - \frac{\partial Q_1}{\partial z_1}({}'z - {}' \zeta)        &\cdots  &{}       & \frac{\partial r}{\partial z_{n-1}}(\zeta) - \frac{\partial Q_1}{\partial z_1}({}'z - {}' \zeta)  & \frac{\partial r}{\partial z_n}(\zeta)
\end{pmatrix}
\end{eqnarray}
so that in particular, the mapping occurring in the definition of the $M$-metric namely, 
\[
D \Phi^z(z)(X) = \Big( X_1, G_z(\tilde{X}), \sum \limits_{j=1}^{n} \partial r/\partial z_j (z) X_j \Big) = \Big( X_1, G_z(\tilde{X}), \langle \nu(z),X \rangle  \Big)
\]
which ofcourse is linear for each fixed $z$ but also is weighted homogeneous in the variables $z_1$ through $z_{n-1}$, with respect to the weights that we have assigned to the variables $z_1, \ldots, z_n$, as well. The above expression is needed in section $\ref{FriL}$. We notice in passing that $D \Phi^\zeta$ for each fixed $\zeta \in U$ belongs to the linear Lie group $\mathcal{E}L_n = \mathcal{E}_L \cap GL_n(\mathbb{C})$ and gives the tangent bundle of the $\mathcal{E}_L$-manifold mentioned above, the structure of a fibre bundle with structure group $\mathcal{E}L_n$.\\

\noindent Next, we record the derivative of the inverse map $\Psi_\zeta$ needed in the last sub-section of section \ref{pfofthm1}. 
\begin{eqnarray}
\; \; \;\; D\Psi_\zeta(w) = 
\begin{pmatrix}
1                                                                                & 0      & \cdots  &  0     & 0 \\
\langle (H_\zeta)_{R_2}, \frac{\partial Q_2}{\partial w_1}(w) \rangle          &{}      & {}      & {}     & 0 \\
\vdots                                                                           &{}      & H_\zeta & {}     & \vdots \\
\langle (H_\zeta)_{R_{n-1}}, \frac{\partial Q_2}{\partial w_1}(w) \rangle  &{}      & {}      & {}     & 0 \\
b_n^\zeta \partial \tilde{Q}_1/\partial w_1({}'w)                                  &b_n^\zeta \partial \tilde{Q}_1/\partial w_2({}'w)  &\ldots &b_n^\zeta\partial \tilde{Q}_1/\partial w_{n-1}({}'w) & b_n^\zeta
\end{pmatrix}
\end{eqnarray}
where we index the rows and columns of the square matrix $H_\zeta=G_\zeta^{-1}$ of order $n-2$ by the integers $2, \ldots,n-1$ and denote its rows herein by $(H_\zeta)_{R_\alpha}$.



\begin{thebibliography}{BFKKMP}

\bibitem{Ab} M. Abate:
\textit{Boundary behaviour of invariant distances and complex geodesics},
Atti Accad. Naz. Lincei Rend. Cl. Sci. Fis. Mat. Natur. \textbf{80}(8) (1986), 100--106.

\bibitem{AT} M. Abate, R. Tauraso: 
\textit{The Lindel\"{o}f principle and angular derivatives in convex domains of finite type},
 J. Aust. Math. Soc. \textbf{73} (2002), no. 2, 221--250. 

\bibitem{Al} G. Aladro:
\textit{Localization of the Kobayashi distance},
J. Math. Anal. Appl. \textbf{181} (1994), 200--204.

\bibitem{Ber} F. Berteloot:
\textit{Principle de Bloch et estimations de la metrique de Kobayashi des domains in $\mbb C^2$}, J. Geom. Anal. \textbf{1} (2003), 29--37.

\bibitem{BV} G. P. Balakumar, K. Verma:
\textit{Some regularity theorems for CR mappings},
To appear in Math. Z., DOI: 10.1007/s00209-012-1060-6.

\bibitem{BB} Z. Balogh, M. Bonk:
\textit{Gromov hyperbolicity and the Kobayashi metric on strictly pseudoconvex domains},
Comment. Math. Helv. \textbf{75} (2000), 504--533. 

\bibitem{BaBe} E. Barletta, E. Bedford:
\textit{Existence of proper mappings from domains in $\mathbb{C}^2$.} Indiana Univ. Math. J. 39 (1990), no. 2, 315–-338.

\bibitem{BP} E. Bedford, S. I. Pinchuk:
\textit{Domains in $ \mathbb{C}^{n+1} $ with noncompact automorphism group},
J. Geom. Anal. \textbf{1} (1991), no. 3, 165--191.

\bibitem{Be} S. Bell:
\textit{Smooth bounded strictly and weakly pseudoconvex domains cannot be biholomorphic},
Bull. Amer. Math. Soc. (N.S.) \textbf{4} (1981), no. 1, 119--�120.

\bibitem{BFKKMP} B. E. Blank, D. Fan, D. Klein, S. Krantz, D. Ma, M. Y. Peng:
\textit{The Kobayashi metric of a complex ellipsoid in $\mbb C^2$},
Experiment. Math. \textbf{1} (1992), no. 1, 47--�55.

\bibitem{Ca} D. Catlin:
\textit{Estimates of invariant metrics on pseudoconvex domains of dimension two},
math. Z. \textbf{200} (1989), 429--466.

\bibitem{Cho} S. Cho: 
\textit{A lower bound on the Kobayashi metric near a point of finite type in $\mbb C^n$},
J. Geom. Anal. \textbf{2} (1992), no. 4, 317--325.

\bibitem{Cho1} S. Cho:
\textit{Estimates of invariant metrics on some pseudoconvex domains in $\mbb C^n$},
J. Korean Math. Soc. \textbf{32} (1995),  no. 5, 661--678.

\bibitem{Cho2} S. Cho:
\textit{Boundary behaviour of the Bergman kernel function on some pseudoconvex domains in $\mbb C^n$},
Trans. Amer. Math. Soc. \textbf{345} (1994), no. 2, 803--817.

\bibitem{CP} B. Coupet, S. Pinchuk:
\textit{Holomorphic equivalence problem for weighted homogeneous rigid domains in $ \mathbb{C}^{n+1} $},
Complex Analysis in Modern Mathematics FAZIS, Moscow (2001), 57--70.

\bibitem{CPS} B. Coupet, S. Pinchuk, A. Sukhov:
\textit{On boundary rigidity and regularity of holomorphic mappings},
Internat. J. Math. \textbf{7} (1996), 617--643.

\bibitem{D1} K. Diederich:
\textit{Das Randverhalten der Bergmanschen Kernfunktion und Metrik in streng pseudokonvexen Gebieten},
Math. Ann. \textbf{187} (1970), 9--36.

\bibitem{DF} K. Diederich, J. E. Fornaess:
\textit{Proper holomorphic images of strictly pseudoconvex domains},
Math. Ann. \textbf{259} (1982), no. 2, 279--286.

\bibitem{DF1}K. Diederich, J. E. Fornaess: 
\textit{Pseudoconvex domains: bounded strictly plurisubharmonic exhaustion functions},
Invent. Math. \textbf{39} (1977), no. 2, 129--141.

\bibitem{Fr} B. Fridman:
\textit{Biholomorphic invariants of a hyperbolic manifold and some applications},
Trans. Amer. Math. Soc. \textbf{276} (1983), 685--698.

\bibitem{FrMi} S. Friedland, J. Milnor:
\textit{Dynamical properties of plane polynomial automorphisms}, Ergodic Theory Dynam. Systems  \textbf{9}(1) (1989), 67-–99. 

\bibitem{FS} J. E. Fornaess, N. Sibony:
\textit{Consuction of P.S.H. functions on weakly pseudoconvex domains},
Duke Math. J. \textbf{58} (1989), 633--655.

\bibitem{FR} F. Forstneric, J. P. Rosay:
\textit{Localization of the Kobayashi metric and the boundary continuity of proper holomorphic mappings},
Math. Ann. \textbf{279} (1987), 239--252.

\bibitem{GS} H. Gaussier, H. Seshadri:
\textit{Totally geodesic discs in strongly convex domains},
To appear in Math. Z.

\bibitem{G} I. Graham:
\textit{Boundary behaviour of the Carath\'{e}odory and Kobayashi metrics on strongly pseudoconvex domains in $\mbb C^n$ with smooth boundary},
Trans. Amer. Math. Soc. \textbf{207} (1975), 219--240.

\bibitem{GK} R. Greene, S. Krantz:
\textit{Deformation of complex structures, estimates for the $\overline \partial$ equation and stability of the Bergman kernel}, 
Adv. in Math. \textbf{43} (1982), no. 1, 1--86.

\bibitem{H} L. H\"ormander: 
\textit{$L^2$-estimates and existence theorems for the $\overline{\partial}$ operator}, Acta Math. \textbf{113} (1965) 89--152.

\bibitem{Her1} G. Herbort:
\textit{On the invariant differential metrics near pseudoconvex boundary points where the Levi form has corank one}, 
Nagoya Math. J. \textbf{130} (1993), 25--54.
 
\bibitem{Her2} G. Herbort:
\textit{Invariant metrics and peak functions on pseudoconvex domains of homogeneous finite diagonal type},
Math. Z. \textbf{209} (1992), no. 2, 223--23.

\bibitem{Her} G. Herbort:
\textit{Estimation on invariant distances on pseudoconvex domains of finite type in dimension two},
Math. Z. \textbf{251} (2005), no. 3, 673--703.

\bibitem{JP} M. Jarnicki, P. Pflug:
\textit{Invariant distances and metrics in Complex Analysis},
de Gruyter Expoitions in Mathematics, 9. Walter de Gruyter  Co., Berlin (1993).

\bibitem{JPZ} M. Jarnicki, P. Pflug, R. Zeinstra:
\textit{Geodesics for convex complex ellipsoids},
Ann. Scuola Norm. Sup. Pisa Cl. Sci.  \textbf{20} (1993), no. 4, 535--543. 

\bibitem{KK} K.T. Kim, S. Krantz:
\textit{A Kobayashi metric version of Bun Wong's theorem}, 
Complex Var. Elliptic Equ. \textbf{54} (2009), no. 3-4, 355--369.

\bibitem{K} M. Klimek:
\textit{Extremal plurisubharmonic functions and invariant pseudodistances.} Bull. Soc. Math. France \textbf{113} (1985), no. 2, 231-240.

\bibitem{KM} K.T. Kim and D. Ma:
\textit{Characterisation of the Hilbert ball by its automorphisms}, 
J. Korean Math. Soc. \textbf{40} (2003), 503--516.

\bibitem{Kr} S. Krantz:
\textit{Invariant metrics and the boundary behaviour of holomorphic functions on domains in $\mbb C^n$},
J. Geom. Anal. \textbf{1} (1991), 71--97.


\bibitem{Ma} D. Ma:
\textit{Sharp estimates of the Kobayashi metric near strongly pseudoconvex points},
The Madison Symposium on Complex Analysis (Madison, WI, 1991), 329--338,\textit{Contemp. Math.}, \textbf{137} Amer. Math. Soc., Providence, RI, 1992. 

\bibitem{Ma1} D. Ma:
\textit{Smoothness of Kobayashi metric of ellipsoids},
Complex Variables Theory Appl. \textbf{26} (1995), no. 4, 291--298.

\bibitem{M} P. Mahajan:
\textit{On isometries of the Kobayashi and Carath\'{e}odory metrics}, 
Ann. Polon. Math. \textbf{104} (2012), no. 2, 121--151.

\bibitem{Mok} N. Mok:
\textit{Extension of germs of holomorphic isometries up to normalizing constants with respect to the Bergman metric.} J. Eur. Math. Soc. (JEMS) \textbf{14} (2012), no. 5, 1617–-1656.

\bibitem{MV} P. Mahajan, K. Verma:
\textit{Some aspects of the Kobayashi and Carath\'{e}odory metrics on pseudoconvex domains}, 
J. Geom. Anal. \textbf{22} (2012), no. 2, 491--560.

\bibitem{N1} N. Nikolov:
\textit{Comparison of invariant functions on strongly pseudoconvex domains}, arXiv:1212.2428

\bibitem{N2} N. Nikolov:
\textit{Estimates of invariant distances on ``convex" domains}, arXiv:1210.7223

\bibitem{Pi} S. I. Pinchuk:
\textit{Holomorphic inequivalence of certain classes of domains in $\mbb C^n$}, 
Math. USSR Sb. (N.S), \textbf{39} (1980), 61--86.

\bibitem{R} R. M. Range:
\textit{A remark on bounded strictly plurisubharmonic exhaustion functions.} Proc. Amer. Math. Soc. 81 (1981), no. 2, 220--222.

\bibitem{Rei} H. J. Reiffen: 
\textit{Die differentialgeometrischen Eigenschaften der invarianten Distanzfunktion von Carathéodory.} (German) Schr. Math. Inst. Univ. Münster. \textbf{26} 1963.

\bibitem{Ro} H. L. Royden: 
\textit{Remarks on the Kobayashi metric},
Several Complex Variables II (College Park, MD, 1970), Lecture Notes in Math. 185, Springer, Berlin (1971), 125--137.

\bibitem{S2} N. Sibony:
\textit{Une classe de domaines pseudoconvexes. (French) [A class of pseudoconvex domains]}, Duke Math. J. \textbf{55} (1987), no. 2, 299–-319.

\bibitem{S} N. Sibony:
\textit{Some aspects of weakly pseudoconvex domains.} Several complex variables and complex geometry, Part 1 (Santa Cruz, CA, 1989), 199-?231, Proc. Sympos. Pure Math., \textbf{52}, Part 1, Amer. Math. Soc., Providence, RI, 1991.


\bibitem {SV1} H. Seshadri, K. Verma:
\textit{On isometries of the Carath\'{e}odory and Kobayashi metrics on
strongly pseudoconvex domains}, 
Ann. Scuola Norm. Sup. Pisa Cl. Sci.(5), \textbf{V} (2006), 393--417.

\bibitem {SV2} H. Seshadri, K. Verma:
\textit{On the compactness of isometry groups in complex analysis}, 
Complex Var. Elliptic Equ., \textbf{54} (2009), 387--399.

\bibitem{Su} A. Sukhov:
\textit{On boundary behaviour of holomorphic mappings}, 
(Russian), Mat. Sb. \textbf{185} (1994), 131--142; (English transl.) Russian Acad. Sci. Sb. Math. \textbf{83} (1995), 471--483.

\bibitem{TT} Do D. Thai, Ninh V. Thu:
\textit{Characterization of domains in $ \mathbb{C}^n $ by their noncompact automorphism groups},
Nagoya Math. J. \textbf{196} (2009), 135--160.


\end{thebibliography}
\end{document}